\documentclass[12pt,a4paper]{amsart}
\usepackage[czech,english]{babel}
\usepackage{amssymb,eucal}
\usepackage[all,cmtip]{xy}
\usepackage{hyperref}

\calclayout

\newcommand{\+}{\nobreakdash-}
\renewcommand{\:}{\colon}  

\newcommand{\rarrow}{\longrightarrow}

\newcommand{\bu}{{\text{\smaller\smaller$\scriptstyle\bullet$}}}
\newcommand{\lrarrow}{\mskip.5\thinmuskip\relbar\joinrel\relbar\joinrel
 \rightarrow\mskip.5\thinmuskip\relax} 

\DeclareMathOperator{\Hom}{Hom}
\DeclareMathOperator{\Ext}{Ext}
\DeclareMathOperator{\Tot}{Tot}
\DeclareMathOperator{\coker}{coker}
\DeclareMathOperator{\cone}{cone}
\DeclareMathOperator{\cof}{cof}

\newcommand{\id}{\mathrm{id}}

\newcommand{\sA}{\mathsf A}
\newcommand{\sB}{\mathsf B}
\newcommand{\sC}{\mathsf C}
\newcommand{\sD}{\mathsf D}
\newcommand{\sE}{\mathsf E}
\newcommand{\sF}{\mathsf F}
\newcommand{\sG}{\mathsf G}
\newcommand{\sH}{\mathsf H}
\newcommand{\sK}{\mathsf K}
\newcommand{\sL}{\mathsf L}
\newcommand{\sP}{\mathsf P}
\newcommand{\sR}{\mathsf R}
\newcommand{\sS}{\mathsf S}
\newcommand{\sT}{\mathsf T}
\newcommand{\sW}{\mathsf W}
\newcommand{\sX}{\mathsf X}

\newcommand{\cF}{\mathcal F}
\newcommand{\cL}{\mathcal L}
\newcommand{\cR}{\mathcal R}
\newcommand{\cS}{\mathcal S}
\newcommand{\cT}{\mathcal T}
\newcommand{\cW}{\mathcal W}

\newcommand{\Sets}{{\mathsf{Sets}}}
\newcommand{\Fil}{{\mathsf{Fil}}}
\newcommand{\AdF}{{\mathsf{AdF}}}
\newcommand{\Add}{{\mathsf{Add}}}

\newcommand{\upit}{\fontshape{ui}\selectfont}
\newcommand{\Mono}{{\operatorname{-\mathcal M\mbox{\upit ono}}}}
\newcommand{\Epi}{{\operatorname{-\mathcal E\mbox{\upit pi}}}}

\newcommand{\Cof}{\mathcal C\mbox{\upit of}}
\newcommand{\Cell}{\mathcal C\mbox{\upit ell}}

\newcommand{\inj}{{\mathsf{inj}}}
\newcommand{\proj}{{\mathsf{proj}}}
\newcommand{\hin}{{\mathsf{hin}}}
\newcommand{\hpr}{{\mathsf{hpr}}}
\newcommand{\ac}{{\mathsf{ac}}}
\newcommand{\co}{{\mathsf{co}}}
\newcommand{\ctr}{{\mathsf{ctr}}}
\newcommand{\sgr}{{\mathsf{gr}}}
\newcommand{\sop}{{\mathsf{op}}}

\newcommand{\boZ}{\mathbb Z}

\theoremstyle{plain}
\newtheorem{thm}{Theorem}[section]

\newtheorem{lem}[thm]{Lemma}
\newtheorem{prop}[thm]{Proposition}
\newtheorem{cor}[thm]{Corollary}
\theoremstyle{definition}

\newtheorem{exs}[thm]{Examples}
\newtheorem{rem}[thm]{Remark}
\newtheorem{qst}[thm]{Question}

\newcommand{\Section}[1]{\bigskip\section{#1}\medskip}

\begin{document}

\title{Derived, coderived, and contraderived categories \\
of locally presentable abelian categories}

\author[L.~Positselski]{Leonid Positselski}

\address[Leonid Positselski]{%
Institute of Mathematics of the Czech Academy of Sciences \\
\v Zitn\'a~25, 115~67 Prague~1 \\ Czech Republic; and
\newline\indent Laboratory of Algebra and Number Theory \\
Institute for Information Transmission Problems \\
Moscow 127051 \\ Russia}

\email{positselski@math.cas.cz}

\author[J.~\v S\v tov\'\i\v cek]{Jan \v S\v tov\'\i\v cek}

\address[Jan {\v S}{\v{t}}ov{\'{\i}}{\v{c}}ek]{%
Charles University, Faculty of Mathematics and Physics,
Department of Algebra, Sokolovsk\'a 83, 186 75 Praha,
Czech Republic}

\email{stovicek@karlin.mff.cuni.cz}

\begin{abstract}
 For a locally presentable abelian category $\sB$ with a projective
generator, we construct the projective derived and contraderived model
structures on the category of complexes, proving in particular
the existence of enough homotopy projective complexes of projective
objects.
 We also show that the derived category $\sD(\sB)$ is generated, as
a triangulated category with coproducts, by the projective generator
of~$\sB$.
 For a Grothendieck abelian category $\sA$, we construct the injective
derived and coderived model structures on complexes.
 Assuming Vop\v enka's principle, we prove that the derived category
$\sD(\sA)$ is generated, as a triangulated category with products, by
the injective cogenerator of~$\sA$.
 More generally, we define the notion of an exact category with
an object size function and prove that the derived category of any such
exact category with exact $\kappa$\+directed colimits of chains of
admissible monomorphisms has Hom sets.
 In particular, the derived category of any locally presentable abelian
category has Hom sets. 
\end{abstract}

\maketitle

\tableofcontents

\section*{Introduction}
\medskip

 The definition of the unbounded derived category of an abelian category
goes back to the work of Grothendieck and Verdier
in 1960's \cite{Ver0,Ver1}, but efficient techniques of working with
such derived categories started to be developed only in the 1988 paper
of Spaltenstein~\cite{Spa} (who attributes the idea to J.~Bernstein).
 Subsequently they were extended to the derived categories of
DG\+modules by Keller~\cite{Kel} (see also~\cite{BL}).

 The problem is that the derived category $\sD^+(\sA)$ of bounded below
complexes in an abelian category $\sA$ with enough injective objects is
equivalent to the homotopy category of bounded below complexes of
injective objects $\sK^+(\sA_\inj)$, and similarly, the derived category
$\sD^-(\sB)$ of bounded above complexes in an abelian category $\sB$
with enough projective objects is equivalent to the homotopy category
of bounded above complexes of projective objects $\sK^-(\sB_\proj)$.
 Such triangulated equivalences are used in the constructions of
derived functors acting between bounded above or below
derived categories.
 But these equivalences \emph{fail} for unbounded derived categories,
generally speaking.
 In fact, an unbounded complex of projective (or injective) modules
over a ring of infinite homological dimension can be acyclic without
being contractible.

 So one needs to use what Spaltenstein called ``K\+injective'' or
``K\+projective'' resolutions for unbounded complexes.
 Nowadays some people call them ``homotopy injective'' or ``homotopy
projective'' complexes; we will use this terminology.
 The homotopy projective complexes of projective objects, which are
more suitable for some purposes than just arbitrary homotopy projective
complexes, are sometimes called ``semi-projective'' or
``DG\+projective'', and similarly for the homotopy injective complexes
of injective objects (see Remark~\ref{homotopy-adjusted-terminology}
for a terminological discussion).

 It took another decade or two to realize the existence and importance
of an alternative point of view on unbounded complexes, in which
arbitrary complexes of injective or projective objects are used as
resolutions.
 The equivalence relation on the unbounded complexes then needs to be
modified accordingly; so the conventional quasi-isomorphism is replaced
by a finer equivalence relation, making the resulting ``exotic derived
category'' larger than the conventional one.

 The homotopy category of unbounded complexes of projective modules was
first considered by J\o rgensen~\cite{Jor}, and the homotopy category
of unbounded complexes of injective objects in a locally Noetherian
Grothendieck category was first studied by Krause~\cite{Kra}.
 Constructions realizing such triangulated categories, similarly to
the conventional derived category, as \emph{quotient categories} of
the homotopy category of the ambient abelian category, were
emphasized in the memoir~\cite{Pkoszul}, where the terminology of
the \emph{derived categories of the first} and \emph{second kind} was
suggested.
 In fact, there are several ``derived categories of the second kind'':
at the bare minimum, one has to distinguish between the \emph{coderived}
and the \emph{contraderived} category.

 In subsequent publications, two approaches to derived categories of
the second kind emerged.
 On the one hand, there is an elementary construction of coderived and
contraderived (as well as the so-called \emph{absolute derived}
categories) as certain quotient categories of the homotopy categories.
 It goes back to the book~\cite{Psemi} and the memoir~\cite{Pkoszul},
and was developed further in the papers~\cite{PP2,EP,PS2,Pps} (see
Remark~\ref{coderived-history} for a historical and terminological
discussion).
 This approach is applicable to a wide class of abelian categories,
as well as to Quillen exact categories, exact DG\+categories (such as
categories of curved DG\+modules), etc.

 On the other hand, there is an approach based on the set-theoretic
techniques of contemporary model category theory (essentially,
the small object argument).
 It goes back to the papers~\cite{Jor,Kra}, and was developed by
Becker~\cite{Bec} in the context of curved DG\+modules over curved
DG\+rings.
 Other relevant papers in this direction include~\cite{BGH},
\cite{Neem1,Neem2}, and~\cite{Sto}.
 For categories suited for applicability of set-theoretical methods,
such as the categories of modules over associative rings, this approach
leads to stronger and more general results that the elementary one.
 It is still an open question whether the two approaches agree, e.~g.,
for module categories.

 In this paper, we follow the approach of Becker, generalizing it from
module categories to some locally presentable abelian categories.
 In fact, both derived categories (or abelian model structures) of
the first and the second kind were considered in the paper~\cite{Bec},
and we also work out both of these in this paper.
 On the other hand, we do not consider (curved or uncurved)
DG\+structures, restricting ourselves to the categories of complexes
in abelian categories.

 One of the aims of this paper is to emphasize the duality-analogy
between two natural classes of abelian categories.
 On the one hand, there are the \emph{Grothendieck abelian categories}.
 Grothendieck~\cite{GrToh} proved that they have enough injective
objects.
 Now it is known that any Grothendieck abelian category has enough
homotopy injective complexes~\cite{AJS}, and even enough
homotopy injective complexes of injective objects~\cite{Ser,Gil}.
 Following in the footsteps of Becker~\cite{Bec}, we construct
the injective derived and the coderived abelian model structures on
the category $\sC(\sA)$ of complexes in~$\sA$.
 We also show that, assuming Vop\v enka's principle, the derived
category $\sD(\sA)$ of a Grothendieck abelian category $\sA$ is
generated, as a triangulated category with products, by any injective
cogenerator of~$\sA$.
 
 On the other hand, there is a much less familiar, but no less natural
class of \emph{locally presentable abelian categories $\sB$ with enough
projective objects}~\cite{PR,Pper,PS1}.
 Various contramodule categories~\cite[Section~III.5]{EM},
\cite{Psemi,Pweak,Prev,BP,Prem} are representatives of this class. 
 We construct the projective derived and the contraderived abelian
model structures on the category $\sC(\sB)$ of complexes in~$\sB$.
 In particular, it follows that there are enough homotopy projective
complexes of projective objects in~$\sB$.
 We also show that the derived category $\sD(\sB)$ is generated, as
a triangulated category with coproducts, by any projective generator
of~$\sB$.

 Furthermore, the contraderived category $\sD^\ctr(\sB)$ (in the sense
of Becker) is equivalent to the homotopy category $\sK(\sB_\proj)$ of
complexes of projective objects in~$\sB$.
 For any locally presentable additive category $\sE$ and a fixed
object $M\in\sE$, there exists a (unique) locally presentable abelian
category $\sB$ with enough projective objects such that the full
subcategory $\sB_\proj\subset\sB$ of projective objects in $\sB$ is
equivalent to the full subcategory $\Add(M)\subset\sE$ of direct
summands of coproducts of copies of $M$ in~$\sE$
\,\cite{PS1,PS2,Pper}.
 So the homotopy category $\sK(\Add(M))$ of complexes in $\sE$ with
the terms in $\Add(M)$ can be interpreted as the contraderived
category $\sD^\ctr(\sB)$.
 It follows from our results that the homotopy category $\sK(\Add(M))$
is a well-generated triangulated category in the sense of
Neeman~\cite{Neem-book,Kra0}.

 More generally, we consider the derived category $\sD(\sE)$ of
an arbitrary locally presentable abelian category~$\sE$.
 Such categories may have neither injective nor projective objects,
and neither infinite direct sum nor infinite product functors in $\sE$
need to be exact.
 So one cannot speak of the coderived or contraderived categories of
$\sE$, and model category methods in the spirit of~\cite{Bec} do not
seem to be applicable.
 Still we prove \emph{something} about the derived category $\sD(\sE)$,
namely, that it has Hom sets (in other words, the derived category
``exists'' in the same universe in which $\sE$ is a category).
 In fact, our most general result in this direction is applicable to
a certain class of exact categories~$\sE$ (in the sense of Quillen).

 In Sections~\ref{cotorsion-pairs-secn}\+-%
\ref{categories-of-complexes-secn} we review and collect the preparatory
material on cotorsion pairs, weak factorization systems, and abelian
model structures, following mostly the papers~\cite{Hov}, \cite{Bec},
and~\cite{PR}.
 In Sections~\ref{loc-pres-secn}\+-\ref{contraderived-secn}, we
consider a locally presentable abelian category $\sB$ with enough
projective objects.
 In Section~\ref{loc-pres-secn}, we construct the projective derived
model category structure on the category of complexes $\sC(\sB)$, and
in Section~\ref{contraderived-secn} we produce the contraderived
model category structure on $\sC(\sB)$.
 In Sections~\ref{grothendieck-secn}\+-\ref{coderived-secn}, we
consider a Grothendieck abelian category~$\sA$.
 In Section~\ref{grothendieck-secn}, we work out the injective derived
model category structure on $\sC(\sA)$, and in
Section~\ref{coderived-secn}, we construct the coderived model
category structure on $\sC(\sA)$.

 In all the four cases, we obtain a hereditary abelian (hence stable)
combinatorial model category.
 Consequently, it follows that the related derived, contraderived, and
coderived categories are well-generated triangulated
categories~\cite{Ros,CR}.
 These results are known for the derived and coderived categories of
Grothendieck abelian categories, and the related (derived and
coderived) injective model structures~\cite{Gil,Kra2,Neem2,Kra3,Gil4}.
 So we include these in our discussion mostly for the sake of
completeness of the exposition, and in order to demonstrate a uniform
approach, making the duality-analogy between the abelian categories
$\sA$ and $\sB$ visible.

 Concerning the derived and contraderived projective abelian model
structures and the related triangulated categories, these have been
studied in the published literature in the special case of module
categories, or slightly more generally, Grothendieck abelian categories
with enough projective objects~\cite{Neem1,BGH,Gil3}.
 One of the aims of this paper is to emphasize that the natural
generality level for projective abelian model structures is much wider
than that, and actually includes all the locally presentable abelian
categories with enough projective objects.
 We refer the reader to the books and papers~\cite{Psemi,Pkoszul,Pweak,
Prev,PR,Pper,PS1,BP,Prem} for examples of various classes of abelian
categories of contramodules, which are all locally presentable with
enough projective objects, but almost never Grothendieck.

 Finally, in the last Section~\ref{hom-sets-secn} we show that
the derived category $\sD(\sE)$ ``exists'' (or in a different
terminology, has Hom sets, rather than classes) for any locally
presentable abelian category~$\sE$.
 Our results in this direction also apply to Quillen exact categories
with a cardinal-valued object size function and exact
$\kappa$\+directed colimits of chains of admissible monomorphisms
for a large enough regular cardinal~$\kappa$.

\subsection*{Acknowledgement}
 We wish to thank an anonymous referee for reading the manuscript
carefully and suggesting a number of relevant references.
 The first-named author is supported by the GA\v CR project 20-13778S
and research plan RVO:~67985840.
 The second-named author is supported by the GA\v CR project 20-13778S.

\Section{Cotorsion Pairs} \label{cotorsion-pairs-secn}

 Let $\sE$ be an abelian category and $\sL$, $\sR\subset\sE$ be two
classes of objects (equivalently, full subcategories) in~$\sE$.
 The class $\sL$ is said to be \emph{generating} if every object of
$\sE$ is a quotient object of an object from~$\sL$.
 Dually, the class $\sR$ is said to be \emph{cogenerating} if every
object of $\sE$ is a subobject of an object from~$\sR$.

 We denote by $\sL^{\perp_1}\subset\sE$ the class of all objects
$X\in\sE$ such that $\Ext^1_\sE(L,X)=0$ for all $L\in\sL$.
 Dually, we let ${}^{\perp_1}\sR\subset\sE$ denote the class of all
objects $Y\in\sE$ such that $\Ext^1_\sE(Y,R)=0$ for all $R\in\sR$.
 A pair of classes of objects $(\sL,\sR)$ in $\sE$ is said to be
a \emph{cotorsion pair} if $\sR=\sL^{\perp_1}$ and
$\sL={}^{\perp_1}\sR$.
 The class $\sL$ is then always closed under direct
summands and all coproducts which exist in $\sE$, and dually $\sR$
is always closed under direct summands and products
(regardless of exactness properties of products and coproducts;
see~\cite[Corollary 8.3]{CoFu} or~\cite[Corollary A.2]{CoSt}).

 For any class of objects $\sS\subset\sE$, the pair of classes
$\sR=\sS^{\perp_1}$ and $\sL={}^{\perp_1}\sR\subset\sE$ is
a cotorsion pair.
 The cotorsion pair $(\sL,\sR)$ is said to be \emph{generated} by
the class~$\sS$.

 Dually, for any class of objects $\sT\subset\sE$, the pair of classes
$\sL={}^{\perp_1}\sT$ and $\sR=\sL^{\perp_1}$ is a cotorsion pair.
 The cotorsion pair $(\sL,\sR)$ is said to be \emph{cogenerated} by
the class~$\sT$.

 A cotorsion pair $(\sL,\sR)$ in $\sE$ is said to \emph{admit special
precover sequences} if for every object $E\in\sE$ there exists
a short exact sequence
\begin{equation} \label{special-precover-sequence}
 0\lrarrow R'\lrarrow L\lrarrow E\lrarrow0
\end{equation}
in $\sE$ with $R'\in\sR$ and $L\in\sL$.
 The cotorsion pair $(\sL,\sR)$ is said to \emph{admit special
preenvelope sequences} if for every object $E\in\sE$ there exists
a short exact sequence
\begin{equation} \label{special-preenvelope-sequence}
 0\lrarrow E\lrarrow R\lrarrow L'\lrarrow0
\end{equation}
in $\sE$ with $R\in\sR$ and $L'\in\sL$.
 The \emph{approximation sequences} is a generic name for the special
precover and special preenvelope sequences.

\begin{lem} \label{salce}
 Let $(\sL,\sR)$ be a cotorsion pair in\/ $\sE$ such that the class\/
$\sL$ is generating and the class $\sR$ is cogenerating in\/~$\sE$.
 Then the pair of classes $(\sL,\sR)$ admits special precover
sequences if and only if it admits special preenvelope sequences.
\end{lem}

\begin{proof}
 This is a categorical generalization of the Salce lemmas~\cite{Sal}.
 The key observation is that, in any cotorsion pair $(\sL,\sR)$, both
the classes $\sL$ and $\sR$ are closed under extensions in~$\sE$.
 Suppose, e.~g., that $(\sL,\sR)$ admits special preenvelope sequences,
and let $E\in\sE$ be an object.
 By assumption, $E$ is a quotient object of an object $M\in\sL$; so
there is a short exact sequence $0\rarrow F\rarrow M\rarrow E\rarrow 0$
in~$\sE$.
 Let $0\rarrow F\rarrow R\rarrow L'\rarrow0$ be a special preenvelope
sequence for the object $F\in\sE$; so $R\in\sR$ and $L'\in\sL$.
 Taking the pushout of the pair of monomorphisms $F\rarrow M$ and
$F\rarrow R$ produces a special precover sequence $0\rarrow R\rarrow L
\rarrow E\rarrow0$, where $L\in\sL$ is an extension of the objects
$M$ and~$L'$.
\end{proof}

 A cotorsion pair $(\sL,\sR)$ is $\sE$ is said to be \emph{complete} if
it admits both special precover and special preenvelope sequences.

 Whenever there are enough projective objects in $\sE$, the class $\sL$
in any cotorsion pair $(\sL,\sR)$ is generating, because it contains all
the projective objects.
 It follows that any cotorsion pair $(\sL,\sR)$ admitting special
preenvelope sequences is complete.
 Dually, whenever there are enough injective objects in $\sE$,
the class $\sR$ in any cotorsion pair $(\sL,\sR)$ is cogenerating,
because it contains all the injective objects.
 It follows that any cotorsion pair $(\sL,\sR)$ admitting special
precover sequences is complete.

 When there are both enough projectives and injectives in $\sE$,
the assumption of Lemma~\ref{salce} holds automatically, so a cotorsion
pair admits special precover sequences if and only if it admits special
preenvelope sequences.
 Such is the situation in the abelian categories of modules over
associative rings.

 For any class of objects $\sF\subset\sE$, we denote by $\sF^\oplus$
the class of all direct summands of objects from~$\sF$.

\begin{lem}
 Let $(\sL,\sR)$ be a pair of classes of objects in\/ $\sE$ such that\/
$\Ext^1_\sE(L,R)=0$ for all $L\in\sL$ and $R\in\sR$.
 Assume that the approximation
 sequences~\textup{(\ref{special-precover-sequence}\+-%
\ref{special-preenvelope-sequence})} with $R$, $R'\in\sR$ and
$L$, $L'\in\sL$ exist for all objects $E\in\sE$.
 Then $(\sL^\oplus,\sR^\oplus)$ is a complete cotorsion pair in\/~$\sE$.
\qed
\end{lem}

\begin{proof}
 This is an analogue of Lemma~\ref{retract-lemma} below.
 It suffices to show that $\sL^{\perp_1}\subset\sR^\oplus$ and
${}^{\perp_1}\sR\subset\sL^\oplus$.
 Indeed, let $E\in\sE$ be an object belonging to ${}^{\perp_1}\sR$.
 By assumption, there exists a short exact sequence $0\rarrow R'
\rarrow L\rarrow E\rarrow0$ in $\sE$ with $R'\in\sR$ and $L\in\sL$.
 Now $\Ext^1_\sE(E,R')=0$, hence $E$ is a direct summand of~$L$.
\end{proof}

 Let $(f_{ij}\:F_i\to F_j)_{0\le i<j\le\alpha}$ be an inductive system
in $\sE$ indexed by an ordinal~$\alpha$.
 Such a inductive system is said to be a \emph{smooth chain} if
$F_j=\varinjlim_{i<j}F_i$ for every limit ordinal $j\le\alpha$.
 A smooth chain $(f_{ij}\:F_i\to F_j)_{0\le i<j\le\alpha}$ is said to
be an \emph{$\alpha$\+filtration} (of the object $F=F_\alpha\in\sE$) if
$F_0=0$ and the morphism $f_{i,i+1}\:F_i\rarrow F_{i+1}$ is
a monomorphism in $\sE$ for every $0\le i<\alpha$.
 If an $\alpha$\+filtration $(f_{ij}\:F_i\to F_j)_{0\le i<j\le\alpha}$
is given, then the object $F=F_\alpha$ is said to be
\emph{$\alpha$\+filtered} by the cokernels
$(S_i=F_{i+1}/F_i)_{0\le i<\alpha}$ of the morphisms~$f_{i,i+1}$.

 In a different terminology, one says that the object $F$ is
a \emph{transfinitely iterated extension} of the family of objects
$(S_i)_{0\le i<\alpha}$ \cite[Definition~4.3]{PR}.
 To remove an ambiguity, one speaks of \emph{transfinitely iterated
extensions in the sense of the directed colimit} (as opposed to
similar extensions in the sense of the directed limit).

 Notice that, when the directed colimits in $\sE$ are not exact,
the morphisms $f_{ij}\:F_i\rarrow F_j$ in an $\alpha$\+filtration
need \emph{not} be monomorphisms, in general.
 See~\cite[Examples~4.4]{PR} for examples in which a zero object in
a contramodule category is represented as a transfinitely iterated
extension of a sequence of nonzero objects (indexed by the ordinal
of nonnegative integers~$\omega$).

 Given a class of objects $\sS\subset\sE$, we say that an object
$F\in\sE$ is \emph{filtered by\/~$\sS$} if $F$ admits
an $\alpha$\+filtration by objects from $\sS$ for some
ordinal~$\alpha$.
 The class of all objects filtered by $\sS$ is denoted by
$\Fil(\sS)\subset\sE$.

 The following result is a categorical generalization of what is
known as the \emph{Eklof lemma} in module
theory~\cite[Theorem~1.2]{Ekl}, \cite[Lemma~1]{ET}.
 Simultaneously, it generalizes the ``dual Eklof
lemma''~\cite[Proposition~18]{ET}, which goes back to
Lukas~\cite[Theorem~3.2]{Luk}.

\begin{prop} \label{eklof-lemma-prop}
 For any class of objects\/ $\sR\subset\sE$, the class of objects
${}^{\perp_1}\sR\subset\sE$ is closed under transfinitely iterated
extensions (in the sense of the directed colimit).
 In other words, we have ${}^{\perp_1}\sR=\Fil({}^{\perp_1}\sR)$. \qed
\end{prop}

\begin{proof}
 This is~\cite[Lemma~4.5]{PR}.
 Alternatively, the assertion can be deduced from
Lemma~\ref{cof-lifting}, Proposition~\ref{mono-epi-lifting}(a),
and Proposition~\ref{fil-cell}(a) below.
\end{proof}

\begin{lem} \label{hereditary-cotorsion-pairs}
 Let $(\sL,\sR)$ be a cotorsion pair in\/ $\sE$ such that the class\/
$\sL$ is generating and the class $\sR$ is cogenerating in\/~$\sE$.
 Then the following conditions are equivalent:
\begin{enumerate}
\item the class\/ $\sL$ contains the kernels of epimorphisms between
its objects;
\item the class\/ $\sR$ contains the cokernels of monomorphisms between
its objects;
\item $\Ext^2_\sE(L,R)=0$ for all $L\in\sL$ and $R\in\sR$;
\item $\Ext^n_\sE(L,R)=0$ for all $L\in\sL$, \,$R\in\sR$, and $n\ge1$.
\end{enumerate}
\end{lem}

\begin{proof}
 This lemma goes back to~\cite[Theorem~1.2.10]{GR}.
 Use the long exact sequences of $\Ext$ groups associated with short
exact sequences of objects in $\sE$ in order to prove the equivalences
(1)\,$\Longleftrightarrow$\,(3)\,$\Longleftrightarrow$\,(2).
 Then deduce~(4) from either~(1) or~(2).
 Details can be found in~\cite[Lemma 6.17]{Sto-ICRA}
or~\cite[Lemma 4.25]{SS}.
\end{proof}

 Notice that the assumption of Lemma~\ref{hereditary-cotorsion-pairs}
holds, in particular, for any complete cotorsion pair~$(\sL,\sR)$.
 A cotorsion pair satisfying the equivalent conditions of
Lemma~\ref{hereditary-cotorsion-pairs} is said to be \emph{hereditary}.

\Section{Weak Factorization Systems}  \label{wfs-secn}

 A weak factorization system is ``a half of a model structure''
(in the sense of~\cite{Quil,Hov-book}).
 The concept seems to appear for the first time explicitly in
the paper~\cite{AHRT}, but it is implicit, e.~g., in the formulation
of the small object argument in~\cite[Proposition~1.3]{Bek}.
 The terminology ``weak factorization system'' is explained by
the opposition to the conventional (non-weak) ``factorization
systems'', in which the diagonal filling of a square and
(consequently) the factorization of a morphism is presumed
to be unique~\cite{Bous}.
 See the discussion in the beginning of~\cite{AHRT}.

 Let $\sE$ be a category.
 One says that an object $A\in\sE$ is a \emph{retract} of an object
$B\in\sE$ if there exist morphisms $i\:A\rarrow B$ and $p\:B\rarrow A$
such that $pi=\id_A$ is the identity endomorphism.
 A morphism $f\:A\rarrow B$ in $\sE$ is said to be a retract of
a morphism $g\:C\rarrow D$ in $\sE$ if $f$~is a retract of~$g$ as
objects of the category $\sE^\to$ of morphisms in~$\sE$ (with
the commutative squares in $\sE$ being the morphisms in~$\sE^\to$).
 Given a class of morphisms $\cF$ in $\sE$, we denote by $\overline\cF$
the class of all the retracts of morphisms from $\cF$ in~$\sE$.

 Let $l\:A\rarrow B$ and $r\:C\rarrow D$ be two morphisms in~$\sE$.
 One says that $r$~has the \emph{right lifting property} with respect
to~$l$ or, which is the same, $l$~has the \emph{left lifting property}
with respect to~$r$, if for every pair of morphisms $f\:A\rarrow C$ and
$g\:B\rarrow D$ such that $rf=gl$ there exists a morphism $t\:B\rarrow C$
such that $f=tl$ and $g=rt$.
 In other words, any commutative square as in the diagram can be filled
with a diagonal arrow making both the triangles commutative:
$$
 \xymatrix@C=4em@R=4em{
  A \ar[r] \ar[d]_-l & C \ar[d]^-r \\
  B \ar[r] \ar@{-->}[ru] & D
 }
$$

 One can easily check that if a morphism~$r$ has the right lifting
property with respect to a morphism~$l$, then any retract of~$r$ has
the right lifting property with respect to any retract of~$l$.

 Let $\cL$ and $\cR$ be two classes of morphisms in~$\sE$.
 We will denote by $\cL^\square$ the class of all morphisms~$x$ in
$\sE$ having the right lifting property with respect to all
morphisms $l\in\cL$.
 Similarly, we let ${}^\square\cR$ denote the class of all morphisms~$y$
in $\sE$ having the left lifting property with respect to all
morphisms $r\in\cR$.

 Let $(\cL,\cR)$ be a pair of classes of morphisms in~$\sE$.
 We will say that the pair $(\cL,\cR)$ \emph{has the lifting property}
if $\cR\subset\cL^\square$, or equivalently, $\cL\subset{}^\square\cR$.
 Furthermore, we will say that the pair $(\cL,\cR)$ \emph{has
the factorization property} if every morphism~$f$ in $\sE$ can be
decomposed as $f=rl$ with $l\in\cL$ and $r\in\cR$.
 A pair of classes of morphisms $(\cL,\cR)$ in $\sE$ is called
a \emph{weak factorization system} if $\cR=\cL^\square$, \
$\cL={}^\square\cR$, and the pair $(\cL,\cR)$ has the factorization
property.

\begin{lem} \label{retract-lemma}
 Let $(\cL,\cR)$ be a pair of classes of morphisms in\/ $\sE$ having
the lifting and factorization properties.
 Then the pair of classes of morphisms $(\overline\cL,\overline\cR)$
is a weak factorization system in\/~$\sE$.
\end{lem}

\begin{proof}
 In view of the above discussion of the lifting properties of retracts,
we only need to check the inclusions $\cL^\square\subset\overline\cR$
and ${}^\square\cR\subset\overline\cL$.
 Indeed, let $f=rl$ be a composition of two morphisms in~$\sE$ with $l\in\cL$ and $r\in\cR$.
 One observes that if $f$~has the right lifting property with respect
to~$l$, then $f$~is a retract of~$r$; and similarly, if $f$~has the left
lifting property with respect to~$r$, then $f$~is a retract of~$l$;
see~\cite[Lemma~1.1.9]{Hov-book}.
\end{proof}

 Let $(f_{ij}\:E_i\to E_j)_{0\le i<j\le\alpha}$ be an inductive system
in a category $\sE$ indexed by an ordinal~$\alpha$.
 Assume that this inductive system is a smooth chain (as defined in
Section~\ref{cotorsion-pairs-secn}, i.~e., $E_j=\varinjlim_{i<j}E_i$
for every limit ordinal $j\le\alpha$).
 Then we will say that the morphism $f_{0,\alpha}\:E_0\rarrow E_\alpha$
is a \emph{transfinite composition} of the morphisms $f_{i,i+1}\:
E_i\rarrow E_{i+1}$, \ $0\le i<\alpha$ (\emph{in the sense of
the directed colimit}).

 Let $f\:A\rarrow B$ be a morphism in~$\sE$.
 Then a morphism $f'\:X\rarrow Y$ in $\sE$ is said to be
a \emph{pushout} of the morphism~$f$ if there exists a pair of morphisms
$g\:A\rarrow X$ and $g'\:B\rarrow Y$ such that $A\rarrow B\rarrow Y$,
\ $A\rarrow X\rarrow Y$ is a cocartesian square (or in a different
terminology, a pushout square) in~$\sE$.

 Let $\cS$ be a class of morphisms in~$\sE$.
 We will denote by $\Cell(\cS)$ the closure of $\cS$ under pushouts
and transfinite compositions (in the sense of the directed colimit).
 Assuming that $\sE$ is cocomplete, any pushout of a transfinite
composition is a transfinite composition of pushouts; so $\Cell(\cS)$
consists of all the transfinite compositions of pushouts of morphisms
from~$\cS$.
 Moreover, $\Cell(\cS)$ is then also closed under coproducts
of morphisms; see~\cite[Lemma 2.1.13]{Hov-book}.
 Next, let $\Cof(\cS)$ denote the closure of $\cS$ under
pushouts, transfinite compositions, and retracts.
 Assuming that $\sE$ is cocomplete, any pushout of a retract is
a retract of a pushout, and any transfinite composition of retracts is
a retract of a transfinite composition; so the class
$\Cof(\cS)=\overline{\Cell(\cS)}$ is the closure of the class
$\Cell(\cS)$ under retracts.

\begin{lem} \label{cof-lifting}
 For any class of morphisms\/ $\cR$ in\/ $\sE$, the class of morphisms
${}^\square\cR$ is closed under pushouts, transfinite compositions
(in the sense of the directed colimit), coproducts and retracts.
 In other words, we have ${}^\square\cR=\Cof({}^\square\cR)$. \qed
\end{lem}

 Now let $\sE$ be an abelian category.
 Let $\sL$ and $\sR\subset\sE$ be two classes of objects.
 A morphism in $\sE$ is said to be an \emph{$\sL$\+monomorphism} if
it is a monomorphism whose cokernel belongs to~$\sL$.
 Similarly, a morphism in $\sE$ is said to be
an \emph{$\sR$\+epimorphism} if it is an epimorphism whose kernel
belongs to~$\sR$.
 We denote the class of all $\sL$\+monomorphisms by $\sL\Mono$ and
the class of all $\sR$\+epimorphisms by $\sR\Epi$.

\begin{prop} \label{mono-epi-lifting}
\textup{(a)} The inclusion\/ $\sR\Epi\subset\sL\Mono^\square$ holds
if and only if\/ $\sR\subset\sL^{\perp_1}$. \par
\textup{(b)} Moreover, for any class of objects\/ $\sL\subset\sE$,
the kernel of any morphism from $\sL\Mono^\square$ belongs
to\/~$\sL^{\perp_1}$. \par
\textup{(c)} If the class\/ $\sL\subset\sE$ is generating, then
the class\/ $\sL\Mono^\square$ consists of epimorphisms.
 Consequently, $\sL\Mono^\square=\sL^{\perp_1}\Epi$ in this case.
\end{prop}

\begin{proof}
 This is~\cite[Lemmas~3.1 and~4.2]{PR}.
 In particular, the ``if'' implication in part~(a) can be found
in~\cite[Lemma~2.4]{Gil} or~\cite[Lemma~1 in Section~9.1]{Psemi}
(the argument seems to go back to~\cite[proof of Proposition~4.2]{Hov}).
 For later use, we note that in order to prove in part~(c) that any
$r\in\sL\Mono^\square$ is an epimorphism, it suffices to use that~$r$
has the right lifting property with respect to all the morphisms
$0\rarrow L$, \,$L\in\sL$.
\end{proof}

 Notice that the class $\sL\Mono^\square$ does not always consist of
epimorphisms.
 To give a trivial example, if $\sL=\varnothing$ or $\sL=\{0\}$, then
all the $\sL$\+monomorphisms are isomorphisms, hence all the morphisms
in $\sE$ belong to~$\sL\Mono^\square$ (cf.\
Examples~\ref{nongenerating-examples} below).
 However, \emph{if} the class $\sL\Mono^\square$ consists of
epimorphisms, \emph{then} it is clear from
Proposition~\ref{mono-epi-lifting}(a\+-b) that $\sL\Mono^\square=
\sL^{\perp_1}\Epi$.

 A weak factorization system $(\cL,\cR)$ in an abelian category $\sE$
is said to be \emph{abelian} if the class $\cL$ consists of
monomorphisms, the class $\cR$ consists of epimorphisms,
a monomorphism $l\:A\rarrow B$ belongs to $\cL$ if and only if
the morphism $0\rarrow\coker(l)$ belongs to $\cL$, and an epimorphism
$r\:C\rarrow D$ belongs to $\cR$ if and only if the morphism
$\ker(r)\rarrow0$ belongs to~$\cR$.
 In other words, a weak factorization system $(\cL,\cR)$ is abelian
if and only if there exists a pair of classes of objects $(\sL,\sR)$
in $\sE$ such that $\cL=\sL\Mono$ and $\cR=\sR\Epi$.

\begin{thm} \label{wfs-complete-cotorsion}
 Let $(\sL,\sR)$ be a pair of classes of objects in an abelian
category~$\sE$.
 Then the pair of classes of morphisms\/ $\cL=\sL\Mono$ and\/
$\cR=\sR\Epi$ forms a weak factorization system in\/ $\sE$ if and
only if $(\sL,\sR)$ is a complete cotorsion pair.
 So abelian weak factorization systems correspond bijectively to
complete cotorsion pairs in\/~$\sE$.
\end{thm}

\begin{proof}
 This result is essentially due to Hovey~\cite{Hov}.
 ``Only if'': it is clear from Proposition~\ref{mono-epi-lifting}(a)
that the equation $\sR\Epi=\sL\Mono^\square$ implies
$\sR=\sL^{\perp_1}$; similarly, the equation
$\sL\Mono={}^\square(\sR\Epi)$ implies $\sL={}^{\perp_1}\sR$.

 To prove existence of special precover sequences for $(\sL,\sR)$,
consider an object $E\in\sE$.
 Since the pair of classes $(\cL,\cR)$ has the factorization property
by assumption, the morphism $0\rarrow E$ can be factorized as
$0\rarrow M\rarrow E$, where the morphism $l\:0\rarrow M$ belongs to
$\cL$ and the morphism $r\:M\rarrow E$ belongs to~$\cR$.
 Now we have $M\in\sL$ and the morphism $r\:M\rarrow E$ is
an epimorphism with the kernel belonging to~$\sR$.
 The dual argument proves existence of special preenvelope sequences.

 ``If'': If $(\sL,\sR)$ is a complete cotorsion pair, then the class
$\sL$ is generating and the class $\sR$ is cogenerating.
 By Proposition~\ref{mono-epi-lifting}(c) and its dual version,
it follows that $\sL\Mono^\square=\sR\Epi$ and ${}^\square(\sR\Epi)
=\sL\Mono$.

 Now let $f\:A\rarrow B$ be a morphism in~$\sE$.
 We want to decompose~$f$ into an $\sL$\+monomorphism followed by
$\sR$\+epimorphism.
 We follow an argument from~\cite[Section~9.1]{Psemi}
(for the classical approach, see~\cite[Proposition~5.4]{Hov}).
 Choose an object $P\in\sL$ together with a morphism $p\:P\rarrow B$
such that the morphism $(f,p)\:A\oplus P\rarrow B$ is an epimorphism.
 (E.~g., one can choose $p$~to be an epimorphism; alternatively,
when $f$~is an epimorphism, one can take $P=0$.)
 Denote by $K$ the kernel of the morphism~$(f,p)$, and choose
a special preenvelope sequence $0\rarrow K\rarrow R\rarrow L'\rarrow0$
for the object $K\in\sE$ (so $R\in\sR$ and $L'\in\sL$).

 Denote the monomorphism $K\rarrow A\oplus P$ by $(a,q)$ and
the monomorphism $K\rarrow R$ by~$k$.
 Let $C$ be the cokernel of the monomorphism $(a,q,k)\:K\rarrow
A\oplus P\oplus R$.
 Consider the morphism $(f,p,0)\:A\oplus P\oplus R\rarrow B$.
 The composition $(f,p,0)\circ(a,q,k)=(f,p)\circ(a,q)$ vanishes;
so the morphism $(f,p,0)$ factorizes through the epimorphism
$A\oplus P\oplus R\rarrow C$, providing a morphism $r\:C\rarrow B$.
 Denote by $l\:A\rarrow C$ the composition of the coproduct
inclusion $(\id_A,0,0)\:A\rarrow A\oplus P\oplus R$ with
the epimorphism $A\oplus P\oplus R\rarrow C$.
 By construction, we have $rl=f$.

 Finally, the morphism~$r$ is an epimorphism with the kernel
$\ker(r)=R$, hence $r\in\sR\Epi$.
 The morphism~$l$ is a monomorphism whose cokernel is the middle
term of a short exact sequence $0\rarrow P\rarrow\coker(l)
\rarrow L'\rarrow0$.
 As both the objects $P$ and $L'$ belong to $\sL$ and the class $\sL$
is closed under extensions, we can conclude that $l\in\sL\Mono$.

 Notice that the above construction is not self-dual; one could also
proceed in the dual way, choosing an object $J\in\sR$ together with
a morphism $j\:A\rarrow J$ such that $(f,j)\:A\rarrow B\oplus J$ is
a monomorphism, considering the cokernel of $(f,j)$, etc.\
(see~\cite[Section~9.2]{Psemi}).
\end{proof}

\begin{prop} \label{fil-cell}
 Let\/ $\sS$ be a class of objects in an abelian category\/~$\sE$.
 Then \par
\textup{(a)} any\/ $\Fil(\sS)$\+monomorphism in\/ $\sE$ belongs to\/
$\Cell(\sS\Mono)$; \par
\textup{(b)} the cokernel of any morphism from\/ $\Cell(\sS\Mono)$
belongs to\/ $\Fil(\sS)$.
\end{prop}

\begin{proof}
 This is~\cite[Lemma~4.6(a\+-b) and Remark~4.7]{PR}.
 Part~(b) is straightforward; part~(a) is more involved.
\end{proof}

\Section{Small Object Argument}

 Let $\lambda$~be an infinite cardinal.
 A poset $I$ is said to be \emph{$\lambda$\+directed} if any subset
$J\subset I$ of the cardinality less than~$\lambda$ has an upper bound
in $I$, i.~e., an element $i\in I$ such that $j\le i$ for all $j\in J$.
 In particular, $I$ is $\omega$\+directed (where $\omega$~denotes
the cardinal of nonnegative integers) if and only if $I$ is directed
in the usual sense.

 A \emph{$\lambda$\+directed colimit} in a category $\sE$ is the colimit
of a diagram indexed by a $\lambda$\+directed poset.
 Assuming that all the $\lambda$\+directed colimits exist in $\sE$,
an object $E\in\sE$ is said to be \emph{$\lambda$\+presentable} if
the functor $\Hom_\sE(E,{-})\:\sE\rarrow\Sets$ preserves
$\lambda$\+directed colimits.

 Let $\lambda$~be a regular infinite cardinal.
 A category $\sE$ is said to be \emph{$\lambda$\+accessible} if all
the $\lambda$\+directed colimits exist in $\sE$ and there is a set of
$\lambda$\+presentable objects $\sG\subset\sE$ such that every object
in $\sE$ is a $\lambda$\+directed colimit of objects from~$\sG$.
 A cocomplete $\lambda$\+accessible category is said to be
\emph{locally $\lambda$\+presentable}.

 Equivalently, a cocomplete category $\sE$ is locally
$\lambda$\+presentable if and only if it has a strongly generating set
of $\lambda$\+presentable objects~\cite[Theorem~1.20]{AR}.
 We do not define what it means for a set of generators in a category
to be a set of strong generators (see~\cite[Section~0.6]{AR} for
the discussion), as we are only interested in locally presentable
\emph{abelian} categories in this paper.
 In an abelian category, any set of generators is a set of strong
generators.

 A category is said to be \emph{accessible} if it is
$\lambda$\+accessible for some regular cardinal~$\lambda$.
 Similarly, a category is said to be \emph{locally presentable} if it is
locally $\lambda$\+presentable for some regular cardinal~$\lambda$.

 The following theorem summarizes Quillen's classical ``small object
argument''~\cite[Lemma~II.3.3]{Quil}.
 For a more general formulation, see~\cite[Theorem~2.1.14]{Hov-book} or~\cite[Proposition~2.1]{SS}.

\begin{thm} \label{small-object-argument}
 Let\/ $\sE$ be a locally presentable category and\/ $\cS$ be a set
of morphisms in\/~$\sE$.
 Then the pair of classes of morphisms\/ $\Cell(\cS)$ and\/
$\cS^\square$ has the factorization property.
 Consequently, the pair of classes of morphisms\/ $\Cof(\cS)$ and\/
$\cS^\square$ is a weak factorization system.
 Moreover, the factorization $f=rl$ of an arbitrary morphism~$f$ in\/
$\sE$ into the composition of a morphism $r\in\cS^\square$ and
$l\in\Cell(\cS)$ can be chosen so that it depends functorially on~$f$.
\end{thm}

\begin{proof}
 See, e.~g., \cite[Proposition~1.3]{Bek}.
 The assertion of the theorem can be strengthened, in particular,
by defining the class $\Cof(\cS)$ in a more restrictive way (this involves
formulating a correspondingly adjusted version of Lemma~\ref{retract-lemma});
see~\cite[Definition~1.1(ii)]{Bek}.
\end{proof}

\begin{prop} \label{S-monos-are-pushouts}
 Let $\lambda$~be a regular cardinal, $\sE$ be a locally
$\lambda$\+presentable abelian category, and $\sS$ be a set of
$\lambda$\+presentable objects in\/~$\sE$.
 Then any\/ $\sS$\+monomorphism in\/ $\sE$ is a pushout of
an\/ $\sS$\+monomorphism with a $\lambda$\+presentable codomain.
\end{prop}

\begin{proof}
 This is~\cite[Lemma~3.4]{PR}.
\end{proof}

 The next two theorems extend the classical Eklof--Trlifaj
theorem about the cotorsion pair generated by a set of
modules~\cite[Theorem~10]{ET} to the realm of locally presentable
abelian categories.
 (For a much earlier generalization of the Eklof--Trlifaj theorem
to Grothendieck categories with enough projective objects,
see~\cite[Corollary~2.7]{AEGO}, while an early version applicable to
arbitrary Grothendieck categories can be found
in~\cite[Theorem~6.5]{Hov}.)
 Given an additive category $\sE$ and a class of objects $\sS\subset
\sE$, we denote by $\Add(\sS)=\Add_\sE(\sS)\subset\sE$ the class
of all direct summands of coproducts of copies of objects from
$\sS$ in~$\sE$.

\begin{thm} \label{cotorsion-pair-generated-by-set-complete}
 Let\/ $\sE$ be a locally presentable abelian category and\/
$(\sL,\sR)$ be a cotorsion pair generated by a set of objects
in\/~$\sE$.
 Assume that the class\/ $\sL$ is generating and the class\/ $\sR$ is
cogenerating in\/~$\sE$.
 Then the cotorsion pair $(\sL,\sR)$ is complete.
\end{thm}

\begin{proof}
 This is~\cite[Corollary~3.6]{PR}.
 Let $\sS_0$ be a set of objects in $\sE$ such that
$\sR=\sS_0^{\perp_1}$.
 Using the assumption that the class $\sL$ is generating, one can
construct a possibly larger set of objects $\sS\subset\sE$, \
$\sS_0\subset\sS$, such that $\sR=\sS^{\perp_1}$ and every object
of $\sE$ is a quotient object of an object from $\Add(\sS)$.

 Let $\lambda$ be a regular cardinal such that the category $\sE$ is
locally $\lambda$\+presentable and all the objects in $\sS$ are
$\lambda$\+presentable
(this is always possible, see~\cite[Remark below Theorem~1.20]{AR}).
 Applying Theorem~\ref{small-object-argument} to the set $\cS$ of
(representatives of isomorphism classes of) all $\sS$\+monomorphisms
with $\lambda$\+presentable codomains in $\sE$, we conclude the pair
of classes $\Cof(\cS)$ and $\cS^\square$ forms a weak factorization
system in~$\sE$.

 By Proposition~\ref{S-monos-are-pushouts} and Lemma~\ref{cof-lifting},
we have $\cS^\square=\sS\Mono^\square$ and we also deduce that each
$r\in\sS\Mono^\square$ has the right lifting property with
respect to all morphisms of the form $0\rarrow L$, $L\in\Add(\sS)$.
 Using the fact that the class $\Add(\sS)$ is generating in~$\sE$
and following the proof of Proposition~\ref{mono-epi-lifting}(c),
one can see that $\sS\Mono^\square=\sR\Epi$.
 The dual assertion to Proposition~\ref{mono-epi-lifting}(c) tells
that ${}^\square(\sR\Epi)=\sL\Mono$, due to the assumption that
the class $\sR$ is cogenerating.
 Thus we have $\cS^\square=\sR\Epi$ and $\Cof(\cS)=\sL\Mono$, so
the pair of classes $\sL\Mono$ and $\sR\Epi$ is a weak factorization
system, and the desired assertion follows from
Theorem~\ref{wfs-complete-cotorsion}.
\end{proof}

\begin{thm} \label{cotorsion-pair-generated-by-set-left-class}
 Let\/ $\sE$ be a locally presentable abelian category and\/
$(\sL,\sR)$ be the cotorsion pair generated by a set of objects\/ $\sS$
in\/~$\sE$.
 Assume that the class of objects\/ $\Fil(\sS)$ is generating
in\/~$\sE$.
 Then one has\/ $\sL=\Fil(\sS)^\oplus$.
\end{thm}

\begin{proof}
 This is a particular case of~\cite[Theorem~4.8(d)]{PR}.
 By Proposition~\ref{fil-cell}(a), we have $\Fil(\sS)\Mono\subset
\Cell(\sS\Mono)$.
 In view of Lemma~\ref{cof-lifting}, it follows that
$\Fil(\sS)\Mono^\square=\sS\Mono^\square$.
 Applying Proposition~\ref{mono-epi-lifting}(c), we conclude that
the class $\sS\Mono^\square$ consists of epimorphisms.
 By Proposition~\ref{mono-epi-lifting}(a\+-b), it follows that
$\sS\Mono^\square=\sS^{\perp_1}\Epi=\sR\Epi$.

 Let $\lambda$ be a regular cardinal such that the category $\sE$ is
locally $\lambda$\+presentable and all the objects in $\sS$ are
$\lambda$\+presentable.
 Applying Theorem~\ref{small-object-argument} to the $\cS$ of
(representatives of isomorphism classes of) all $\sS$\+monomorphisms
with $\lambda$\+presentable codomains in $\sE$, we conclude that
every morphism in $\sE$ factorizes as the composition of a morphism
from $\Cell(\cS)$ followed by a morphism from $\cS^\square$.
 By Proposition~\ref{S-monos-are-pushouts} and Lemma~\ref{cof-lifting},
we have $\cS^\square=\sS\Mono^\square$; thus $\cS^\square=\sR\Epi$.

 Given an object $M\in\sE$, consider the morphism $0\rarrow M$ and
decompose it as $0\rarrow F\rarrow M$ so that the morphism
$0\rarrow F$ belongs to $\Cell(\cS)$ and the morphism $F\rarrow M$
belongs to $\cS^\square$.
 By Proposition~\ref{fil-cell}(b), we have $F\in\Fil(\sS)$.
 The argument above tells that $F\rarrow M$ is an epimorphism with
a kernel $R\in\sR$.
 We have constructed a special precover (short exact) sequence
$0\rarrow R\rarrow F\rarrow M\rarrow0$ for $M$ with $F\in\Fil(\sS)$
and $R\in\sR$.
 Now if $M\in\sL$, then $\Ext^1_\sE(M,R)=0$, hence $M$ is a direct
summand of~$F$.
\end{proof}

\begin{rem}
 For any cotorsion pair $(\sL,\sR)$ in an abelian category $\sE$,
the class $\sL$ contains all the projective objects in $\sE$ and
the class $\sR$ contains all the injective objects in~$\sE$.
 Hence the assumption that the class $\sL$ is generating holds
automatically if $\sE$ has enough projectives, and the assumption
that $\sR$ is cogenerating holds automatically if $\sE$ has enough
injectives.
 So Theorem~\ref{cotorsion-pair-generated-by-set-complete} is usually
stated without these assumptions for module categories $\sE$, or
more generally, for Grothendieck abelian categories $\sE$ with
enough projective objects.
 For categories without enough projective/injective objects, however,
Theorem~\ref{cotorsion-pair-generated-by-set-complete} does \emph{not}
hold without such assumptions about the classes $\sL$ and $\sR$,
as the following counterexamples demonstrate.
\end{rem}

\begin{exs} \label{nongenerating-examples}
 (1)~Let $p$~be a prime number and $\sA$ be the category of $p$\+primary
torsion abelian groups, that is abelian groups $A$ such that for every
$a\in A$ there exists $n\ge1$ for which $p^na=0$ in~$A$.
 Then $\sA$ is a locally finitely presentable Grothendieck
abelian category.
 Let $\sS$ be the empty set of objects in $\sA$, or alternatively, if
one wishes, let $\sS$ be the set consisting of the zero object only,
$\sS=\{0\}\subset\sA$.

 Then the class $\sR=\sS^{\perp_1}\subset\sA$ coincides with the whole
category $\sA$, and the class $\sL={}^{\perp_1}\sR$ consists of all
the projective objects in~$\sA$, which means the zero object only.
 Indeed, for any abelian group $A\in\sA$, one has
$\Ext^1_\sA(A,\boZ/p\boZ)\simeq\Hom_\sA({}_pA,\boZ/p\boZ)\ne0$ if
$A\ne0$, where ${}_pA\subset A$ denotes the subgroup of all elements
annihilated by~$p$ in~$\sA$.
 Thus we have $\sL=\{0\}$ and $\sR=\sA$, which is clearly \emph{not}
a complete cotorsion pair in $\sA$ (the special precover sequences
do not exist).

 It is instructive to consider the classes $\cL=\Cof(\cS)$ and
$\cR=\cS^\square$ from the proof of
Theorem~\ref{cotorsion-pair-generated-by-set-left-class} in this case.
 The set $\cS$ of representatives of isomorphism classes of
$\sS$\+monomorphisms with finitely presentable
(or $\lambda$\+presentable) codomains consists of isomorphisms;
so $\cR$ is the class of all morphisms in~$\sA$.
 It follows that $\cL$ is the class of all isomorphisms in~$\sA$.

\smallskip
 (2)~Let $p$~be a prime number and $\sB$ be the category of
$p$\+contramodule abelian groups, that is abelian groups $B$ such that
$\Hom_{\boZ}(\boZ[p^{-1}],B)=0=\Ext^1_{\boZ}(\boZ[p^{-1}],B)$.
 Then $\sB$ is a locally $\aleph_1$\+presentable abelian category with
enough projective objects~\cite[Example~4.1(3)]{PR}.
 In fact, the group of $p$\+adic integers $P=\boZ_p$ is a projective
generator of~$\sB$.
 Let $\sS=\{P,S\}$ be the set consisting of the projective generator $P$
and the simple abelian group $S=\boZ/p\boZ$, or alternatively, if one
wishes, let $\sS=\{S\}$ be the set consisting of the group $S$ only.

 Then the class $\sR=\sS^{\perp_1}\subset\sB$ consists of the zero
object only.
 Indeed, for any abelian group $B\in\sB$, one has
$\Ext^1_\sB(\boZ/p\boZ,B)\simeq B/pB\ne0$ if $B\ne0$ (as $B=pB$ would
imply surjectivity of the natural map $\Hom_{\boZ}(\boZ[p^{-1}],B)
\rarrow B$) \cite[Example~4.1(4)]{PR}.
 Hence the class $\sL={}^{\perp_1}\sR$ coincides with the whole
category~$\sB$.
 Thus we have $\sL=\sB$ and $\sR=\{0\}$, which is clearly \emph{not}
a complete cotorsion pair in $\sB$ (the special preenvelope sequences
do not exist).

 Notice that Theorem~\ref{cotorsion-pair-generated-by-set-left-class}
is not applicable in this example with $\sS=\{S\}$, but it becomes
applicable if one takes $\sS=\{P,S\}$.
 So one can conclude that $\sB=\Fil(\{P,S\})^\oplus$.

 It is instructive to consider the classes $\cL=\Cof(\cS)$ and
$\cR=\cS^\square$ from the proofs of
Theorems~\ref{cotorsion-pair-generated-by-set-complete}
and~\ref{cotorsion-pair-generated-by-set-left-class} in this case.
 Take $\sS=\{P,S\}$; then, following the proof of
Theorem~\ref{cotorsion-pair-generated-by-set-left-class},
\,$\cR=\sR\Epi$ is the class of all epimorphisms in $\sB$ with
the kernels in $\sR=\{0\}$, which means that $\cR$ is the class
of all isomorphisms in~$\sB$.
 Hence $\cL$ is the class of all morphisms in~$\sB$.
\end{exs}

 A weak factorization system $(\cL,\cR)$ in $\sE$ is said to be
\emph{cofibrantly generated} if there exists a set of morphisms
$\cS$ in $\sE$ such that $\cR=\cS^\square$.

\begin{lem} \label{abelian-wfs-cof-gen}
 Let\/ $\sE$ be a locally presentable abelian category and
$(\cL,\cR)$ be an abelian weak factorization system in\/~$\sE$.
 Let $(\sL,\sR)$ be the corresponding complete cotorsion pair in
$\sE$, as in Theorem~\ref{wfs-complete-cotorsion}; so\/
$\cL=\sL\Mono$ and\/ $\cR=\sR\Epi$.
 Then the weak factorization system $(\cL,\cR)$ is cofibrantly
generated if and only if the cotorsion pair $(\sL,\sR)$ is generated
by a set of objects.
\end{lem}

\begin{proof}
 This is essentially a generalization
of~\cite[Proposition~1.2.7]{Bec}.
 The ``if'' implication is clear from the proof of
Theorem~\ref{cotorsion-pair-generated-by-set-complete}.
 To prove the ``only if'', let $\cS$ be a set of morphisms in $\sE$
such that $\cR=\cS^\square$.
 Then $\cS\subset\sL\Mono$.
 Denote by $\sS$ the set (of representatives of isomorphism classes)
of cokernels of all the morphisms from~$\cS$.
 Then $\cS\subset\sS\Mono\subset\sL\Mono$, hence
$\cR=\sS\Mono^\square$.
 As the class $\cR$ consists of epimorphisms, one can see from
Lemma~\ref{mono-epi-lifting}(a\+-b) that $\sR=\sS^{\perp_1}$.
\end{proof}

\Section{Abelian Model Structures} \label{abelian-model-secn}

 Abstracting from the definition of derived categories, if one is
given a category $\sE$ and a class of morphisms $\cW$ in $\sE$, one
often wishes to understand the category $\sE[\cW^{-1}]$, where
one freely adds inverses to all morphisms in $\cW$. This is analogous
to localization in commutative algebra and can be 
achieved by a similar construction~\cite[\S I.1]{GaZi}, but unlike
in commutative algebra it is in general extremely difficult
to understand what $\sE[\cW^{-1}]$ looks like.
The topologically motivated notion of model category
solves this problem (see~\cite[Section~I.1]{Quil}
or~\cite[Theorem 1.2.10]{Hov-book})
at the cost of requiring more structure than just the pair $(\sE,\cW)$.

\smallskip
 A \emph{model structure} on a category $\sE$ is a triple of classes of
morphisms $\cL$, $\cR$, and $\cW$ satisfying the following conditions:
\begin{enumerate}
\renewcommand{\theenumi}{\roman{enumi}}
\item the pair of classes $\cL$ and $\cR\cap\cW$ is a weak factorization
system;
\item the pair of classes $\cL\cap\cW$ and $\cR$ is a weak factorization
system;
\item the class $\cW$ is closed under retracts and satisfies
the two-out-of-three property: for any
composable pair of morphisms $f$ and~$g$ in $\sE$, if two of the three
morphisms $f$, $g$, and $gf$ belong to $\cW$, then the third one also
does.
\end{enumerate}
 Morphisms in the classes $\cL$, $\cR$, and $\cW$ are called
\emph{cofibrations}, \emph{fibrations}, and \emph{weak equivalences},
respectively.
 Morphisms in the class $\cL\cap\cW$ are called \emph{trivial
cofibrations}, and morphisms in the class $\cR\cap\cW$ are called
\emph{trivial fibrations}.

 A \emph{model category} is a complete, cocomplete category with
a model structure $(\cL,\cW,\cR)$, and $\sE[\cW^{-1}]$ is
called the \emph{homotopy category} of the model category in this context.
We will be  interested only in the so-called stable model categories.
This condition implies that the homotopy category carries
a natural triangulated structure; see \cite[Chapter 7]{Hov-book}.
We will not use this condition directly, however, but only through
\cite[Corollary~1.1.15 and the preceding discussion]{Bec}, which says
that any hereditary abelian model category (to be defined later in this
section) is stable.

 In order to use the advantage of set-theoretic methods, one
usually considers the following technical, but widely satisfied conditions.
 A model structure $(\cL,\cW,\cR)$ on a category $\sE$ is said to be
\emph{cofibrantly generated} if both the weak factorization systems
$(\cL,\>\cR\cap\cW)$ and $(\cL\cap\cW,\>\cR)$ are cofibrantly generated.
Cofibrantly generated model categories whose underlying category is
locally presentable are called \emph{combinatorial}
(all locally presentable categories are complete and cocomplete).

 An important point is that the homotopy category of
any stable combinatorial model category is well-generated triangulated
in the sense of Neeman;
see~\cite[Proposition~6.10]{Ros} or~\cite[Theorems~3.1 and~3.9]{CR}.
The theory of well-generated triangulated categories~\cite{Neem-book,Kra0}
gained popularity because it is applicable to a wide range of naturally occurring
triangulated categories and allows (to a somewhat limited extent)
to obtain results analogous to consequences of the small object argument
purely in the language of triangulated categories, without a reference
to any enhancement.

\smallskip
 In Sections~\ref{loc-pres-secn}\+-\ref{coderived-secn}, our aim will be
to construct certain concrete stable combinatorial model structures
on the categories of (unbounded) complexes in some
locally presentable abelian categories.
It will follow that the corresponding homotopy categories---in our case 
derived, coderived and contraderived categories of the abelian
categories---are well-generated triangulated.

%
%
%

\smallskip
 Let $\sE$ be a category with a model structure.
 An object $L\in\sE$ is said to be \emph{cofibrant} if the morphism
$\varnothing\rarrow L$ is a cofibration.
 An object $R\in\sE$ is said to be \emph{fibrant} if the morphism
$R\rarrow*$ is a fibration.
 Here $\varnothing$ and~$*$ denote the initial and the terminal object
in the category $\sE$, respectively (which we presume to exist).
 For any additive category $\sE$, they are the same: $\varnothing=0=*$.

 More generally, a category $\sE$ is called \emph{pointed} if it has
an initial and a terminal object, and they coincide.
 The initial-terminal object of a pointed category is called
the \emph{zero object} and denoted by~$0$.
 Given a pointed category $\sE$ with a model structure, an object
$W\in\sE$ is said to be \emph{weakly trivial} if the morphism
$0\rarrow W$ is a weak equivalence, or equivalently, the morphism
$W\rarrow 0$ is a weak equivalence.
 Weakly trivial cofibrant objects are said to be \emph{trivially
cofibrant}, and weakly trivial fibrant objects are said to be
\emph{trivially fibrant}.

 The next definition of an \emph{abelian model structure} is due
to Hovey~\cite{Hov} (see~\cite{Gil2} for a generalization to exact
categories the sense of Quillen).
 Let $\sE$ be an abelian category.
 A model structure $(\cL,\cW,\cR)$ on $\sE$ is said to be \emph{abelian}
if $\cL$ is the class of all monomorphisms with cofibrant cokernels
and $\cR$ is the class of all epimorphisms with fibrant kernels.
 A model category is said to be \emph{abelian} if its underlying
category is abelian and the model structure is abelian.

\begin{lem} \label{trivial-co-fibrations}
 In an abelian model structure, $\cL\cap\cW$ is the class of all
monomorphisms with trivially cofibrant cokernels and $\cR\cap\cW$ is
the class of all epimorphisms with trivially fibrant kernels.
\end{lem}

\begin{proof}
 This is a part of~\cite[Proposition~4.2]{Hov}.
 Let $\sL$, $\sR$, and $\sW$ denote the classes of cofibrant, fibrant,
and weakly trivial objects in $\sE$, respectively.
 Then we have $\cL=\sL\Mono$ and $\cR=\sR\Epi$.
 The class $\sL\Mono^\square=\cR\cap\cW\subset\cR$ consists
of epimorphisms.
 By Proposition~\ref{mono-epi-lifting}(a\+-b), it follows that
$\sL\Mono^\square=\sL^{\perp_1}\Epi$.
 Similarly, the class ${}^\square(\sR\Epi)=\cL\cap\cW\subset\cL$
consists of monomorphisms.
 The dual assertions to
Proposition~\ref{mono-epi-lifting}(a\+-b) tell that
${}^\square(\sR\Epi)=({}^{\perp_1}\sR)\Mono$.
 Now it is clear that $\sL^{\perp_1}=\sR\cap\sW$ is the class of all
trivially fibrant objects and ${}^{\perp_1}\sR=\sL\cap\sW$ is
the class of all trivially cofibrant objects.
\end{proof}

 A class of objects $\sW\subset\sE$ is said to be \emph{thick} if it
is closed under direct summands and, for any short exact sequence
$0\rarrow A\rarrow B\rarrow C\rarrow0$ in $\sE$, if two of the three
objects $A$, $B$, $C$ belong to $\sW$ then the third one also does.

\begin{thm} \label{abelian-model-structures-thm}
 Abelian model structures on an abelian category\/ $\sE$ correspond
bijectively to triples of classes of objects $(\sL,\sW,\sR)$
such that \par
\begin{enumerate}
\item the pair of classes\/ $\sL$ and\/ $\sR\cap\sW$ is a complete
cotorsion pair;
\item the pair of classes\/ $\sL\cap\sW$ and\/ $\sR$ is a complete
cotorsion pair;
\item $\sW$ is a thick class.
\end{enumerate}
 The correspondence assigns to a triple of classes of morphisms
$(\cL,\cW,\cR)$ the triple of classes of objects $(\sL,\sW,\sR)$,
where\/ $\sL$ is the class of all cofibrant objects, $\sR$ is
the class of all fibrant objects, and\/ $\sW$ is the class of all
weakly trivial objects in the model structure $(\cL,\cW,\cR)$.
 Conversely, to a triple of classes of objects $(\sL,\sW,\sR)$
the triple of classes of morphisms $(\cL,\cW,\cR)$ is assigned,
where\/ $\cL=\sL\Mono$, \ $\cR=\sR\Epi$, \
$\cL\cap\cW=(\sL\cap\nobreak\sW)\Mono$, \
$\cR\cap\cW=(\sR\cap\nobreak\sW)\Epi$,
and\/ $\cW$ is the class of all morphisms~$w$ decomposable as $w=rl$,
where $l\in(\sL\cap\nobreak\sW)\Mono$ and
$r\in(\sR\cap\nobreak\sW)\Epi$.
\end{thm}

\begin{proof}
 This is a part of~\cite[Theorem~2.2]{Hov} (for a further
generalization, see~\cite[Theorem~3.3]{Gil2}).
 Given a model structure $(\cL,\cW,\cR)$ on $\sE$ such that
$\cL=\sL\Mono$ and $\cR=\sR\Epi$, we have
$\cL\cap\cW=(\sL\cap\nobreak\sW)\Mono$ and
$\cR\cap\cW=(\sR\cap\nobreak\sW)\Epi$ by
Lemma~\ref{trivial-co-fibrations}.
 Hence the pairs of classes of objects $(\sL,\>\sR\cap\nobreak\sW)$
and $(\sL\cap\nobreak\sW,\>\sR)$ are complete cotorsion pairs by
Theorem~\ref{wfs-complete-cotorsion}.
 The class of objects $\sW\subset\sE$ is thick by~\cite[Lemma~4.3]{Hov}.

 Conversely, given a triple of classes of objects $(\sL,\sW,\sR)$
in $\sE$ satisfying~(1\+-3), define the triple of classes of
morphisms $(\cL,\cW,\cR)$ as stated in the theorem.
 Then the pairs of classes $(\cL,\>(\sR\cap\nobreak\sW)\Epi)$
and $((\sL\cap\nobreak\sW\Epi),\>\cR)$ are weak factorization systems
by Theorem~\ref{wfs-complete-cotorsion}.
 The equations $\cL\cap\cW=(\sL\cap\nobreak\sW)\Mono$ and
$\cR\cap\cW=(\sR\cap\nobreak\sW)\Epi$ hold by~\cite[Lemma~5.8]{Hov}.
 The class of morphisms $\cW$ satisfies the two-out-of-three
property by~\cite[Proposition~5.12]{Hov}.
\end{proof}

 In the sequel, we will identify abelian model structures with
the triples of classes of objects $(\sL,\sW,\sR)$ using
Theorem~\ref{abelian-model-structures-thm}, and write simply
``an abelian model structure $(\sL,\sW,\sR)$ on
an abelian category~$\sE$''.

\begin{cor} \label{abelian-cofibrantly-generated}
 An abelian model structure $(\sL,\sW,\sR)$ on a locally presentable
abelian category\/
$\sE$ is cofibrantly generated if and only if both the cotorsion pairs
$(\sL,\>\sR\cap\sW)$ and $(\sL\cap\sW,\>\sR)$ are generated by some
sets of objects.
\end{cor}

\begin{proof}
 Follows immediately from Lemma~\ref{abelian-wfs-cof-gen}.
\end{proof}

\begin{lem} \label{hereditary-model-structures}
 Let $(\sL,\sW,\sR)$ be an abelian model structure on an abelian
category\/~$\sE$.
 Then the following conditions are equivalent:
\begin{enumerate}
\item the class\/ $\sL$ is closed under the kernels of epimorphisms
in\/~$\sE$;
\item the class\/ $\sL\cap\sW$ is closed under the kernels of
epimorphisms in\/~$\sE$;
\item the class\/ $\sR$ is closed under the cokernels of monomorphisms
in\/~$\sE$;
\item the class\/ $\sR\cap\sW$ is closed under the cokernels of
monomorphisms in\/~$\sE$.
\end{enumerate}
\end{lem}

\begin{proof}
 The equivalences (1)\,$\Longleftrightarrow$\,(4) and
(2)\,$\Longleftrightarrow$\,(3) hold by
Lemma~\ref{hereditary-cotorsion-pairs}.
 The implications (1)\,$\Longrightarrow$\,(2) and
(3)\,$\Longrightarrow$\,(4) hold because the class of weakly
trivial objects $\sW$ is thick (so it is closed under both
the kernels of epis and the cokernels of monos).
\end{proof}

 An abelian model structure satisfying the equivalent conditions of
Lemma~\ref{hereditary-model-structures} is said to be \emph{hereditary}.

 A model structure $(\cL,\cW,\cR)$ on a category $\sB$ is called
\emph{projective} if all the objects of $\sB$ are fibrant.
 For an abelian model structure, this means that $\sR=\sB$, or
equivalently, $\sL\cap\sW=\sB_\proj$ is the class of all
projective objects in~$\sB$.
 In other words, an abelian model structure is projective if and only
if the trivial cofibrations are the monomorphisms with projective
cokernel.

 Dually, a model structure $(\cL,\cW,\cR)$ on a category $\sA$ is called
\emph{injective} if all the objects of $\sA$ are cofibrant.
 For an abelian model structure, this means that $\sL=\sA$, or
equivalently, $\sR\cap\sW=\sA_\inj$ is the class of all
injective objects in~$\sA$.
 In other words, an abelian model structure is injective if and only if
the trivial fibrations are the epimorphisms with injective kernel.

 It is clear from the definitions that all projective abelian model
structures and all injective abelian model structures are
hereditary~\cite[Corollary~1.1.12]{Bec}.
 (Indeed, any projective abelian model structure obviously satisfies
the condition~(3) of Lemma~\ref{hereditary-model-structures}, while
any injective abelian model structure obviously satisfies
the condition~(1).)

\begin{lem}[{\cite[Corollary~1.1.9]{Bec}}]
\label{inj-proj-model-structures}
 Let\/ $\sA$ and\/ $\sB$ be abelian categories.
 Then \par
\textup{(a)} a pair of classes of objects $(\sW,\sR)$ in\/ $\sA$
defines an injective abelian model structure $(\sA,\sW,\sR)$ on\/ $\sA$
if and only if\/ $\sA$ has enough injective objects, $(\sW,\sR)$ is
a complete cotorsion pair in\/ $\sA$ with\/ $\sR\cap\sW=\sA_\inj$,
and the class\/ $\sW$ is thick.  \par
\textup{(b)} a pair of classes of objects $(\sL,\sW)$ in\/ $\sB$
defines a projective abelian model structure $(\sL,\sW,\sB)$ on\/ $\sB$
if and only if\/ $\sB$ has enough projective objects, $(\sL,\sW)$ is
a complete cotorsion pair in\/ $\sB$ with\/ $\sL\cap\sW=\sB_\proj$,
and the class\/ $\sW$ is thick.
\end{lem}

\begin{proof}
 The assertions follow from Theorem~\ref{abelian-model-structures-thm}.
 One only needs to observe~\cite{Bec} that $(\sA,\sR\cap\sW)$ is
a complete cotorsion pair in $\sA$ if and only if $\sA$ has enough
injectives and $\sR\cap\sW=\sA_\inj$.
 Similarly, $(\sL\cap\sW,\sB)$ is a complete cotorsion pair in $\sB$
if and only if $\sB$ has enough projectives and $\sL\cap\sW=\sB_\proj$.
\end{proof}

The idea of the following lemma can be traced back at least to~\cite[Theorem~VI.2.1]{BeRe}.

\begin{lem}[{\cite[Lemma~1.1.10]{Bec}}] \label{W-is-thick}
\textup{(a)} Let $(\sW,\sR)$ be a hereditary complete cotorsion pair
in an abelian category\/ $\sA$ such that\/ $\sR\cap\sW=\sA_\inj$.
 Then the class of objects\/ $\sW$ is thick in\/~$\sA$. \par
\textup{(b)} Let $(\sL,\sW)$ be a hereditary complete cotorsion pair
in an abelian category\/ $\sB$ such that\/ $\sL\cap\sW=\sB_\proj$.
 Then the class of objects\/ $\sW$ is thick in\/~$\sB$.
\end{lem}

\begin{proof}
 Part~(a): in any hereditary cotorsion pair $(\sW,\sR)$, the class of
objects $\sW$ is closed under direct summands, extensions, and
the kernels of epimorphisms.
 It remains to show that $\sW$ is closed under the cokernels of
monomorphisms in our assumptions.
 We follow the argument from~\cite{Bec}.
 Let $0\rarrow W\rarrow V\rarrow U\rarrow0$ be a short exact
sequence of objects in $\sA$ with $W$, $V\in\sW$.
 Then $\Ext^2_\sA(U,R)=0$ for all $R\in\sR$, since
$\Ext^1_\sA(W,R)=0=\Ext^2_\sA(V,R)$.
 Let us show that this implies $\Ext^1_\sA(U,R)=0$.
 Choose a special precover sequence $0\rarrow R'\rarrow J\rarrow R
\rarrow0$ for the object $R\in\sA$ with $R'\in\sR$ and $J\in\sW$.
 Then $J\in\sR\cap\sW=\sA_\inj$, since $R'$, $R\in\sR$.
 Hence $\Ext^1_\sA(U,R)\simeq\Ext^2_\sA(U,R')=0$.
 Part~(b) is dual.
\end{proof}

\Section{Categories of Complexes} \label{categories-of-complexes-secn}

 Let $\sE$ be an additive category.
 We denote by $\sC(\sE)$ the category of (unbounded) complexes in~$\sE$.
 The category $\sC(\sE)$ is abelian whenever the category $\sE$~is.
 We denote by $C^\bu\longmapsto C^\bu[1]$ the usual shift functor, where $C^\bu[1]^i=C^{i+1}$ and $d^i_{C^\bu[1]}=-d^{i+1}_{C^\bu}$.

 Furthermore, let $\sK(\sE)$ denote the homotopy category of complexes
in $\sE$, that is, the additive quotient category of $\sC(\sE)$ by
the ideal of morphisms cochain homotopic to zero.
 Then $\sK(\sE)$ is a triangulated category~\cite{Ver0,Ver1}.

 Let $\sE$ be an abelian category.
 A short exact sequence of complexes $0\rarrow A^\bu\rarrow C^\bu
\rarrow B^\bu\rarrow0$ in $\sE$ is said to be \emph{termwise split} if
the short exact sequence $0\rarrow A^i\rarrow C^i\rarrow B^i\rarrow0$
is split in $\sE$ for every $i\in\boZ$.

 The next lemma is well-known (cf.~\cite[Section~1.3]{Bec}).

\begin{lem} \label{ext-1-hom-hot}
 For any two complexes $A^\bu$ and $B^\bu\in\sC(\sE)$, the subgroup
in\/ $\Ext^1_{\sC(\sE)}(B^\bu,A^\bu)$ formed by the termwise split
short exact sequences is naturally isomorphic to the group\/
$\Hom_{\sK(\sE)}(B^\bu,A^\bu[1])$.
 In particular, if either \hbadness=1125\par
\textup{(a)} $A^\bu$ is a complex of injective objects in\/ $\sE$, or
\par
\textup{(b)} $B^\bu$ is a complex of projective objects in\/ $\sE$,
\par\noindent
then\/ $\Ext^1_{\sC(\sE)}(B^\bu,A^\bu)\simeq
\Hom_{\sK(\sE)}(B^\bu,A^\bu[1])$.
\end{lem}

\begin{proof}
 The first assertion can be formulated using the termwise
split exact category structure on $\sC(\sE)$; then it holds for
any additive category~$\sE$.
 The natural map $\Hom_{\sK(\sE)}(B^\bu,A^\bu[1])\rarrow
\Ext^1_{\sC(\sE)}(B^\bu,A^\bu)$ takes the homotopy class of
a morphism $f\:B^\bu\rarrow A^\bu[1]$ to the equivalence class of
the termwise split short exact sequence $0\rarrow A^\bu\rarrow
\cone(f)[-1]\rarrow B^\bu\rarrow0$, where $\cone(f)$ denotes
the cone of a morphism of complexes.
 We leave further (elementary) details to the reader.
\end{proof}

 Denote by $\sE^\sgr=\prod_{i\in\boZ}\sE$ the additive category of
graded objects in an additive category~$\sE$.
 Following~\cite{Pkoszul} and~\cite{Bec}, we denote the functor of forgetting the differential $\sC(\sE)\rarrow\sE^\sgr$ by
$C^\bu\longmapsto C^\bu{}^\sharp$.

 The functor $({-})^\sharp$ has adjoints on both sides; we denote
the left adjoint functor to $({-})^\sharp$ by $G^+\:\sE^\sgr
\rarrow\sC(\sE)$ and the right adjoint one by $G^-\:\sE^\sgr
\rarrow\sC(\sE)$.
 The functors $G^+$ and $G^-$ were originally introduced for
CDG\+modules in~\cite[proof of Theorem~3.6]{Pkoszul}; in the case
of complexes in an additive category, they are much easier to
describe.
 To a graded object $E=(E^i)_{i\in\boZ}\in\sE^\sgr$, the functor
$G^+$ assigns the contractible complex with the terms $G^+(E)^i=
E^{i-1}\oplus E^i$ and the differential $d^i\:G^+(E)^i\rarrow
G^+(E)^{i+1}$ given by the $2\times 2$ matrix of morphisms whose
only nonzero entry is the identity morphism $E^i\rarrow E^i$.
 The complex $G^-(E)$ has the terms $G^-(E)^i=E^i\oplus E^{i+1}$;
otherwise it is constructed similarly to $G^+(E)$.
 So the two functors only differ by the shift: one has $G^-=G^+[1]$.

 Given an abelian category $\sB$, we denote by $\sB_\proj^\sgr$
the category of projective objects in the abelian category $\sB^\sgr$
or, which is the same, the category of graded objects in the additive
category $\sB_\proj$.
 Furthermore, we denote by $\sC(\sB_\proj)$ the category of complexes
of projective objects in $\sB$ and by $\sC(\sB)_\proj$ the category
of projective objects in the abelian category $\sC(\sB)$.
 Similar notation is used for injective objects/graded objects/complexes
in an abelian category~$\sA$.

 The abelian category $\sC(\sE)$ has enough projective (respectively,
injective) objects whenever an abelian category $\sE$ has.
 The next lemma describes such projective or injective complexes.

\begin{lem} \label{inj-proj-complexes}
\textup{(a)} For any abelian category\/ $\sA$ with enough injective
objects, a complex $J^\bu\in\sC(\sA)$ is an injective object of\/
$\sC(\sA)$ if and only if $J^\bu$ is contractible and all its
components $J^i$ are injective objects of\/ $\sC(\sA)$, and if and only
if $J^\bu$ has the form $G^-(I)$ for some collection of injective
objects $(I^i)_{i\in\boZ}$ in\/~$\sA$.
 Symbolically, $J^\bu\in\sC(\sA)_\inj$ if and only if $J^\bu$ is
contractible and $J^\bu\in\sC(\sA_\inj)$, and if and only if
$J^\bu\in G^-(\sA^\sgr_\inj)$. \par
\textup{(b)} For any abelian category\/ $\sB$ with enough projective
objects, a complex $P^\bu\in\sC(\sB)$ is a projective object of\/
$\sC(\sB)$ if and only if $P^\bu$ is contractible and all its
components $P^i$ are projective objects of\/ $\sC(\sB)$, and if and only
if $P^\bu$ has the form $G^+(Q)$ for some collection of projective
objects $(Q^i)_{i\in\boZ}$ in\/~$\sB$.
 Symbolically, $P^\bu\in\sC(\sB)_\proj$ if and only if $P^\bu$ is
contractible and $P^\bu\in\sC(\sB_\proj)$, and if and only if
$P^\bu\in G^+(\sB^\sgr_\proj)$. \par
\end{lem}

\begin{proof}
 This observation is known, at least,
since~\cite[Theorem~IV.3.2]{EM}; see~\cite[Lemma~1.3.3]{Bec}
for a CDG\+module version.
 We leave the details to the reader.
\end{proof}

\begin{lem} \label{intersection-is-inj-proj-complexes}
\textup{(a)}
 Let\/ $\sA$ be an abelian category with enough injectives,
and let $(\sW,\sR)$ be a cotorsion pair in\/ $\sC(\sA)$ such that\/
$\sR\subset\sC(\sA_\inj)$.
 Assume that the cotorsion pair $(\sW,\sR)$ is invariant under
the shift: $\sW=\sW[1]$, or equivalently, $\sR=\sR[1]$.
 Then\/ $\sR\cap\sW=\sC(\sA)_\inj$.
If, moreover, $(\sW,\sR)$ is complete, then $\sW$ is thick in $\sC(\sA)$. \par
\textup{(b)}
 Let\/ $\sB$ be an abelian category with enough projectives,
and let $(\sL,\sW)$ be a cotorsion pair in\/ $\sC(\sB)$ such that\/
$\sL\subset\sC(\sB_\proj)$.
 Assume that the cotorsion pair $(\sL,\sW)$ is invariant under
the shift: $\sW=\sW[1]$, or equivalently, $\sL=\sL[1]$.
 Then\/ $\sL\cap\sW=\sC(\sB)_\proj$.
If, moreover, $(\sL,\sW)$ is complete, then $\sW$ is thick in $\sC(\sB)$.
\end{lem}

\begin{proof}
 This is our version of~\cite[Lemma~1.3.4]{Bec}, and we follow
the argument in~\cite{Bec}.
 Part~(a): the inclusion $\sC(\sA)_\inj\subset\sW^{\perp_1}=\sR$
is obvious.
 The inclusion $\sC(\sA)_\inj\subset{}^{\perp_1}\sC(\sA_\inj)$
holds, because $\sC(\sA)_\inj$ is the category of projective-injective
objects of the Frobenius exact category $\sC(\sA_\inj)$.
 Specifically, for any $J^\bu\in\sC(\sA)_\inj$ and
$I^\bu\in\sC(\sA_\inj)$ we have, by Lemmas~\ref{ext-1-hom-hot}(a)
and~\ref{inj-proj-complexes}(a), $\Ext^1_{\sC(\sA)}(J^\bu,I^\bu)
\simeq\Hom_{\sK(\sA)}(J^\bu,I^\bu[1])=0$ since the complex $J^\bu$
is contractible.
 Hence $\sC(\sA)_\inj\subset{}^{\perp_1}\sR=\sW$.
 To prove the inclusion $\sR\cap\sW\subset\sC(\sA)_\inj$, it suffices
to show that every complex $J^\bu\in\sR\cap\sW$ is contractible
(in view of Lemma~\ref{inj-proj-complexes}(a), as we already know
that $\sR\cap\sW\subset\sR\subset\sC(\sA_\inj)$).
 Indeed, by Lemma~\ref{ext-1-hom-hot}(a),
\,$\Hom_{\sK(\sA)}(J^\bu,J^\bu)=\Ext^1_{\sC(\sA)}(J^\bu,J^\bu[-1])=0$.

Regarding the moreover clause, the class $\sR$ is closed
under cosyzygies in $\sC(A)$
(each $J^\bu\in\sR$ admits a short exact sequence
$0\rarrow J^\bu\rarrow G^-({J^\bu}^\sharp)\rarrow J^\bu[1]\rarrow 0$
and $G^-({J^\bu}^\sharp)\in\sC(\sA)_\inj$ by the previous lemma).
It follows by a standard dimension shifting argument that the cotorsion
pair $(\sW,\sR)$ is hereditary. Now we just apply Lemma~\ref{W-is-thick}(a).

 Part~(b) is dual.
\end{proof}

\Section{\texorpdfstring{Locally Presentable Abelian Categories with~a~Projective~Generator}{Locally Presentable Abelian Categories with a Projective Generator}}
\label{loc-pres-secn}

 In this section we study the conventional derived category
$\sD(\sB)$ of a locally presentable abelian category $\sB$ with
enough projective objects.
 We start with the following lemma describing the class of abelian
categories we are interested in.

\begin{lem}
 A locally presentable abelian category\/ $\sB$ has enough projective
objects if and only if it has a single projective generator.
\end{lem}

\begin{proof}
 The ``if'' implication holds, since any abelian category with
coproducts and a projective generator has enough projective objects.
 To prove the ``only if'', observe that any abelian category with
enough projective objects and a set of generators has a set of
projective generators, and any abelian category with coproducts and
a set of projective generators has a single projective generator.
\end{proof}

\begin{lem} \label{fil-proj}
 For any abelian category\/ $\sB$ with coproducts and a set of
projective generators $\{P_\alpha\}$, one has\/
$\Fil(\{P_\alpha\})^\oplus=\Add(\{P_\alpha\})=\sB_\proj$.
\end{lem}

\begin{proof}
 The assertion is straightforward.
 The least obvious part is the inclusion $\Fil(\sB_\proj)\subset
\sB_\proj$, which can be also deduced from
Proposition~\ref{eklof-lemma-prop}.
\end{proof}

 The following lemma is quite general.

\begin{lem} \label{complexes-locally-presentable}
 For any locally presentable additive category\/ $\sE$ and any small
additive category $\sX$, the additive category\/ $\AdF(\sX,\sE)$
of additive functors\/ $\sX\rarrow\sE$ is locally presentable.
 In particular, the category of complexes\/ $\sC(\sE)$ is locally
presentable.
\end{lem}

\begin{proof}
 Let $X$ be the set of objects of the small category~$\sX$.
 Consider the product $\sE^X$ of $X$ copies of~$\sE$.
 Then the forgetful functor $\AdF(\sX,\sE)\rarrow\sE^X$ is monadic,
i.~e., it identifies $\AdF(\sX,\sE)$ with the Eilenberg--Moore
category of a monad on~$\sE^X$.
 This assertion holds for any cocomplete additive category~$\sE$.
 Moreover, the underlying functor of the relevant monad on $\sE^X$
preserves all colimits.
 Now if $\sE$ is locally presentable, then $\sE^X$ is locally
presentable by~\cite[Proposition~2.67]{AR}, and it follows that
$\AdF(\sX,\sE)$ is locally presentable by~\cite[Theorem and
Remark~2.78]{AR}.
It remains to mention that the category of complexes $\sC(\sE)$ is
isomorphic to $\AdF(\sX,\sE)$ for a suitable choice of $\sX$,
see~\cite[Proposition 4.5]{SS}.

 Here is also a direct elementary argument showing that $\sC(\sE)$ is
locally presentable.
 Clearly, the category $\sC(\sE)$ is cocomplete whenever
the category $\sE$ is.
 Given an object $E\in\sE$, consider $E$ as a one-term complex
concentrated in degree~$0$.
 Then the functor $E\longmapsto G^+(E)[-i]\:\sE\rarrow\sC(\sE)$
is left adjoint to the functor $\sC(\sE)\rarrow\sE$ taking
a complex $C^\bu$ to its degree\+$i$ term~$C^i$.
 Assume that the category $\sE$ is locally $\lambda$\+presentable for
some regular cardinal~$\lambda$, and let $\sS$ be a (strongly)
generating set of $\lambda$\+presentable objects in~$\sE$.
 Then the objects $G^+(S)[-i]$, \,$S\in\sS$, \,$i\in\boZ$ form
a (strongly) generating set of $\lambda$\+presentable objects in
$\sC(\sE)$.
 Hence the category $\sC(\sE)$ is locally $\lambda$\+presentable
by~\cite[Theorem~1.20]{AR}.
\end{proof}

 Let $\sB$ be an abelian category.
 A complex $P^\bu\in\sK(\sB)$ is said to be \emph{homotopy projective}
if $\Hom_{\sK(\sB)}(P^\bu,X^\bu)=0$ for any acyclic complex $X^\bu$
in~$\sB$.
 We denote the full subcategory of homotopy projective complexes by
$\sK(\sB)_\hpr\subset\sK(\sB)$.
 Furthermore, let us denote by $\sK(\sB_\proj)_\hpr=\sK(\sB_\proj)\cap
\sK(\sB)_\hpr\subset\sK(\sB)$ the full subcategory of \emph{homotopy
projective complexes of projective objects} in the homotopy category
$\sK(\sB)$.
 Clearly, both $\sK(\sB)_\hpr$ and $\sK(\sB_\proj)_\hpr$ are
triangulated subcategories in $\sK(\sB)$.

\begin{rem} \label{homotopy-adjusted-terminology}
 Spaltenstein introduced what he called ``K\+projective'',
``K\+injective'', and ``K\+flat'' resolutions in his paper~\cite{Spa}
(he attributed the idea to J.~Bernstein).
 It was not explained in~\cite{Spa} what the letter ``K'' was supposed
to stand for; but the natural (and widespread) guess was that it referred
to the common notation for the homotopy category of complexes (which we
use in the present paper as well).
 Hence the terminology of ``homotopically adjusted'' or ``homotopy
adjusted'', or for brevity ``h\+adjusted'' complexes emerged
(where ``adjusted'' is our generic name for projective, injective,
or flat).

 For example, the ``homotopically flat'' terminology was used
in~\cite{Dr} (see~\cite[Section~3.3]{Dr}, where one can find the same
explanation of its origins as we suggest above), while
the ``homotopically injective'' and ``homotopically projective''
terminology can be found in~\cite[Definition~14.1.4]{KS}.
 The ``h\+projective'' and ``h\+injective'' terminology was used
in~\cite{ELO} (see~\cite[Section~3.1]{ELO}), while
the ``homotopically projective'' terminology can be found (in
a different but somewhat related context) in~\cite[Section~4.1]{KLN}.
 The ``homotopy injective'' terminology was used in~\cite{KK}.
 In the first-named author's work, the ``homotopy
projective/injective/flat'' terminology was adopted
in~\cite[Section~2.1]{PP2}, \cite[Sections~0.2 and~2.3]{Pweak},
\cite[Sections~1.1--1.2]{Pmc}, and subsequent publications
(see also~\cite[Remark~0.1.3]{Psemi}).

 What we call ``homotopy projective complexes of projective objects''
are otherwise known in the literature as
``DG\+projective''~\cite{AF,GR,Gil0} or
``semi-projective''~\cite{CF,Gil3} complexes.
 Similarly people speak of ``DG\+injective'' and ``DG\+flat''
complexes~\cite{AF,GR,Gil4} or ``semi-injective'' and ``semi-flat''
complexes~\cite{CF,Gil3,NT}.
\end{rem}

\begin{prop} \label{orthogonal-fully-faithful}
 For any abelian category\/ $\sB$, the composition of the fully
faithful inclusion\/ $\sK(\sB)_\hpr\rarrow\sK(\sB)$ with the Verdier
quotient functor\/ $\sK(\sB)\rarrow\sD(\sB)$ is a fully faithful
functor\/ $\sK(\sB)_\hpr\rarrow\sD(\sB)$.
 Hence we have a pair of fully faithful triangulated functors\/
$\sK(\sB_\proj)_\hpr\rarrow\sK(\sB)_\hpr\rarrow\sK(\sB)$.
\end{prop}

\begin{proof}
 More generally, for any triangulated category $\sH$ and full
triangulated subcategories $\sP$ and $\sX\subset\sH$ such that
$\Hom_\sH(P,X)=0$ for all $P\in\sP$ and $X\in\sX$, the composition
of the inclusion $\sP\rarrow\sH$ with the Verdier quotient functor
$\sH\rarrow\sH/\sX$ is fully faithful.
 In the situation at hand, take $\sH=\sK(\sB)$ and $\sP=\sK(\sB)_\hpr$,
and let $\sX=\sK(\sB)_\ac\subset\sK(\sB)$ be the full subcategory of
acyclic complexes.
\end{proof}

 We will say that an abelian category $\sB$ has \emph{enough homotopy
projective complexes} if the functor $\sK(\sB)_\hpr\rarrow\sD(\sB)$
is a triangulated equivalence.
 Moreover, we will say that $\sB$ has \emph{enough homotopy projective
complexes of projective objects} if the functor $\sK(\sB_\proj)_\hpr
\rarrow\sD(\sB)$ is a triangulated equivalence.

 In the latter case, clearly, the inclusion $\sK(\sB_\proj)_\hpr
\rarrow\sK(\sB)_\hpr$ is a triangulated equivalence, too.
 So, in an abelian category with enough homotopy projective complexes
of projective objects, every homotopy projective complex is homotopy
equivalent to a homotopy projective complex of projective objects.

 Introduce the notation $\sC(\sB)_\ac\subset\sC(\sB)$ for the full
subcategory of acyclic complexes in the category of complexes
$\sC(\sB)$ and the notation $\sC(\sB_\proj)_\hpr\subset\sC(\sB)$ for
the full subcategory of homotopy projective complexes of projective
objects in $\sC(\sB)$.
That is, these categories are the full preimages of
$\sK(\sB)_\ac$ and $\sK(\sB_\proj)_\hpr$, respectively,
under the natural functor $\sC(\sB)\rarrow\sK(\sB)$.

\begin{thm} \label{loc-pres-cotorsion-pair}
 Let\/ $\sB$ be a locally presentable abelian category with enough
projective objects.
 Then the pair of classes of objects\/ $\sC(\sB_\proj)_\hpr$ and\/
$\sC(\sB)_\ac$ is a hereditary complete cotorsion pair in the abelian
category\/ $\sC(\sB)$.
\end{thm}

\begin{proof}
 In the special setting when $\sB$ is a Grothendieck abelian category
with enough projective objects, this theorem becomes a particular case
of~\cite[Theorem~4.6]{Gil3}.

 Let $P$ be a projective generator of~$\sB$.
 The claim is that the cotorsion pair we are interested in
is generated by the set of objects $\sS=\{P[i]\}_{i\in\boZ}\subset
\sC(\sB)$ (where $P$ is viewed as a one-term complex concentrated
in degree~$0$).

 Indeed, for any complex $C^\bu\in\sC(\sB)$ we have, by
Lemma~\ref{ext-1-hom-hot}(b), \,$\Ext^1_{\sC(\sB)}(P[i],C^\bu)\simeq
\Hom_{\sK(\sB)}(P[i],C^\bu[1])\simeq\Hom_\sB(P,H^{1-i}(C^\bu))$.
 The latter abelian group vanishes if and only if $H^{1-i}(C^\bu)=0$.
 Hence $\sS^{\perp_1}=\sC(\sB)_\ac\subset\sC(\sB)$.

 In the pair of adjoint functors $({-})^\sharp\:\sC(\sB)\rarrow
\sB^\sgr$ and $G^-\:\sB^\sgr\rarrow\sC(\sB)$, both the functors
are exact.
 Hence for any complex $C^\bu\in\sC(\sB)$ and any graded object
$B\in\sB^\sgr$ we have $\Ext^n_{\sC(\sB)}(C^\bu,G^-(B))\simeq
\Ext^n_{\sB^\sgr}(C^\bu{}^\sharp,B)$ for all $n\ge0$.
 In particular, this isomorphism holds for $n=1$.
 Since the complex $G^-(B)$ is acyclic, we have shown that
${}^{\perp_1}(\sC(\sB)_\ac)\subset\sC(\sB_\proj)$.
 Now for a complex $Q^\bu\in\sC(\sB_\proj)$ and any complex $X^\bu
\in\sC(\sB)$ we have $\Ext^1_{\sC(\sB)}(Q^\bu,X^\bu)\simeq
\Hom_{\sK(\sB)}(Q^\bu,X^\bu[1])$ by Lemma~\ref{ext-1-hom-hot}(b).
 Thus ${}^{\perp_1}(\sC(\sB)_\ac)=\sC(\sB_\proj)_\hpr\subset\sC(\sB)$.
 So our pair of classes of objects is indeed the cotorsion pair
generated by the set $\sS\subset\sC(\sB)$.

 Obviously, any complex $C^\bu$ is a subcomplex of an acyclic (or even
contractible) complex via the adjunction morphism
$C^\bu\rarrow G^-({C^\bu}^\sharp)$; so the class $\sC(\sB)_\ac$ is cogenerating
in $\sC(\sB)$.
 It is also easy to present any complex in $\sB$ as a quotient
complex of a contractible complex of projective objects (since there
are enough projectives in~$\sB$).
 All contractible complexes are obviously homotopy projective.
 Hence the class $\sC(\sB_\proj)_\hpr$ is generating in $\sC(\sB)$.

 By Lemma~\ref{complexes-locally-presentable}, the category $\sC(\sB)$
is locally presentable.
 Applying Theorem~\ref{cotorsion-pair-generated-by-set-complete}, we
conclude that our cotorsion pair is complete.
 The cotorsion pair is hereditary, since the class $\sC(\sB)_\ac$
is closed under the cokernels of monomorphisms in $\sC(\sB)$.
\end{proof}

\begin{cor} \label{enough-homotopy-projectives}
 Any locally presentable abelian category with a projective generator
has enough homotopy projective complexes of projective objects.
In other words, the natural functor
$\sK(\sB_\proj)_\hpr\rarrow\sD(\sB)$ is a triangulated equivalence.
\end{cor}

\begin{proof}
 Let $\sB$ be a locally presentable abelian category with enough
projective objects.
 Given a complex $C^\bu\in\sC(\sB)$, we need to find a homotopy
projective complex of projective objects $Q^\bu$ together with
a quasi-isomorphism $Q^\bu\rarrow C^\bu$ of complexes in~$\sB$.
 For this purpose, consider a special precover short exact sequence
$0\rarrow X^\bu\rarrow Q^\bu\rarrow C^\bu\rarrow0$ in the complete
cotorsion pair of Theorem~\ref{loc-pres-cotorsion-pair}.
 So we have $X^\bu\in\sC(\sB)_\ac$ and $Q^\bu\in\sC(\sB_\proj)_\hpr$.
 Since the complex $X^\bu$ is acyclic, it follows that the epimorphism
of complexes $Q^\bu\rarrow C^\bu$ is a quasi-isomorphism.
\end{proof}

The proof of Theorem~\ref{loc-pres-cotorsion-pair} also provides us
with the following description of the class of homotopy projective complexes
of projectives.

\begin{cor}
Let\/ $\sB$ be a locally presentable abelian category
with enough projective objects.
If we choose a projective generator $P$ and denote\/
$\sS=\{P[i]\}_{i\in\boZ}\subset \sC(\sB)$, then\/
$\sC(\sB_\proj)_\hpr=\Fil(\sS)^\oplus\subset\sC(\sB)$.
\end{cor}

\begin{proof}
 The special case when $\sB$ is a Grothendieck category with enough
projectives is covered by~\cite[Theorem~4.6]{Gil3}.
 Quite generally, we know that $\sC(\sB_\proj)_\hpr$ is the left-hand
side of the cotorsion pair in $\sC(\sB)$ generated by~$\sS$.
Moreover, $\Fil(\sS)$ is generating since the objects
$G^+(P)[i]$, \,$i\in\boZ$, form a set of
projective generators of $\sC(\sB)$ and, obviously,
$G^+(P)[i]\in\Fil(\sS)$.
Using Lemma~\ref{fil-proj}, one can see that the class $\Fil(\sS)$
is generating (notice that $\Fil(\Fil(\sS))= \Fil(\sS)$ by~\cite[Lemma~4.6(d)]{PR}), and we just apply Theorem~\ref{cotorsion-pair-generated-by-set-left-class}.
\end{proof}

\begin{thm} \label{loc-pres-model-structure}
 Let\/ $\sB$ be a locally presentable abelian category with enough
projective objects. 
 Then the triple of classes of objects\/ $\sL=\sC(\sB_\proj)_\hpr$,
\ $\sW=\sC(\sB)_\ac$, and $\sR=\sC(\sB)$ is a cofibrantly generated
hereditary abelian model structure on the abelian category of
complexes\/ $\sC(\sB)$.
\end{thm}

\begin{proof}
 Once again, in the special setting when $\sB$ is a Grothendieck abelian
category with enough projective objects, this theorem becomes
a particular case of~\cite[Corollary~4.7]{Gil3}.

 The pair of classes $(\sL,\sW)$ is a complete cotorsion pair in
$\sC(\sB)$ by Theorem~\ref{loc-pres-cotorsion-pair}.
 The abelian category of complexes $\sC(\sB)$ has enough projective
objects, since the abelian category $\sB$ has.
 According to Lemma~\ref{intersection-is-inj-proj-complexes}(b),
it follows that $\sL\cap\sW=\sC(\sB)_\proj$.
 The class of acyclic complexes $\sW$ is obviously thick in $\sC(\sB)$
(see also Lemma~\ref{W-is-thick}(b)
or~\ref{intersection-is-inj-proj-complexes}(b)).
 By Lemma~\ref{inj-proj-model-structures}(b) (applied to
the category $\sC(\sB)$), the triple $(\sL,\sW,\sR)$ is
a projective abelian model structure on $\sC(\sB)$.

 As explained in Section~\ref{abelian-model-secn}, any projective
abelian model structure is hereditary.
 Finally, it is clear from the proof of
Theorem~\ref{loc-pres-cotorsion-pair} that the cotorsion pair
$(\sL,\sW)$ in $\sC(\sB)$ is generated by a set of objects.
 The cotorsion pair $(\sC(\sB)_\proj,\,\sC(\sB))$ is generated
by the empty set of objects (or by the single zero object) in
$\sC(\sB)$.
 According to Corollary~\ref{abelian-cofibrantly-generated}, this
means that our abelian model structure is cofibrantly generated.
\end{proof}

 The abelian model structure $(\sL,\sW,\sR)$ defined in
Theorem~\ref{loc-pres-model-structure} is called the \emph{projective
derived model structure} on the abelian category of complexes
$\sC(\sB)$.

\begin{lem} \label{loc-pres-weak-equivalences}
 For any locally presentable abelian category\/ $\sB$ with enough
projective objects, the class\/ $\cW$ of all weak equivalences in
the projective derived model structure on the abelian category\/
$\sC(\sB)$ coincides with the class of all quasi-isomorphisms of
complexes in\/~$\sB$.
\end{lem}

\begin{proof}
 According to~\cite[Lemma~5.8]{Hov}, a monomorphism in an abelian model
category is a weak equivalence if and only if its cokernel is weakly
trivial.
 Dually, an epimorphism is a weak equivalence if and only if its
kernel is weakly trivial.

 In the situation at hand, the class of weakly trivial objects $\sW$
is the class of all acyclic complexes.
 This implies the assertion of the lemma for all monomorphisms and
epimorphisms of complexes.
 It remains to recall that any morphism in $\sC(\sB)$ is the composition
of, say, a trivial cofibration (which is a monomorphism, a weak
equivalence, and a quasi-isomorphism) and a fibration (which is
an epimorphism), and that both the classes of weak equivalences and
quasi-isomorphisms satisfy the two-out-of-three property.
 Alternatively, one can use the fact that any morphism is
the composition of a cofibration and a trivial fibration.
\end{proof}

 The following corollary presumes existence of (set-indexed) coproducts
in the derived category $\sD(\sB)$.
 Notice that such coproducts can be simply constructed as
the coproducts in the homotopy category $\sK(\sB)_\hpr$ or
$\sK(\sB_\proj)_\hpr$, which is equivalent to $\sD(\sB)$ by
Corollary~\ref{enough-homotopy-projectives}.
 The full subcategory $\sK(\sB)_\hpr$ is closed under coproducts
in $\sK(\sB)$, so the coproducts in $\sK(\sB)_\hpr$ (just as
in $\sK(\sB)$) can be computed as the termwise coproducts of complexes.

\begin{cor} \label{loc-pres-well-generated}
 For any locally presentable abelian category\/ $\sB$ with enough
projective objects, the (unbounded) derived category\/ $\sD(\sB)$ is
a well-generated triangulated category (in the sense of
the book~\cite{Neem-book} and the paper~\cite{Kra0}).
\end{cor}

\begin{proof}
 This property is well-known for Grothendieck categories (see
references in Corollary~\ref{grothendieck-well-generated} below).
 For Grothendieck categories with enough projectives, it becomes
also a particular case of~\cite[Corollary~4.7]{Gil3}.

 Quite generally, following~\cite[Corollary~1.1.15 and the preceding
discussion]{Bec}, any hereditary abelian model category is stable
(so its homotopy category is triangulated).
  Hence, in the situation at hand, the derived category $\sD(\sB)$
can be equivalently defined as the homotopy category
$\sC(\sB)[\cW^{-1}]$ of the stable combinatorial model category
$\sC(\sB)$ with the projective derived model category structure.
 The conclusion follows by~\cite[Proposition~6.10]{Ros}
or~\cite[Theorems~3.1 and~3.9]{CR}.
\end{proof}

 The next result appeared in the context of DG\+contramodules over
a DG\+coalgebra (over a field) in~\cite[Section~5.5]{Pkoszul};
see~\cite[Theorem~1.1(b)]{Pmc}.

\begin{thm} \label{loc-pres-generated-by-P-as-localizing}
 Let $\sB$ be a locally presentable abelian category with a projective
generator~$P$.
 Then the category\/ $\sD(\sB)$ is generated, as a triangulated
category with coproducts, by the single object $P$ (viewed as
a one-term complex concentrated in degree~$0$).
 In other words, the full subcategory of homotopy projective
complexes\/ $\sK(\sB)_\hpr\subset\sK(\sB)$ is the minimal strictly
full triangulated subcategory in\/ $\sK(\sB)$ containing the object $P$
and closed under coproducts.
\end{thm}

\begin{proof}
 The key observation is that $\Hom_{\sD(\sB)}(P,X^\bu[i])=0$ for all
$i\in\boZ$ implies $X^\bu=0$ for a given object $X^\bu\in\sD(\sB)$,
since for any complex $X^\bu$ one has a natural isomorphism of abelian
groups $\Hom_{\sD(\sB)}(P,X^\bu[i])\simeq\Hom_\sB(P,H^i(X^\bu))$.

 By Corollary~\ref{loc-pres-well-generated}, the triangulated category
$\sD(\sB)$ is well-generated.
 Denote by $\sD'\subset\sD(\sB)$ the minimal full triangulated
subcategory in $\sD(\sB)$ containing $P$ and closed under coproducts.
 By~\cite[Theorem~7.2.1(2)]{Kra2}, the localizing subcategory generated
by any set of objects in a well-generated triangulated category is
also well-generated; so $\sD'$ is well-generated.
 Any well-generated triangulated category is perfectly generated by
definition; so the Brown representability
theorem~\cite[Theorem~5.1.1]{Kra2} is applicable, and in particular
the inclusion functor $\sD'\rarrow\sD(\sB)$ has a right adjoint.

 The right adjoint functor to a fully faithful triangulated functor
is a Verdier quotient functor.
 The kernel of our functor $\sD(\sB)\rarrow\sD'$ consists of
complexes $X^\bu$ satisfying $\Hom_{\sD(\sB)}(P[i],X^\bu)=0$ for all
$i\in\boZ$, since $P[i]\in\sD'$.
 Thus all such objects $X^\bu\in\sD(\sB)$ vanish, hence the functor
$\sD(\sB)\rarrow\sD'$ is a triangulated equivalence; and it follows
that the inclusion $\sD'\rarrow\sD(\sB)$ is a triangulated equivalence,
too; the latter means that $\sD'=\sD(\sB)$, as desired.
\end{proof}

\begin{rem}
 An object $S$ in a triangulated category $\sD$ is said to be
a \emph{weak generator} if, for any object $X\in\sD$, one has $X=0$
whenever $\Hom_\sD(S,X[i])=0$ for every $i\in\boZ$.
 Being a weak generator is a weaker property than the one described
in Theorem~\ref{loc-pres-generated-by-P-as-localizing}.
 In particular, for any abelian category $\sB$ with a projective
generator $P$, one can immediately see that $P$ is a weak generator
of the derived category $\sD(\sB)$.
 A generalization of this observation to (not necessarily projective)
generators of Grothendieck abelian categories and certain related
triangulated categories can be found
in~\cite[Remark~1 in Section~4]{Gil3}.
\end{rem}

\Section{Contraderived Model Structure}  \label{contraderived-secn}

 In this section we consider contraderived categories \emph{in
the sense of Becker}~\cite{Bec}.
 In well-behaved cases, this means simply the homotopy category of
complexes of projective objects (which was studied first by
J\o rgensen~\cite{Jor} in the case of module categories); see
the discussion in the introduction.

 The contraderived category in the sense of Becker needs to be
distinguished from the contraderived category in the sense of
the books and papers~\cite{Psemi,Pkoszul,PP2,Pweak,PS2,Pps}.
 The two definitions of a contraderived category are known to be
equivalent under certain assumptions~\cite[Theorem~3.8]{Pkoszul},
but it is still an open question whether they are equivalent for
the category of modules over an arbitrary associative ring
(see~\cite[Example~2.6(3)]{Pps} and Remark~\ref{coderived-history}
below for a discussion).

 Let $\sB$ be an abelian category.
 A complex $X^\bu\in\sK(\sB)$ is said to be \emph{contraacyclic}
(in the sense of Becker) if $\Hom_{\sK(\sB)}(P^\bu,X^\bu)=0$ for
any complex of projective objects $P^\bu\in\sK(\sB_\proj)$.
 We denote the full subcategory of contraacyclic complexes
by $\sK(\sB)_\ac^\ctr\subset\sK(\sB)$.
 Clearly, $\sK(\sB)_\ac^\ctr$ is a triangulated (and even thick)
subcategory in the homotopy category $\sK(\sB)$.
 The quotient category $\sD^\ctr(\sB)=\sK(\sB)/\sK(\sB)_\ac^\ctr$
is called the \emph{contraderived category of\/~$\sB$} (in the sense
of Becker).

\begin{lem} \label{contraacyclic-lemma}
\textup{(a)} For any short exact sequence\/ $0\rarrow K^\bu\rarrow
L^\bu\rarrow M^\bu\rarrow0$ of complexes in\/ $\sB$, the total complex\/
$\Tot(K^\bu\to L^\bu\to M^\bu)$ of the bicomplex with three rows
$K^\bu\rarrow L^\bu\rarrow M^\bu$ belongs to\/ $\sK(\sB)_\ac^\ctr$. \par
\textup{(b)} The full subcategory of contraacyclic complexes\/
$\sK(\sB)_\ac^\ctr$ is closed under products in the homotopy category\/
$\sK(\sB)$.
\end{lem}

\begin{proof}
 This is our version of~\cite[Theorem~3.5(b)]{Pkoszul}.
 Part~(b) follows immediately from the definitions.
 To prove part~(a), let us introduce the following notation.

 For any two complexes $C^\bu$ and $D^\bu$ in an additive category
$\sE$, let $\Hom_\sE(C^\bu,D^\bu)$ denote the complex of morphisms
from $C^\bu$ to~$D^\bu$.
 So the degree~$n$ component $\Hom_\sE^n(C^\bu,D^\bu)$ of the complex
$\Hom_\sE(C^\bu,D^\bu)$ is the group $\Hom_{\sE^\sgr}(C^\bu{}^\sharp,
D^\bu{}^\sharp[n])$ of morphisms $C^\bu{}^\sharp\rarrow
D^\bu{}^\sharp[n]$ in the category $\sE^\sgr$ of
graded objects in~$\sE$.
 The group $\Hom_{\sC(\sE)}(C^\bu,D^\bu)$ can be computed as
the kernel of the differential $\Hom_\sE^0(C^\bu,D^\bu)\rarrow
\Hom_\sE^1(C^\bu,D^\bu)$, while the group
$\Hom_{\sK(\sE)}(C^\bu,D^\bu)$ is the degree~$0$ cohomology group
$H^0(\Hom_\sE(C^\bu,D^\bu))$ of the complex $\Hom_\sE(C^\bu,D^\bu)$.
 
 Now in the situation at hand, for any complex of projective objects
$P^\bu\in\sK(\sB_\proj)$ and any short exact sequence of complexes
$0\rarrow K^\bu\rarrow L^\bu\rarrow M^\bu\rarrow0$ in $\sB$ we have
a short exact sequence of complexes of abelian groups
$0\rarrow\Hom_\sB(P^\bu,K^\bu)\rarrow\Hom_\sB(P^\bu,L^\bu)\rarrow
\Hom_\sB(P^\bu,M^\bu)\rarrow0$.
 The complex $\Hom_\sB(P^\bu,\>\Tot(K^\bu\to L^\bu\to M^\bu))$
can be computed as the total complex of the bicomplex of abelian
groups $\Hom_\sB(P^\bu,K^\bu)\rarrow\Hom_\sB(P^\bu,L^\bu)\rarrow
\Hom_\sB(P^\bu,M^\bu)$.
 It remains to observe that the totalization of any short exact
sequence of abelian groups is an acyclic complex.
\end{proof}

\begin{prop} \label{graded-projective-deconstructible}
 Let\/ $\sB$ be a locally presentable abelian category with enough
projective objects.
 Then there exists a set of complexes of projective objects\/
$\sS\subset\sC(\sB_\proj)$ such that the class of all complexes
of projective objects\/ $\sC(\sB_\proj)\subset\sC(\sB)$ is the class
of all direct summands of complexes filtered by\/ $\sS$, that is\/
$\sC(\sB_\proj)=\Fil(\sS)^\oplus$.
\end{prop}

\begin{proof}[First proof]
 The stronger statement that there exists a set $\sS'\subset
\sC(\sB_\proj)$ such that $\sC(\sB_\proj)=\Fil(\sS')$ is provable
using the results of~\cite[Section~A.1.5]{Lur}, \cite[Section~3]{MR},
and~\cite[Lemma~4.6]{PR}.
 Without going into details, let us describe the overall logic of
this argument, generalizing~\cite[Remark~3.5]{MR}, while a direct
and more elementary argument is presented in the second proof below.
 The idea is to assign to every class of objects $\sL$ in an abelian
category a related class of morphisms, namely the class of all
$\sL$\+monomorphisms $\sL\Mono$, and use the known results about
cofibrantly generated weak factorization systems in locally
presentable categories.

 We will say that a class of objects $\sL$ in an abelian category
$\sE$ is \emph{transmonic} if any transfinite composition of
$\sL$\+monomorphisms is a monomorphism.
 Then it is claimed that, for any transmonic set of objects $\sS$
in a locally presentable abelian category $\sE$, there exists
a set of objects $\sS'\subset\sE$ such that
$\Fil(\sS')=\Fil(\sS)^\oplus$.
 This essentially follows from~\cite[Proposition~A.1.5.12]{Lur}
in view of Proposition~\ref{fil-cell}.
 Furthermore, a class of objects $\sL\subset\sE$ is said to be
\emph{deconstructible} if there exists a set $\sS\subset\sE$ such that
$\sL=\Fil(\sS)$.
In other words, if $\sL$ is transmonic and deconstructible in $\sE$,
so is $\sL^\oplus$.

 For any small additive category $\sX$ and any transmonic
deconstructible class of objects $\sL$ in a locally presentable
abelian category $\sE$, the class $\AdF(\sX,\sL)$ of all additive
functors $\sX\rarrow\sL$ ($\sL$~being viewed as a full subcategory
in~$\sE$) is (transmonic and) deconstructible in the locally
presentable abelian category $\AdF(\sX,\sE)$ of all additive functors
$\sX\rarrow\sE$.
 This is essentially~\cite[Corollary~3.4]{MR} or a particular case
of~\cite[Corollary~3.6]{MR}, in view of Proposition~\ref{fil-cell}.
%
%

 In particular, for any locally presentable abelian category $\sB$
with a projective generator $P$, we start with the transmonic
and deconstructible class $\Fil(\{P\})$ consisting of all coproducts
of copies of $P$. As this class is transmonic and deconstructible,
so is the class of all projective
objects $\sB_\proj=\Fil(\{P\})^\oplus$, and it further follows,
for a suitable choice of the additive category $\sX$
(see the proof of \cite[Proposition~4.5]{SS}), that the class
$\sC(\sB_\proj)$ is deconstructible in $\sC(\sB)$.
\end{proof}

\begin{proof}[Second proof]
 Here is a direct proof of the assertion stated in
Proposition~\ref{graded-projective-deconstructible}.
 It is clear from Lemma~\ref{fil-proj} that $\Fil(\sC(\sB_\proj))^\oplus
\subset\sC(\sB_\proj)$ (as the forgetful functor $\sC(\sB)\rarrow
\sB^\sgr$ preserves extensions and colimits, hence it also preserves
transfinitely iterated extensions in the sense of the directed colimit).
 So it suffices to find a set of objects $\sS\subset\sC(\sB_\proj)$
such that $\sC(\sB_\proj)\subset\Fil(\sS)^\oplus$. We will actually
find such a set of bounded below complexes, following the argument
for~\cite[Proposition 4.9, (1)$\implies$(2)]{SS}.

 Let $P$ be a projective generator of~$\sB$.
 Then any complex of projective objects in $\sB$ is a direct summand
of a complex whose terms are coproducts of copies of~$P$.
 Choose an uncountable regular cardinal~$\kappa$ such that
the object $P\in\sB$ is $\kappa$\+presentable (then the category $\sB$
is locally $\kappa$\+presentable).
 Let $\sS$ be the set of (representatives of isomorphism classes) of
bounded below complexes whose terms are coproducts of
less than~$\kappa$ copies of~$P$.
 We claim that any complex in $\sB$ whose terms are coproducts of
copies of $P$ belongs to $\Fil(\sS)\subset\sC(\sB)$.

 Let $Q^\bu$ be a complex in $\sB$ whose term $Q^n=P^{(X^n)}$ is
the coproduct of copies of $P$ indexed by a set $X^n$, for every
$n\in\boZ$.
 Let $\alpha$~be the successor cardinal of the cardinality of the
disjoint union $\coprod_{n\in\boZ}X^n$.
 Proceeding by transfinite induction on ordinals $0\le\beta\le\alpha$,
we will construct a smooth chain of subsets $Y_\beta^n\subset X^n$
such that $Y_0^n=\varnothing$ and $Y_\alpha^n=X^n$ for every
$n\in\boZ$, the cardinality of $Y_{\beta+1}^n\setminus Y_\beta^n$ is
smaller than~$\kappa$ for every $0\le\beta<\alpha$ and $n\in\boZ$
and is empty for $n\ll0$, and
the graded subobject with the terms $Q^n_\beta=P^{(Y_\beta^n)}
\subset P^{(X^n)}=Q^n$ is a subcomplex $Q^\bu_\beta$ in $Q^\bu$ for
every $0\le\beta\le\alpha$.

 For every element $x\in X^n$, let $\iota_x\:P\rarrow P^{(X^n)}=Q^n$
be the direct summand inclusion corresponding to the element~$x$.
 Since the object $P$ is $\kappa$\+presentable, there exists a subset
$Z_x\subset X^{n+1}$ of the cardinality less than~$\kappa$ such that
the composition of~$\iota_x$ with the differential $d^n\:Q^n\rarrow
Q^{n+1}$ factorizes through the direct summand (subcoproduct)
inclusion $P^{(Z_x)}\rarrow P^{(X^{n+1})}=Q^{n+1}$.

 Suppose that the subsets $Y_\gamma^n\subset X^n$ have been constructed
already for all $\gamma<\beta$ and $n\in\boZ$.
 For a limit ordinal~$\beta$, we put $Y_\beta^n=
\bigcup_{\gamma<\beta}Y_\gamma^n$ for every $n\in\boZ$.
 For a successor ordinal $\beta=\gamma+1$, if $Y_\gamma^n=X^n$ for
every $n\in\boZ$, then we put $Y_\beta^n=X^n$ as well.
 Otherwise, choose $m\in\boZ$ and an element
$z\in X^m\setminus Y_\gamma^m$.

 Proceeding by induction on $n\ge m$, define subsets $Z^n\subset X^n$
by the rules $Z^m=\{z\}$ and $Z^{n+1}=\bigcup_{x\in Z^n}Z_x$.
 Then $\dotsb\rarrow0\rarrow P=P^{(Z^m)}\rarrow P^{(Z^{m+1})}\rarrow
P^{(Z^{m+2})}\rarrow\dotsb$ is a subcomplex in~$Q^\bu$, and
the cardinality of the set $Z^n$ is smaller than~$\kappa$ for
every $n\ge m$.
 It remains to put $Y_\beta^n=Y_\gamma^n$ for $n<m$ and
$Y_\beta^n=Y_\gamma^n\cup Z^n$ for $n\ge m$.
\end{proof}

 In the special setting when $\sB$ is a Grothendieck category with
enough projective objects,
Proposition~\ref{graded-projective-deconstructible} (and its
second proof) becomes a particular case
of~\cite[Proposition~4.3]{Sto0}.

 Introduce the notation $\sC(\sB)^\ctr_\ac$ for the full subcategory
of contraacyclic complexes in $\sC(\sB)$.
 So $\sC(\sB)^\ctr_\ac\subset\sC(\sB)$ is the full preimage of
$\sK(\sB)^\ctr_\ac\subset\sK(\sB)$ under the natural functor
$\sC(\sB)\rarrow\sK(\sB)$.

\begin{thm} \label{contraderived-cotorsion-pair}
 Let\/ $\sB$ be a locally presentable abelian category with enough
projective objects.
 Then the pair of classes of objects\/ $\sC(\sB_\proj)$ and\/
$\sC(\sB)^\ctr_\ac$ is a hereditary complete cotorsion pair in
the abelian category\/ $\sC(\sB)$.
\end{thm}

\begin{proof}
 In the special case of the categories of modules over associative
rings, this theorem can be found in~\cite[Theorem~A.3]{BGH}.

 Notice that the abelian category $\sC(\sB)$ is locally presentable
by Lemma~\ref{complexes-locally-presentable}.
 Let $\sS$ be a set of complexes of projective objects in $\sB$ such
that $\sC(\sB_\proj)=\Fil(\sS)^\oplus$, as in
Proposition~\ref{graded-projective-deconstructible}.
 The claim is that the set $\sS\subset\sC(\sB)$ generates
the cotorsion pair we are interested in.

 Indeed, for any complexes $Q^\bu\in\sC(\sB_\proj)$ and $C^\bu\in
\sC(\sB)$ we have $\Ext^1_{\sC(\sB)}(Q^\bu,C^\bu)\simeq
\Hom_{\sK(\sB)}(Q^\bu,C^\bu[1])$ by Lemma~\ref{ext-1-hom-hot}(b).
 Hence $\sC(\sB_\proj)^{\perp_1}=\sC(\sB)_\ac^\ctr\subset\sC(\sB)$.
 By Proposition~\ref{eklof-lemma-prop}, it follows that
$\sS^{\perp_1}=\sC(\sB)_\ac^\ctr$.

 Furthermore, the class $\sC(\sB_\proj)=\Fil(\sS)^\oplus$ is clearly
generating in $\sC(\sB)$.
 Applying Theorem~\ref{cotorsion-pair-generated-by-set-left-class},
we can conclude that $\sC(\sB_\proj)={}^{\perp_1}(\sC(\sB)_\ac^\ctr)$.
 Alternatively, one can argue as in the proof of
Theorem~\ref{loc-pres-cotorsion-pair} in order to show that any
complex left $\Ext^1$\+orthogonal to all contractible complexes in
$\sC(\sB)$ is a complex of projective objects (as the complexes
$G^-(B)$ are contractible).

 Any complex is a subcomplex of a contractible complex, so the class
$\sC(\sB)_\ac^\ctr$ is cogenerating in~$\sC(\sB)$.
 Hence Theorem~\ref{cotorsion-pair-generated-by-set-complete} is
applicable, and we can conclude that our cotorsion pair is complete.
 The cotorsion pair is hereditary, since the class $\sC(\sB)_\ac^\ctr$
is closed under the cokernels of monomorphisms in $\sC(\sB)$, as
one can see from Lemma~\ref{contraacyclic-lemma}(a).
 It is also clear that the class $\sC(\sB_\proj)\subset\sC(\sB)$ is
closed under the kernels of epimorphisms.
\end{proof}

\begin{cor} \label{enough-projectives-for-contraderived}
 For any locally presentable abelian category\/ $\sB$ with enough
projective objects, the composition of the inclusion of triangulated
categories\/ $\sK(\sB_\proj)\rarrow\sK(\sB)$ with the Verdier quotient
functor\/ $\sK(\sB)\rarrow\sD^\ctr(\sB)$ is a triangulated
equivalence\/ $\sK(\sB_\proj)\simeq\sD^\ctr(\sB)$.
\end{cor}

\begin{proof}
 It is clear from the definitions that the triangulated functor
$\sK(\sB_\proj)\rarrow\sD^\ctr(\sB)$ is fully faithful
(cf.\ the proof of Proposition~\ref{orthogonal-fully-faithful}).
 In order to prove the corollary, it remains to find, for any
complex $C^\bu\in\sK(\sB)$, a complex of projective objects
$Q^\bu\in\sK(\sB_\proj)$ together with a morphism of complexes
$Q^\bu\rarrow C^\bu$ whose cone belongs to $\sK(\sB)_\ac^\ctr$.

 For this purpose, consider a special precover short exact sequence
$0\rarrow X^\bu\rarrow Q^\bu\rarrow C^\bu\rarrow0$ in the complete
cotorsion pair of Theorem~\ref{contraderived-cotorsion-pair}.
 So we have $X^\bu\in\sC(\sB)_\ac^\ctr$ and $Q^\bu\in\sC(\sB_\proj)$.
 By Lemma~\ref{contraacyclic-lemma}(a), the totalization
$\Tot(X^\bu\to Q^\bu\to C^\bu)$ of the short exact sequence
$0\rarrow X^\bu\rarrow Q^\bu\rarrow C^\bu\rarrow 0$ is
a contraacyclic complex.
 Since the complex $X^\bu$ is contraacyclic, it follows that
the cone of the morphism of complexes $Q^\bu\rarrow C^\bu$
is contraacyclic, too.
\end{proof}

 In terms of the full subcategory $\sK(\sB_\proj)\subset\sK(\sB)$,
Corollary~\ref{enough-projectives-for-contraderived} means that
the triangulated inclusion functor $\sK(\sB_\proj)\rarrow\sK(\sB)$
has a right adjoint.
For the categories of modules over an associative ring, this result
follows immediately from~\cite[Corollary~5.10]{Neem1}.
We also refer to~\cite[Theorem 3.5]{BEIJR} for another argument
to deduce this from Theorem~\ref{contraderived-cotorsion-pair}.

\begin{thm} \label{contraderived-model-structure}
 Let\/ $\sB$ be a locally presentable abelian category with enough
projective objects.
 Then the triple of classes of objects\/ $\sL=\sC(\sB_\proj)$, \
$\sW=\sC(\sB)_\ac^\ctr$, and\/ $\sR=\sC(\sB)$ is a cofibrantly generated
hereditary abelian model structure on the abelian category of
complexes\/~$\sC(\sB)$.
\end{thm}

\begin{proof}
 For the categories of modules over associative rings, this theorem
can be found in~\cite[Corollary~6.4]{BGH}.

 Quite generally, the argument is similar to the proof of
Theorem~\ref{loc-pres-model-structure}.
 The pair of classes $(\sL,\sW)$ form a hereditary complete cotorsion pair in
$\sC(\sB)$ by Theorem~\ref{contraderived-cotorsion-pair}.
 According to Lemma~\ref{intersection-is-inj-proj-complexes}(b),
it follows that $\sL\cap\sW=\sC(\sB)_\proj$ and
that the class of contraacyclic complexes $\sW$ is thick in $\sC(\sB)$
(see also Lemma~\ref{contraacyclic-lemma}(a)).
 By Lemma~\ref{inj-proj-model-structures}(b), the triple
$(\sL,\sW,\sR)$ is a projective abelian model structure on
the category $\sC(\sB)$.
 Finally, it was shown in the proof of
Theorem~\ref{contraderived-cotorsion-pair} that the cotorsion pair
$(\sL,\sW)$ in $\sC(\sB)$ is generated by a set of objects.
 For the cotorsion pair $(\sC(\sB)_\proj,\sC(\sB))$, the same was
explained in the proof of Theorem~\ref{loc-pres-model-structure}.
\end{proof}

 The abelian model structure $(\sL,\sW,\sR)$ defined in
Theorem~\ref{contraderived-model-structure} is called
the \emph{contraderived model structure} on the abelian category of
complexes $\sC(\sB)$.

\begin{lem} \label{contraderived-weak-equivalences}
 For any locally presentable abelian category\/ $\sB$ with enough
projective objects, the class\/ $\cW$ of all weak equivalences in
the contraderived model structure on the abelian category\/ $\sC(\sB)$
coincides with the class of all morphisms of complexes in\/ $\sB$
with the cones belonging to\/ $\sC(\sB)_\ac^\ctr$.
\end{lem}

\begin{proof}
 Similar to the proof of Lemma~\ref{loc-pres-weak-equivalences}.
 One needs to notice that, by Lemma~\ref{contraacyclic-lemma}(a),
a monomorphism of complexes in $\sC(\sB)$ has contraacyclic cokernel
if and only if it has contraacyclic cone, and similarly, an epimorphism
of complexes in $\sC(\sB)$ has contraacyclic kernel if and only if it
has contraacyclic cone.
\end{proof}

 The following corollary presumes existence of (set-indexed) coproducts
in the contraderived category $\sD^\ctr(\sB)$.
 Notice that such coproducts can be simply constructed as the coproducts
in the homotopy category $\sK(\sB_\proj)$, which is equivalent to
$\sD^\ctr(\sB)$ by Corollary~\ref{enough-projectives-for-contraderived}.

\begin{cor} \label{contraderived-well-generated}
 For any locally presentable abelian category\/ $\sB$ with enough
projective objects, the contraderived category\/ $\sD^\ctr(\sB)$ is
a well-generated triangulated category (in the sense of
the book~\cite{Neem-book} and the paper~\cite{Kra0}).
\end{cor}

\begin{proof}
 For the categories of modules over associative rings, this result
can be found in~\cite[Facts~2.8(i) and Theorem~5.9]{Neem1}.

 Quite generally, the argument is similar to the proof of
Corollary~\ref{loc-pres-well-generated}.
 One only needs to notice that the contraderived category
$\sD^\ctr(\sB)$ can be equivalently defined as the homotopy category
$\sC(\sB)[\cW^{-1}]$ of the contraderived model structure on $\sC(\sB)$.
 The key observation is that, for any additive category $\sE$,
inverting all the homotopy equivalences in the category of complexes
$\sC(\sE)$ produces the homotopy category $\sK(\sE)$ (i.~e.,
homotopic morphisms become equal after inverting the homotopy
equivalences); see~\cite[III.4.2--3]{GM} (cf.~\cite[last paragraph of
the proof of Theorem~2.5]{Psemi}).
 Then in the situation at hand, it follows that inverting all
the morphisms with contraacyclic cones in $\sC(\sB)$ produces
the contraderived category $\sD^\ctr(\sB)$.
\end{proof}

\begin{prop} \label{tilt}
 Let\/ $\sE$ be an accessible additive category with coproducts
and $M\in\sE$ be an object.
 Then there exists a (unique) locally presentable abelian category\/
$\sB$ with enough projective objects such that the full subcategory\/
$\sB_\proj\subset\sB$ is equivalent to the full subcategory\/
$\Add_\sE(M)\subset\sE$.
\end{prop}

\begin{proof}
 Any accessible additive category is
idempotent-complete~\cite[Observation~2.4]{AR}.
 For any idempotent-complete additive category $\sE$ with coproducts
and any object $M\in\sE$, there exists a unique abelian category\/
$\sB$ with enough projective objects such that the full subcategory of
projective objects\/ $\sB_\proj\subset\sB$ is equivalent to the full
subcategory\/ $\Add_\sE(M)\subset\sE$ \,\cite[Theorem~1.1(a)]{PS2},
\cite[Examples~1.2(1\+-2)]{Pper}.
 The category $\sB$ is locally presentable if and only if the object
$M\in\sE$ is abstractly $\kappa$\+small for some cardinal~$\kappa$;
and any $\kappa$\+presentable object in $\sE$ is abstractly
$\kappa$\+small~\cite[Section~1.1]{PR}, \cite[Section~6.4 and
Proposition~9.1]{PS1}, \cite[Section~1]{Pper}.
\end{proof}

Now we can elegantly recover a consequence of~\cite[Proposition 4.9]{SS}.

\begin{cor} \label{hot-addm-well-generated}
 For any accessible additive category\/ $\sE$ with coproducts and
any object $M\in\sE$, the homotopy category\/ $\sK(\Add_\sE(M))$
of (unbounded complexes in) the additive category\/ $\Add_\sE(M)
\subset\sE$ is a well-generated triangulated category.
\end{cor}

\begin{proof}
 By Proposition~\ref{tilt}, there exists a locally presentable abelian
category $\sB$ with enough projective objects such that the additive
category $\Add_\sE(M)$ is equivalent to\/~$\sB_\proj$.
 By Corollary~\ref{enough-projectives-for-contraderived}, the homotopy
category $\sK(\Add_\sE(M))\simeq\sK(\sB_\proj)$ can be interpreted
as the contraderived category $\sD^\ctr(\sB)$ (in the sense of Becker).
 According to Corollary~\ref{contraderived-well-generated},
the triangulated category $\sD^\ctr(\sB)$ is well-generated.
\end{proof}

\Section{Grothendieck Abelian Categories}  \label{grothendieck-secn}

 In this section we discuss the conventional derived category $\sD(\sA)$
of a Grothendieck abelian category~$\sA$.
 Recall that a cocomplete abelian category is said to be
\emph{Grothendieck} if it has exact functors of directed colimit and
a set of generators.

 The abelian category of complexes $\sC(\sA)$ in a Grothendieck abelian
category $\sA$ (or more generally, any category of additive functors
$\AdF(\sX,\sA)$, as in Lemma~\ref{complexes-locally-presentable})
is again a Grothendieck abelian category.
 It is well-known that all Grothendieck categories are locally
presentable (for a reference, see~\cite[Corollary~5.2]{Kra3},
and for a generalization, \cite[Theorem~2.2]{PR}).

 Let $\sA$ be an abelian category.
 A complex $J^\bu\in\sK(\sA)$ is said to be \emph{homotopy injective}
if $\Hom_{\sK(\sA)}(X^\bu,J^\bu)=0$ for any acyclic complex $X^\bu$
in~$\sA$.
 We denote the full subcategory of homotopy injective complexes by
$\sK(\sA)_\hin\subset\sK(\sA)$.
 Furthermore, let us denote by $\sK(\sA_\inj)_\hin=\sK(\sA_\inj)\cap
\sK(\sA)_\hin\subset\sK(\sA)$ the full subcategory of \emph{homotopy
injective complexes of injective objects} in the homotopy
category~$\sK(\sA)$.
 Both $\sK(\sA)_\hin$ and $\sK(\sA_\inj)_\hin$ are triangulated
subcategories in $\sK(\sA)$.
 (See Remark~\ref{homotopy-adjusted-terminology} for a terminological
discussion with references.)

\begin{prop} \label{dual-orthogonal-fully-faithful}
 For any abelian category\/ $\sA$, the composition of the fully
faithful inclusion\/ $\sK(\sA)_\hin\rarrow\sK(\sA)$ with the Verdier
quotient functor\/ $\sK(\sA)\rarrow\sD(\sA)$ is a fully faithful
functor\/ $\sK(\sA)_\hin\rarrow\sD(\sA)$.
 Hence we have a pair of fully faithful triangulated functors\/
$\sK(\sA_\inj)_\hin\rarrow\sK(\sA)_\hin\rarrow\sK(\sA)$.
\end{prop}

\begin{proof}
 This is a dual assertion to
Proposition~\ref{orthogonal-fully-faithful}.
\end{proof}

 We will say that an abelian category $\sA$ has \emph{enough
homotopy injective complexes} if the functor $\sK(\sA)_\hin\rarrow
\sD(\sA)$ is a triangulated equivalence.
 Moreover, we will say that $\sA$ has \emph{enough homotopy
injective complexes of injective objects} if the functor
$\sK(\sA_\inj)_\hin\rarrow\sD(\sA)$ is a triangulated equivalence.

 In the latter case, clearly, the inclusion $\sK(\sA_\inj)_\hin
\rarrow\sK(\sA)_\hin$ is a triangulated equivalence, too.
 So, in an abelian category with enough homotopy injective complexes
of injective objects, any homotopy injective complex is homotopy
equivalent to a homotopy injective complex of injective objects.

 As in Section~\ref{loc-pres-secn}, we denote by $\sC(\sA)_\ac
\subset\sC(\sA)$ the full subcategory of acyclic complexes
in~$\sC(\sA)$.
 We also denote by $\sC(\sA_\inj)_\hin\subset\sC(\sA)$ the full
subcategory of homotopy injective complexes of injective objects
in~$\sC(\sA)$.

\begin{lem} \label{grothendieck-deconstructible}
 Let\/ $\sA$ be a Grothendieck abelian category.
 Then there exists a set\/ $\sS_0$ of objects in\/ $\sA$ such that
all the objects of\/ $\sA$ are filtered by objects from\/ $\sS_0$,
that is\/ $\sA=\Fil(\sS_0)$.
\end{lem}

\begin{proof}
 This well-known observation can be found
in~\cite[Example~1.2.6(1)]{Bec}.
 Choose a set of generators $\sS$ for the abelian category $\sA$,
and let $\sS_0$ consist of all (representatives of isomorphism classes)
of quotient objects of the objects from $\sS$ in~$\sA$.
 Let $M\in\sA$ be an object.
 Consider the disjoint union $\coprod_{S\in\sS}\Hom_\sA(S,M)$ of
the sets $\Hom_\sA(S,M)$, \,$S\in\sS$, and choose a well-ordering of
this set by identifying it with some ordinal~$\alpha$.
 So for every ordinal $i<\alpha$ we have an object $S_i\in\sS$
and a morphism $f_i\:S_i\rarrow M$.
 For every ordinal $j\le\alpha$, put $F_j=\sum_{i<j}f_i(S_i)\subset M$.
 Then the inductive system of subobjects $F_i\subset M$ is
an $\alpha$\+filtration of the object $M$ by objects from~$\sS_0$.
\end{proof}

\begin{prop} \label{acyclic-deconstructible}
 Let\/ $\sA$ be a Grothendieck abelian category.
 Then there exists a set of acyclic complexes\/ $\sS\subset
\sC(\sA)_\ac$ such that the class of all acyclic complexes\/
$\sC(\sA)_\ac\subset\sC(\sA)$ is the class of all complexes
filtered by\/ $\sS$, that is\/ $\sC(\sA)_\ac=\Fil(\sS)$.
\end{prop}

\begin{proof}
 Taking Lemma~\ref{grothendieck-deconstructible} into account,
the assertion of the proposition becomes a particular case
of~\cite[Proposition~4.4]{Sto0}.
\end{proof}

 The following theorem is to be compared with the discussion
in~\cite[Remark~1.3.13]{Bec}.

\begin{thm} \label{grothendieck-cotorsion-pair}
 Let\/ $\sA$ be a Grothendieck abelian category.
 Then the pair of classes of objects\/ $\sC(\sA)_\ac$ and\/
$\sC(\sA_\inj)_\hin$ is a hereditary complete cotorsion pair in
the abelian category\/ $\sC(\sA)$.
\end{thm}

\begin{proof}
 Let $\sS$ be a set of acyclic complexes in $\sA$ such that
$\sC(\sA)_\ac=\Fil(\sS)$, as
in Proposition~\ref{acyclic-deconstructible}.
 The claim is that the set $\sS\subset\sC(\sA)$ generates
the cotorsion pair we are interested in.

 Indeed, by Proposition~\ref{eklof-lemma-prop} we have
$\sS^{\perp_1}=(\sC(\sA)_\ac)^{\perp_1}\subset\sC(\sA)$.
 In the pair of adjoint functors $G^+\:\sA^\sgr\rarrow\sC(\sA)$ and
$({-})^\sharp\:\sC(\sA)\rarrow\sA^\sgr$, both the functors are exact.
 Hence of any graded object $A\in\sA^\sgr$ and any complex
$C^\bu\in\sC(\sA)$ we have $\Ext^n_{\sC(\sA)}(G^+(A),C^\bu)\simeq
\Ext^n_{\sA^\sgr}(A,C^\bu{}^\sharp)$ for all $n\ge0$.
 In particular, this isomorphism holds for $n=1$.
 Since the complex $G^+(A)$ is acyclic, we have shown that
$(\sC(\sA)_\ac)^{\perp_1}\subset\sC(\sA_\inj)$.
 Now for a complex $J^\bu\in\sC(\sA_\inj)$ and any complex
$X^\bu\in\sC(\sA)$ we have $\Ext^1_{\sC(\sA)}(X^\bu,J^\bu)\simeq
\Hom_{\sK(\sA)}(X^\bu,J^\bu[1])$ by Lemma~\ref{ext-1-hom-hot}(a).
 Thus $\sS^{\perp_1}=(\sC(\sA)_\ac)^{\perp_1}=\sC(\sA_\inj)_\hin
\subset\sC(\sA)$.

 Furthermore, the class of all acyclic complexes $\sC(\sA)_\ac$
is clearly generating in $\sC(\sA)$.
 Applying Theorem~\ref{cotorsion-pair-generated-by-set-left-class},
we can conclude that ${}^{\perp_1}(\sC(\sA_\inj)_\hin)
=\Fil(\sS)^\oplus=\sC(\sA)_\ac$.

 Any complex in $\sA$ is a subcomplex of a contractible complex of
injective objects, so the class $\sC(\sA_\inj)_\hin$ is cogenerating
in $\sC(\sA)$.
 Hence Theorem~\ref{cotorsion-pair-generated-by-set-complete} is
applicable, and we can conclude that our cotorsion pair is complete.
 The cotorsion pair is hereditary, since the class $\sC(\sA)_\ac$ is
closed under the kernels of epimorphisms in $\sC(\sA)$.
\end{proof}

Hovey in~\cite[Example 3.2]{Hov} attributes the following corollary to Joyal, who should have mentioned it in a 1984 letter of his to Grothendieck.
A weaker version appeared in~\cite[Theorem~5.4]{AJS}, where existence
of enough homotopy injective complexes was established.
The existence of enough homotopy injective complexes of injective
objects was shown in~\cite[Theorem~3.13 and Lemma~3.7(ii)]{Ser}
and, following the approach of~\cite{Hov}, independently also in~\cite[Corollary~7.1]{Gil}.

\begin{cor} \label{enough-homotopy-injectives}
 Any Grothendieck abelian category has enough homotopy injective
complexes of injective objects.
In other words, the natural functor\/
$\sK(\sA_\inj)_\hin\rarrow\sD(\sA)$ is a triangulated equivalence.
\end{cor}

\begin{proof}
 Here is a proof based on Theorem~\ref{grothendieck-cotorsion-pair}.
 Let $\sA$ be a Grothendieck abelian category.
 Given a complex $C^\bu\in\sC(\sA)$, we need to find a homotopy
injective complex of injective objects $J^\bu$ together with
a quasi-isomorphism $C^\bu\rarrow J^\bu$ of complexes in~$\sA$.
 For this purpose, consider a special preenvelope short exact
sequence $0\rarrow C^\bu\rarrow J^\bu\rarrow X^\bu\rarrow0$ with
$J^\bu\in\sC(\sA_\inj)_\hin$ and $X^\bu\in\sC(\sA)_\ac$.
 Since the complex $X^\bu$ is acyclic, it follows that the monomorphism
of complexes $C^\bu\rarrow J^\bu$ is a quasi-isomorphism.
\end{proof}

 The following result can be found in~\cite[Theorem~4.12 and
Corollary~7.1]{Gil}.

\begin{thm} \label{grothendieck-model-structure}
 Let\/ $\sA$ be a Grothendieck abelian category.
 Then the triple of classes of objects\/ $\sL=\sC(\sA)$, \
$\sW=\sC(\sA)_\ac$, and\/ $\sR=\sC(\sA_\inj)_\hin$ is a cofibrantly
generated hereditary abelian model structure on the abelian category
of complexes $\sC(\sA)$.
\end{thm}

\begin{proof}
 The pair of classes $(\sW,\sR)$ is a complete cotorsion pair in
$\sC(\sA)$ by Theorem~\ref{grothendieck-cotorsion-pair}.
 The abelian category of complexes $\sC(\sA)$ has enough injective
objects since it is Grothendieck (or since the abelian category $\sA$
has).
 According to Lemma~\ref{intersection-is-inj-proj-complexes}(a),
it follows that $\sR\cap\sW=\sC(\sA)_\inj$.
 The class of acyclic complexes $\sW$ is obviously thick in $\sC(\sA)$
(see also Lemma~\ref{W-is-thick}(a) or~\ref{intersection-is-inj-proj-complexes}(a)).
 By Lemma~\ref{inj-proj-model-structures}(a) (applied to
the category~$\sC(\sA)$), the triple $(\sL,\sW,\sR)$ is an injective
abelian model structure on $\sC(\sA)$.

 As explained in Section~\ref{abelian-model-secn}, 
any injective abelian model structure is hereditary.
 Finally, it is clear from the proof of
Theorem~\ref{grothendieck-cotorsion-pair} that the cotorsion pair
$(\sW,\sR)$ is $\sC(\sA)$ is generated by a set of objects.
 The cotorsion pair $(\sC(\sA),\sC(\sA)_\inj)$ is generated by
any set of complexes $\sS_0$ such that $\sC(\sA)=\Fil(\sS_0)$,
as in Lemma~\ref{grothendieck-deconstructible} applied to
the category $\sC(\sA)$.
 According to Corollary~\ref{abelian-cofibrantly-generated}, this
means that our abelian model structure is cofibrantly generated.
\end{proof}

 The abelian model structure $(\sL,\sW,\sR)$ defined in
Theorem~\ref{grothendieck-model-structure} is called
the \emph{injective derived model structure} on the abelian
category of complexes~$\sC(\sA)$.

\begin{lem} \label{grothendieck-weak-equivalences}
 For any Grothendieck abelian category\/ $\sA$, the class\/ $\cW$ of
all weak equivalences in the injective derived model structure on
the abelian category\/ $\sC(\sA)$ coincides with the class of all
quasi-isomorphisms of complexes in\/~$\sA$.
\end{lem}

\begin{proof}
 Similar to Lemma~\ref{loc-pres-weak-equivalences}.
\end{proof}

 The following corollary presumes existence of (set-indexed) coproducts
in the derived category $\sD(\sA)$.
 Notice that, since the coproducts are exact in the abelian category
$\sA$ and consequently the thick subcategory of acyclic complexes
$\sK(\sA)_\ac\subset\sK(\sA)$ is closed under coproducts,
the coproducts in the derived category $\sD(\sA)$ are induced by
those in the homotopy category $\sK(\sA)$.
 In other words, the Verdier quotient functor $\sK(\sA)\rarrow\sD(\sA)$
preserves coproducts~\cite[Lemma~3.2.10]{Neem-book}.

\begin{cor} \label{grothendieck-well-generated}
 For any Grothendieck abelian category\/ $\sA$, the (unbounded) derived
category\/ $\sD(\sA)$ is a well-generated triangulated category.
\end{cor}

\begin{proof}
 This result can be found in~\cite[Section~7.7]{Kra2}, or in a more 
precise form in~\cite[Theorem~5.10]{Kra3}.
 It is also provable similarly
to Corollary~\ref{loc-pres-well-generated}.
\end{proof}

 The next theorem appeared in the context of DG\+comodules over
a DG\+coalgebra (over a field) in the postpublication \texttt{arXiv}
version of~\cite[Section~5.5]{Pkoszul};
see also~\cite[Theorem~1.1(d)]{Pmc}.

\begin{thm} \label{grothendieck-generated-by-J-as-colocalizing}
 Assume Vop\v enka's principle, and let\/ $\sA$ be
a Grothendieck abelian category with an injective cogenerator~$J$.
 Then the category\/ $\sD(\sA)$ is generated, as a triangulated
category with products, by the single object $J$ (viewed as
a one-term complex concentrated in degree~$0$).
 In other words, the full subcategory of homotopy injective complexes\/
$\sK(\sA)_\hin\subset\sK(\sA)$ is the minimal strictly full triangulated
subcategory in\/ $\sK(\sA)$ containing the object $J$ and closed under
products.
\end{thm}

\begin{proof}
 Notice first of all that (set-indexed) products in the derived
category $\sD(\sA)$ can be simply constructed as the products in
the homotopy category $\sK(\sA)_\hin$, which is equivalent to
$\sD(\sA)$ by Corollary~\ref{enough-homotopy-injectives}.
 The full subcategory $\sK(\sA)_\hin$ is closed under products in
$\sK(\sA)$, so the products in $\sK(\sA)_\hin$ (just as in $\sK(\sA)$)
can be computed as the termwise products of complexes.

 The key observation is that $\Hom_{\sD(\sA)}(X^\bu,J[i])=0$ for all
$i\in\boZ$ implies $X^\bu=0$ for a given object $X^\bu\in\sD(\sA)$,
since for any complex $X^\bu$ one has a natural isomorphism of
abelian groups $\Hom_{\sD(\sA)}(X^\bu,J[i])\simeq
\Hom_\sA(H^{-i}(X^\bu),J)$.
 Denote by $\sD''\subset\sD(\sA)$ the minimal triangulated subcategory
in $\sD(\sA)$ containing $J$ and closed under products.

 Following~\cite[Corollary~1.1.15 and the preceding discussion]{Bec},
any hereditary abelian model category is stable.
 The derived category $\sD(\sA)$ can be equivalently defined as
the homotopy category $\sC(\sA)[\cW^{-1}]$ of the injective derived
model category structure on $\sC(\sA)$.
 The category $\sC(\sA)$ is Grothendieck, hence locally presentable.
 According to~\cite[Theorem~2.4]{CGR}, assuming Vop\v enka's principle,
any triangulated subcategory closed under products in $\sD(\sA)$ is
reflective.
 So, in particular, the inclusion functor $\sD''\rarrow\sD(\sA)$ has
a left adjoint.

 The left adjoint functor to a fully faithful triangulated functor is
a Verdier quotient functor.
 The kernel of our functor $\sD(\sA)\rarrow\sD''$ consists of complexes
$X^\bu$ satisfying $\Hom_{\sD(\sA)}(X^\bu,J[i])=0$ for all $i\in\boZ$,
since $J[i]\in\sD''$.
 Thus all such objects $X^\bu\in\sD(\sA)$ vanish, and the functor
$\sD(\sA)\rarrow\sD''$ is a triangulated equivalence.
 It follows that the inclusion $\sD''\rarrow\sD(\sA)$ is a triangulated
equivalence, too; so $\sD''=\sD(\sA)$, as desired.
\end{proof}

\Section{Coderived Model Structure}  \label{coderived-secn}

 In this section we consider coderived categories \emph{in the sense
of Becker}~\cite{Bec}.
 In well-behaved cases, this means simply the homotopy category of
complexes of injective objects (which was studied first by
Krause~\cite{Kra} in the case of locally Noetherian Grothendieck
categories); see the discussion in the introduction.
 The coderived category in the sense of Becker was also considered
in the preprint~\cite[Section~6]{Sto} and
the paper~\cite[Section~4.1]{Gil4}.

 The coderived category in the sense of Becker has to be
distinguished from the coderived category in the sense of
the books and papers~\cite{Psemi,Pkoszul,PP2,EP,Pweak,PS2,Pps}.
 The two definitions of a coderived category are known to be
equivalent under certain assumptions~\cite[Theorem~3.7]{Pkoszul},
but it is still an open question whether they are equivalent for
the category of modules over an arbitrary associative ring
(see~\cite[Example~2.5(3)]{Pps} for a discussion).

 Let $\sA$ be an abelian category.
 A complex $X^\bu\in\sK(\sA)$ is said to be \emph{coacyclic}
(in the sense of Becker) if $\Hom_{\sK(\sA)}(X^\bu,J^\bu)=0$ for
any complex of injective objects $J^\bu\in\sK(\sA_\inj)$.
 We denote the full subcategory of coacyclic complexes
by $\sK(\sA)_\ac^\co\subset\sK(\sA)$ and its full preimage
under the natural functor $\sC(\sA)\rarrow\sK(\sA)$
by $\sC_\ac^\co(\sA)$.
 Clearly, $\sK(\sA)_\ac^\co$ is a triangulated (and even thick)
subcategory in the homotopy category $\sK(\sA)$.
 The quotient category $\sD^\co(\sA)=\sK(\sA)/\sK(\sA)_\ac^\co$
is called the \emph{coderived category of\/~$\sA$} (in the sense
of Becker).

\begin{lem} \label{coacyclic-lemma}
\textup{(a)} For any short exact sequence\/ $0\rarrow K^\bu\rarrow
L^\bu\rarrow M^\bu\rarrow0$ of complexes in\/ $\sA$, the total complex\/
$\Tot(K^\bu\to L^\bu\to M^\bu)$ of the bicomplex with three rows
$K^\bu\rarrow L^\bu\rarrow M^\bu$ belongs to\/ $\sK(\sA)_\ac^\co$. \par
\textup{(b)} The full subcategory of coacyclic complexes\/
$\sK(\sA)_\ac^\co$ is closed under coproducts in the homotopy category\/
$\sK(\sA)$.
\end{lem}

\begin{proof}
 This is our version of~\cite[Theorem~3.5(a)]{Pkoszul}.
 It can be obtained from Lemma~\ref{contraacyclic-lemma} by
inverting the arrows.
\end{proof}

\begin{rem} \label{coderived-history}
 Let us say a few words about the history of the ``coderived category''
nomenclature and various concepts behind it.
 It appears that the term ``coderived category'' first appeared in
a brief exposition by Keller~\cite{Kel2} of some results from
the Ph.D. thesis of his student Lef\`evre-Hasegawa~\cite{Lef}.
 The results in question represented a general formulation of
differential graded Koszul duality, connecting DG\+modules over
DG\+algebras and DG\+comodules over DG\+coalgebras.
 Thus the coderived category was originally defined for DG\+comodules
over (conilpotent) DG\+coalgebras only; the definition was not
intrinsic to the category of DG\+comodules, in that it used
the passage to DG\+modules over the Koszul dual DG\+algebra.

 Several years earlier, the first-named author of the present paper
came up with his definition of an exotic derived category relevant for
the derived nonhomogeneous Koszul duality purposes.
 This definition had much wider applicability, from curved DG\+modules
to abelian and exact categories with exact coproducts; but it did not
have a convenient name.
 Originally, the term ``derived category of the second kind'' was used
by Positselski in his seminar talks; but it was not convenient in that,
besides being too long, it also stood for two dual concepts
simultaneously (what are now called the coderived and
the contraderived category).
 This problem of the lack of convenient terminology became one of
the (several) reasons why the publication of these results got delayed
for so many years.

 Eventually, the first-named author of the present paper discovered
Keller's note~\cite{Kel2} and picked up the ``coderived category''
nomenclature from it, developing it into a terminological system
featuring also the contraderived (and semiderived) categories.
 Positselski's definitions of the coderived and contraderived categories
appeared in the preprint versions (and subsequently in the published
versions) of the book~\cite{Psemi} and the memoir~\cite{Pkoszul}.
 As a corollary of the Koszul duality theorems from~\cite{Pkoszul}, one
could see that, for DG\+comodules over a DG\+coalgebra over a field,
Positselski's coderived category agrees with the one of
Lef\`evre-Hasegawa and Keller~\cite{Lef,Kel2}.

 So the original point of view was that one should consider the derived
categories of modules and the coderived categories of comodules (also,
the contraderived categories of contramodules); hence the terminology.
 This philosophy proved to be highly illuminating in the context of
semi-infinite homological algebra~\cite{Psemi}.
 It was soon realized that the coderived and contraderived (also
the ``absolute derived'') categories are useful for modules,
too~\cite[Section~6.7]{Pkoszul}, particularly in the context of
matrix factorizations~\cite{PP2,Or,EP,BDFIK}.
 (Of course, the papers~\cite{Jor,Kra,Neem1} came earlier, but
they can be properly classified as representing the approach
which led to Becker's co/contraderived categories.)

 Becker~\cite[Proposition~1.3.6]{Bec} used the ``coderived'' and
``contraderived'' terminology for two abelian model structures on
the category of CDG\+modules over a CDG\+ring which he constructed.
 Hence the ``coderived category in the sense of Becker'', which we
discuss in this section (and the ``contraderived category in
the sense of Becker'', which was studied in
Section~\ref{contraderived-secn}).

 Let us briefly formulate the definitions from~\cite{Psemi,Pkoszul}.
 Let $\sA$ be an abelian category (or more generally, an exact category)
with exact coproduct functors.
 Then a complex in $\sA$ is called coacyclic in the sense
of~\cite{Psemi}, \cite{Pkoszul}, etc., if it belongs to the minimal
triangulated subcategory of $\sK(\sA)$ satisfying the conditions~(a)
and~(b) of Lemma~\ref{coacyclic-lemma}.
 Dually, let $\sB$ be an abelian (or exact) category with exact
products.
 Then a complex in $\sB$ is called contraacyclic in the sense
of~\cite{Psemi}, \cite{Pkoszul}, etc., if it belongs to the minimal
triangulated subcategory of $\sK(\sB)$ satisfying the conditions~(a)
and~(b) of Lemma~\ref{contraacyclic-lemma}.
 The coderived (resp., contraderived) category is defined
in~\cite{Psemi,Pkoszul} as the triangulated quotient category of
the homotopy category by the thick subcategory of coacyclic
(resp., contraacyclic) complexes.
 In an independent development, a similar approach to constructions
of exotic derived categories based on axiomatization of the properties
listed in Lemmas~\ref{contraacyclic-lemma} and~\ref{coacyclic-lemma}
was tried in the paper~\cite{KLN}.

 Thus any complex coacyclic in the sense of~\cite{Psemi,Pkoszul} is
also coacyclic in the sense of Becker~\cite{Bec}; and any complex
contraacyclic in the sense of~\cite{Psemi,Pkoszul} is contraacyclic
in the sense of Becker~\cite{Bec}; but the converse implications
remain an open problem in general.
 All we know is that the coderived category in the sense of Becker
can be viewed as a quotient category or as a subcategory of
the coderived category in the sense of~\cite{Psemi,Pkoszul}; and
similarly the contraderived category in the sense of Becker can be
viewed as a quotient category or as a subcategory of the contraderived
category in the sense of~\cite{Psemi,Pkoszul}.
 The results of~\cite[Sections~3.7--3.8]{Pkoszul} describe
the contexts in which the two approaches are known to agree.
\end{rem}

 The following theorem together with
Theorem~\ref{coderived-model-structure} below are essentially
a more detailed formulation of~\cite[Theorem~4.2]{Gil4}.

\begin{thm} \label{coderived-cotorsion-pair}
 Let\/ $\sA$ be a Grothendieck abelian category.
 Then the pair of classes of objects\/ $\sC(\sA)_\ac^\co$ and\/
$\sC(\sA_\inj)$ is a hereditary complete cotorsion pair in
the abelian category\/ $\sC(\sA)$.
\end{thm}

\begin{proof}
 Let $\sS_0$ be a set of objects in $\sA$ such that $\sA=\Fil(\sS_0)$,
as in Lemma~\ref{grothendieck-deconstructible}.
 The claim is that the cotorsion pair we are interested in is generated
by the set of two-term complexes
$\sS=\{G^+(S)[i]\}_{S\in\sS_0,\,i\in\boZ}\subset\sC(\sA)$
(where $S\in\sS_0$ is viewed as a graded object concentrated in
degree~$0$).

 Indeed, for any graded object $A\in\sA^\sgr$ and any complex
$C^\bu\in\sC(\sA)$ we have $\Ext_{\sC(\sA)}^1(G^+(A),C^\bu)\simeq
\Ext_{\sA^\sgr}^1(A,C^\bu{}^\sharp)$, as explained in the proof of
Theorem~\ref{grothendieck-cotorsion-pair}.
 In particular, for $S\in\sS_0$ we have
$\Ext_{\sC(\sA)}^1(G^+(S)[i],C^\bu)\simeq
\Ext_{\sA^\sgr}^1(S[i],C^\bu{}^\sharp)\simeq\Ext_\sA^1(S,C^{-i})$.
 In view of Proposition~\ref{eklof-lemma-prop}, it follows that
$\sS^{\perp_1}=\sC(\sA_\inj)\subset\sC(\sA)$.

 Furthermore, for any complex $X^\bu\in\sC(\sA)$ and any complex of
injective objects $J^\bu\in\sC(\sA)$ we have
$\Ext^1_{\sC(\sA)}(X^\bu,J^\bu)\simeq\Hom_{\sK(\sA)}(X^\bu,J^\bu[1])$
by Lemma~\ref{ext-1-hom-hot}(a).
 Therefore, ${}^{\perp_1}(\sC(\sA_\inj))=\sC(\sA)_\ac^\co$.
 Hence our pair of classes of objects is indeed the cotorsion pair
generated by the set $\sS\subset\sC(\sA)$.

 Any complex is a quotient complex of a contractible complex.
 All contractible complexes are coacyclic; so the class
$\sC(\sA)_\ac^\co$ is generating in $\sC(\sA)$.
 Any complex in $\sA$ is also a subcomplex of a complex of injective
objects, so the class $\sC(\sA_\inj)$ is cogenerating in $\sC(\sA)$.
 Applying Theorem~\ref{cotorsion-pair-generated-by-set-complete},
we conclude that our cotorsion pair is complete.
 The cotorsion pair is hereditary, since the class $\sC(\sA)_\ac^\co$
is closed under the kernels of epimorphisms in $\sC(\sA)$, as one can
see from Lemma~\ref{coacyclic-lemma}(a).
 It is also clear that the class $\sC(\sA_\inj)\subset\sC(\sA)$ is
closed under the cokernels of monomorphisms.
\end{proof}

The proof of Theorem~\ref{coderived-cotorsion-pair} also provides us
with the following description of the class of coacyclic complexes
in the sense of Becker.

\begin{cor}
Let $\sA$ be a Grothendieck category. Then the coacyclic complexes are
precisely direct summands of those filtered by contractible complexes.
In fact, if\/ $\sS_0\subset\sA$ is a set of objects such that\/
$\sA=\Fil(\sS_0)$ and\/
$\sS=\{G^+(S)[i]\}_{S\in\sS_0,\,i\in\boZ}\subset\sC(\sA)$ as before,
then\/ $\sC(\sA)_\ac^\co=\Fil(\sS)^\oplus$.
\end{cor}
\begin{proof}
%
Let $L\in\sA$ be a generator. Since $L\in\Fil(\sS_0)$, we clearly have
$G^+(L)[i]\in\Fil(\sS)$ for each $i\in\boZ$. Since $\Fil(\sS)$ is also
closed under coproducts in $\sC(\sA)$, it is a generating class and,
thus, Theorem~\ref{cotorsion-pair-generated-by-set-left-class} tells
that $\sC(\sA)_\ac^\co=\Fil(\sS)^\oplus$. Since every complex in $\sS$
is contractible and any contractible complex is coacyclic,
it also follows that $\sC(\sA)_\ac^\co$ is the class of all
direct summands of complexes
filtered by contractible complexes.
\end{proof}

\begin{cor} \label{enough-injectives-for-coderived}
 For any Grothendieck abelian category\/ $\sA$, the composition of
the inclusion of triangulated categories\/ $\sK(\sA_\inj)\rarrow
\sK(\sA)$ with the Verdier quotient functor\/ $\sK(\sA)\rarrow
\sD^\co(\sA)$ is a triangulated equivalence\/ $\sK(\sA_\inj)\simeq
\sD^\co(\sA)$.
\end{cor}

\begin{proof}
 It is clear from the definitions that the functor $\sK(\sA_\inj)\rarrow
\sD^\co(\sA)$ is fully faithful (use the argument dual to the proof
of Proposition~\ref{orthogonal-fully-faithful}).
 In order to prove the corollary, it remains to find, for any complex
$C^\bu\in\sK(\sA)$, a complex of injective objects
$J^\bu\in\sK(\sA_\inj)$ together with a morphism of complexes $C^\bu
\rarrow J^\bu$ whose cone belongs to $\sC(\sA)_\ac^\co$.

 For this purpose, consider a special preenvelope short exact sequence
$0\rarrow C^\bu\rarrow J^\bu\rarrow X^\bu\rarrow0$ in the complete
cotorsion pair of Theorem~\ref{coderived-cotorsion-pair}.
 So we have $J^\bu\in\sC(\sA_\inj)$ and $X^\bu\in\sC(\sA)_\ac^\co$.
 By Lemma~\ref{coacyclic-lemma}, the totalization $\Tot(C^\bu\to J^\bu
\to X^\bu)$ of the short exact sequence $0\rarrow C^\bu\rarrow J^\bu
\rarrow X^\bu\rarrow0$ is a coacyclic complex.
 Since the complex $X^\bu$ is coacyclic, it follows that the cone of
the morphism $C^\bu\rarrow J^\bu$ is coacyclic, too.
\end{proof}

 In terms of the full subcategory $\sK(\sA_\inj)\subset\sK(\sA)$,
Corollary~\ref{enough-injectives-for-coderived} means that
the triangulated inclusion functor $\sK(\sA_\inj)\rarrow\sK(\sA)$ has
a left adjoint.
 This result can be also directly deduced from~\cite[Theorem~3.5]{BEIJR}
and can be found in~\cite[Theorem~2.13]{Neem2}
or~\cite[Corollary~5.13]{Kra3}.

\begin{thm} \label{coderived-model-structure}
 Let\/ $\sA$ be a Grothendieck abelian category.
 Then the triple of classes of objects\/ $\sL=\sC(\sA)$, \
$\sW=\sC(\sA)_\ac^\co$, and\/ $\sR=\sC(\sA_\inj)$ is a cofibrantly
generated hereditary abelian model structure on the abelian
category of complexes\/ $\sC(\sA)$.
\end{thm}

\begin{proof}
 This is similar to the proof of
Theorem~\ref{grothendieck-model-structure}.
 The pair of classes $(\sW,\sR)$ is a complete cotorsion pair in
$\sC(\sA)$ by Theorem~\ref{coderived-cotorsion-pair}.
 According to Lemma~\ref{intersection-is-inj-proj-complexes}(a), it
follows that $\sR\cap\sW=\sC(\sA)_\inj$.
 It follows from Lemma~\ref{coacyclic-lemma}(a) that the class of
coacyclic complexes $\sW$ is thick in $\sC(\sA)$ (see also
Lemma~\ref{W-is-thick}(a) or~\ref{intersection-is-inj-proj-complexes}(a)).
 By Lemma~\ref{inj-proj-model-structures}(a), the triple $(\sL,\sW,\sR)$
is an injective abelian model structure on the category $\sC(\sA)$.
 Finally, it was shown in the proof of
Theorem~\ref{coderived-cotorsion-pair} that the cotorsion pair
$(\sL,\sW)$ in $\sC(\sA)$ is generated by a set of objects.
 For the cotorsion pair $(\sC(\sA),\sC(\sA)_\inj)$, the same was
explained in the proof of Theorem~\ref{grothendieck-model-structure}.
\end{proof}

 The abelian model structure $(\sL,\sW,\sR)$ defined in
Theorem~\ref{coderived-model-structure} is called the \emph{coderived
model structure} on the abelian category of complexes $\sC(\sA)$.

\begin{lem} \label{coderived-weak-equivalences}
 For any Grothendieck abelian category\/ $\sA$, the class\/ $\cW$ of all
weak equivalences in the coderived model structure on the abelian
category\/ $\sC(\sA)$ coincides with the class of all morphisms of
complexes in\/ $\sA$ with the cones belonging to\/ $\sC(\sA)_\ac^\co$.
\end{lem}

\begin{proof}
 Similar to the proof of Lemma~\ref{contraderived-weak-equivalences}.
 One needs to notice that, by Lemma~\ref{coacyclic-lemma}(a),
a monomorphism of complexes in $\sC(\sA)$ has coacyclic cokernel if
and only if it has coacyclic cone, and similarly, an epimorphism of
complexes in $\sC(\sA)$ has coacyclic kernel if and only if it has
coacyclic cone.
\end{proof}

 The following corollary presumes existence of (set-indexed) coproducts
in the coderived category $\sD^\co(\sA)$.
 Here we notice that, since the thick subcategory of coacyclic
complexes $\sK(\sA)_\ac^\co\subset\sK(\sA)$ is closed under coproducts
by Lemma~\ref{coacyclic-lemma}(b), the coproducts in the coderived
category $\sD^\co(\sA)$ are induced by those in the homotopy category
$\sK(\sA)$.
 In other words, the Verdier quotient functor $\sK(\sA)\rarrow
\sD^\co(\sA)$ preserves coproducts~\cite[Lemma~3.2.10]{Neem-book}.

\begin{cor}
 For any Grothendieck abelian category\/ $\sA$, the coderived category\/
$\sD^\co(\sA)$ is a well-generated triangulated category.
\end{cor}

\begin{proof}
 This result can be found in~\cite[Theorem~3.13]{Neem2},
\cite[Theorem~5.12]{Kra3}, or~\cite[Theorem~4.2]{Gil4}.
 It is also provable similarly to
Corollary~\ref{contraderived-well-generated}.
 One notices that the coderived category $\sD^\co(\sA)$ can be
equivalently defined as the homotopy category $\sC(\sA)[\cW^{-1}]$
of the coderived model structure on $\sC(\sA)$ and uses the fact
that the homotopy category of any stable combinatorial model category
is well-generated.
\end{proof}

\begin{rem}
 One can consider the possibility of extending the results of
Sections~\ref{grothendieck-secn}\+-\ref{coderived-secn} to
locally presentable abelian categories $\sA$ \emph{with enough
injective objects}.
 A specific problem arising in this connection is that it is not
clear how to prove a generalization of
Lemma~\ref{grothendieck-deconstructible} not depending on
the assumption that the directed colimits are exact in~$\sA$.
 Why is the directed colimit of the chain of subobjects
$(F_i)_{i<j}$ a subobject in~$M$\,?
 In fact, we do not even know whether the two classes of abelian
categories coincide.
\end{rem}

\begin{qst}
 Let $\sA$ be a locally presentable abelian category with enough
injective objects (or equivalently, a locally presentable abelian
category with an injective cogenerator).
 Does it follow that $\sA$ is Grothendieck?
\end{qst}

\Section{Exact Categories with an Object Size Function}
\label{hom-sets-secn}

 In this section, we consider exact categories in the sense of Quillen.
 An \emph{exact category} is an additive category endowed with a class
of \emph{short exact sequences} satisfying natural axioms.
 For a reference, see~\cite{Bueh}.

 A morphism $A\rarrow B$ in an exact category $\sE$ is said to be
an \emph{admissible monomorphism} if there exists a short exact
sequence $0\rarrow A\rarrow B\rarrow C\rarrow0$.
 The morphism $B\rarrow C$ is said to be an \emph{admissible
epimorphism} in this case.
 We will also say that $A$ is an \emph{admissible subobject} in~$B$.

 An additive category $\sE$ is said to be \emph{weakly
idempotent-complete} if for every pair of morphisms $i\:A\rarrow B$
and $p\:B\rarrow A$ in $\sE$ such that $pi=\id_A$ there exists
an object $C\in\sE$ and an isomorphism $B\simeq A\oplus C$ transforming
the morphism~$i$ into the direct summand inclusion $A\rarrow A\oplus C$
and the morphism~$p$ into the direct summand projection $A\oplus C
\rarrow A$.
 We will assume our exact category $\sE$ to be weakly
idempotent-complete (this assumption simplifies the theory of
exact categories considerably~\cite[Section~7]{Bueh}).

 Let $\sE$ be an exact category.
 We will consider a function~$\psi$ assigning to every object $E\in\sE$
a regular cardinal~$\psi(E)$.
 Given a regular cardinal~$\lambda$, denote by $\sE_\lambda\subset\sE$ the class
of all objects $E\in\sE$ such that $\psi(E)\le\lambda$.
 Instead of the function~$\psi$, one can consider the induced filtration
of the category $\sE$ by the full subcategories~$\sE_\lambda$.
 Obviously, one has $\sE_\lambda\subset\sE_\mu$ for all $\lambda\le\mu$
and $\sE=\bigcup_\lambda\sE_\lambda$.

 The following conditions are imposed on the function~$\psi$,
or equivalently, on the full subcategories $\sE_\lambda\subset\sE$:
\begin{enumerate}
\renewcommand{\theenumi}{\roman{enumi}}
\item $\psi(0)=0$, or in other words, $0\in\sE_0$; for any two
isomorphic objects $E'$ and $E''$ in $\sE$, one has
$\psi(E')=\psi(E'')$;
\item for every regular cardinal~$\lambda$, there is only a set of objects, up
to isomorphism, in the full subcategory $\sE_\lambda\subset\sE$;
\item for every regular cardinal~$\lambda$, the full subcategory
$\sE_\lambda$ is closed under extensions in $\sE$, that is, in other
words, if $0\rarrow A\rarrow B\rarrow C\rarrow0$ is a short exact
sequence in $\sE$ and $A$, $C\in\sE_\lambda$, then $B\in\sE_\lambda$;
\item for any regular cardinal~$\lambda$, an object $B\in\sE$,
and a collection of less than~$\lambda$ morphisms $(A_i\to B)_{i\in I}$
with the domains $A_i\in\sE_\lambda$, there exists an admissible
subobject $D\subset B$ with $D\in\sE_\lambda$ such that the morphism
$A_i\rarrow B$ factorizes through the admissible monomorphism
$D\rarrow B$ for every $i\in I$;
\item for any regular cardinal~$\lambda$ and any admissible epimorphism
$B\rarrow C$ with $C\in\sE_\lambda$ there exists an admissible
subobject $D\subset B$ such that $D\in\sE_\lambda$ and the composition
$D\rarrow B\rarrow C$ is an admissible epimorphism.
\end{enumerate}
 We will say that $\psi$~is an \emph{object size function} on $\sE$ if
the conditions~(i\+-v) are satisfied.
 It follows from the conditions~(i) and~(iii) that the full
subcategory $\sE_\lambda\subset\sE$ inherits an exact category
structure from the exact category~$\sE$.
 Specifically, a composable pair of morphisms in $\sE_\lambda$
is said to be a short exact sequence in $\sE_\lambda$
if it is a short exact sequence in~$\sE$.

\begin{prop} \label{locally-presentable-abelian-object-size-function}
 Any locally presentable abelian category\/ $\sE$ with its abelian
exact structure admits an object size function.
\end{prop}

\begin{proof}
 Let $\sE$ be a locally $\kappa$\+presentable abelian category (where
$\kappa$~is a regular cardinal).
 Define a function $\psi$ on the objects $E\in\sE$ by the rules
$\psi(0)=0$, \ $\psi(E)=\kappa$ if $E\ne0$ is $\kappa$\+generated,
and $\psi(E)=\lambda$ if $\lambda\ge\kappa$ is the minimal regular
cardinal for which $E$ is $\lambda$\+generated.

 Here an object $E\in\sE$ is said to be \emph{$\lambda$\+generated}
for a regular cardinal~$\lambda$ if the functor $\Hom_\sE(E,{-})\:
\sE\rarrow\Sets$ preserves the colimits of $\lambda$\+directed diagrams
of monomorphisms~\cite[Section~1.E]{AR}.
 An object $E\in\sE$ is $\lambda$\+generated for a given
$\lambda\ge\kappa$ if and only if $E$ is a quotient of
a $\lambda$\+presentable object~\cite[Proposition~1.69]{AR}.

 Then condition~(i) is satisfied by construction, and condition~(ii)
holds by~\cite[Corollary~1.69]{AR}.
 To check the remaining conditions, the following lemma will be
useful.

\begin{lem}
 Let $\lambda$~be a regular cardinal and\/ $\sE$ be a locally
$\lambda$\+presentable abelian category.
 Then an object $E\in\sE$ is $\lambda$\+generated if and only if
it cannot be presented as the sum of a $\lambda$\+directed set
of its proper subobjects (in the inclusion order).
\end{lem}

\begin{proof}
 The key observation is that, in a locally $\lambda$\+presentable
category, $\lambda$\+directed colimits commute with $\lambda$\+small
(in particular, finite) limits~\cite[Proposition~1.59]{AR}.
 For an abelian locally $\lambda$\+presentable category $\sE$, this
means that the functors of $\lambda$\+directed colimit are exact.
 We will use the facts that $\lambda$\+directed colimits in $\sE$
commute with the kernels and pullbacks.

 ``Only if'': Let $(E_i\subset E)_{i\in I}$ be a $\lambda$\+directed
family of subobjects in $E$, ordered by inclusion.
 Then the natural morphism $\varinjlim_{i\in I}E_i\rarrow E$ is
injective, since the morphisms $E_i\rarrow E$ are injective and
the functor $\varinjlim_{i\in I}$ commutes with the kernels.
 It follows that $\varinjlim_{i\in I}E_i=\sum_{i\in I}E_i\subset E$.

 Now assume that $E=\varinjlim_{i\in I}E_i$.
 Then, since the object $E$ is $\lambda$\+generated, the identity
morphism $E\rarrow\varinjlim_{i\in I}E_i=E$ factorizes through
the object $E_{i_0}$ for some index $i_0\in I$.
 Thus $E_{i_0}=E$.
 
 ``If'': Let $(f_{ij}\:F_i\rarrow F_j)_{i<j\in I}$ be
a $\lambda$\+directed diagram of monomorphisms in $\sE$ with the colimit
$F=\varinjlim_{i\in I}F_i$.
 Then the natural morphism $f_i\:F_i\rarrow F$ is a monomorphism for
every $i\in I$, since the morphisms $f_{ij}\:F_i\rarrow F_j$ are
monomorphisms for all $j\ge i$ and the functor $\varinjlim_{i\in I}$
preserves monomorphisms.

 Let $g\:E\rarrow F$ be a morphism.
 Denote by $E_i=E\times_FF_i$ the pullback of the pair of morphisms
$g\:E\rarrow F$ and $f_i\:F_i\rarrow F$.
 Then the objects $E_i$ form a family of subobjects in $\sE$, which is
$\lambda$\+directed by inclusion.
 Furthermore, $\varinjlim_{i\in I}E_i=\varinjlim_{i\in I} E\times_FF_i
=E\times_F\varinjlim_{i\in I}F_i=E\times_FF=E$, so $E=\sum_{i\in I}E_i$.
 By assumption, it follows that there exists $i_0\in\sE$ such that
$E_{i_0}=E$.
 This means that the morphism $g\:E\rarrow F$ factorizes through
the monomorphism $f_{i_0}\:F_{i_0}\rarrow F$.

 Thus the map of sets $\varinjlim_{i\in I}\Hom_\sE(E,F_i)\rarrow
\Hom_\sE(E,F)$ is surjective.
 Obviously, a factorization of~$g$ through~$f_i$ is unique if it
exists for a given $i\in I$, implying that the above map of sets
is injective.
\end{proof}

 Now we are ready to check conditions~(iii\+-v).

 (iii)~Let $(B_i\subset B)_{i\in I}$ be a family of subobjects,
$\lambda$\+directed in the inclusion order and such that
$\sum_{i\in I}B_i=B$.
 Put $A_i=A\cap B_i\subset A$, and denote by $C_i\subset C$
the image of the composition $B_i\rarrow B\rarrow C$.
 Then $\sum_{i\in I}A_i=A$, since the $\lambda$\+directed colimits
commute with the pullbacks, and obviously $\sum_{i\in I}C_i=C$.
 Hence there exist indices $i_1$ and $i_2\in I$ such that
$A_{i_1}=A$ and $C_{i_2}=C$.
 Then, for any index $i\in I$, \ $i\ge i_1$, $i_2$, one has $B_i=B$.

 (iv)~It is provable similarly to~\cite[Proposition~1.16]{AR}
that the class of all $\lambda$\+generated objects is closed under
$\lambda$\+small colimits.
 In particular, the coproduct $\coprod_{i\in I}A_i$ is
$\lambda$\+generated.
 The class of all $\lambda$\+generated objects is also closed under
quotients.
 So it suffices to let $D$ be the image of the morphism
$\coprod_{i\in I}A_i\rarrow B$.

 (v)~In a locally $\lambda$\+presentable category, every object
is the sum of its $\lambda$\+generated subobjects, which form
a $\lambda$\+directed poset with
the inclusion order (cf.~\cite[Theorem~1.70]{AR}).
 In the situation at hand, we thus have $B=\sum_{i\in I}B_i$, where
$I$~is $\lambda$\+directed and $B_i\subset B$ is
$\lambda$\+generated for every $i\in I$.
 Denote by $C_i\subset C$ the image of the composition $B_i\rarrow B
\rarrow C$.
 Then $C=\sum_{i\in I}C_i$.
 Since $C$ is $\lambda$\+generated, it follows that there exists
$i_0\in I$ such that $C_{i_0}=C$.
 Then it suffices to take $D=B_{i_0}\subset B$.
\end{proof}

 Let $\sE$ be an exact category with an object size function~$\psi$.
 For any cardinal~$\lambda$, there is only a set of morphisms in
the full subcategory $\sE_\lambda\subset\sE$ (up to isomorphism);
in particular, there is only a set of admissible epimorphisms
between objects from $\sE_\lambda$ in~$\sE$.
 Denote by $\phi(\lambda)=\phi_\sE(\lambda)$ the minimal regular
cardinal such that $\psi(K)\le\phi(\lambda)$ for every kernel $K$ of
an admissible epimorphism between objects from $\sE_\lambda$ in~$\sE$.
Note that $\phi(\lambda)\ge\lambda$ as one sees when considering
the admissible epimorphisms in $\sE_\lambda$ of the form $K\rarrow 0$.

 Given a regular cardinal~$\lambda$ and an ordinal~$\alpha$, we
construct by transfinite induction regular cardinals
\[
  \phi_\alpha(\lambda)=\phi_{\sE,\alpha}(\lambda)
	\le
  \phi_\sE^\alpha(\lambda)=\phi^\alpha(\lambda)
\]
as follows.
 Put $\phi_0(\lambda)=0$ and $\phi^0(\lambda)=\lambda$.
For a successor ordinal $\alpha=\beta+1$, we put
$\phi_\alpha(\lambda)=\phi^\beta(\lambda)$, while for a limit
ordinal~$\alpha$, we let $\phi_\alpha(\lambda)$ be
the minimal regular cardinal such that $\phi^\beta(\lambda)
\le\phi_\alpha(\lambda)$ for all $\beta<\alpha$.
Finally, in both the cases considered in the former sentence,
we put $\phi^\alpha(\lambda)=\phi(\phi_\alpha(\lambda))$.

 Let $\sE$ be an exact category, and let
$(f_{ij}\:E_i\rarrow E_j)_{0\le i<j<\alpha}$ be an inductive system
in $\sE$ indexed by an ordinal~$\alpha$.
 We will say that such an inductive system is a \emph{chain of
admissible monomorphisms} if the morphism~$f_{ij}$ is an admissible
monomorphism for every $0\le i<j<\alpha$.
 Notice that a chain of admissible monomorphisms does \emph{not}
need to be smooth.

 Let $\kappa$~be an infinite regular cardinal.
 We will say that an exact category $\sE$ has \emph{exact colimits of
$\kappa$\+directed chains of admissible monomorphisms} if
\begin{enumerate}
\renewcommand{\theenumi}{\roman{enumi}}
\setcounter{enumi}{5}
\item for every chain of admissible monomorphisms 
$(f_{ij}\:E_i\to E_j)_{0\le i<j\le\alpha}$ in $\sE$ such that
the cofinality of the ordinal~$\alpha$ is not smaller than~$\kappa$,
the $\kappa$\+directed colimit $\varinjlim_{i<\alpha}E_i$ exists
in~$\sE$;
\item for every inductive system of short exact sequences
$0\rarrow A_i\rarrow B_i\rarrow C_i\rarrow0$ in $\sE$, indexed
by an ordinal~$\alpha$ of the cofinality not smaller than~$\kappa$,
such that all the three inductive systems of objects
$(a_{ij}\:A_i\to A_j)_{0\le i<j<\alpha}$, \
$(b_{ij}\:B_i\to B_j)_{0\le i<j<\alpha}$, and
$(c_{ij}\:C_i\to C_j)_{0\le i<j<\alpha}$ are chains of admissible
monomorphisms, the short sequence $0\rarrow\varinjlim_{i<\alpha}A_i
\rarrow\varinjlim_{i<\alpha}B_i\rarrow\varinjlim_{i<\alpha}C_i
\rarrow0$ is exact in~$\sE$.
\end{enumerate}

 Finally, we will say that colimits of $\kappa$\+directed chains
of admissible monomorphisms in $\sE$ \emph{preserve an object size
function}~$\psi$ if 
\begin{enumerate}
\renewcommand{\theenumi}{\roman{enumi}}
\setcounter{enumi}{7}
\item for every regular cardinal $\lambda>\kappa$, every chain of
admissible monomorphisms $(f_{ij}\:E_i\to E_j)_{0\le i<j\le\alpha}$
in $\sE$ such that the cofinality $\cof(\alpha)$ of
the ordinal~$\alpha$ satisfies the inequalities $\kappa\le
\cof(\alpha)<\lambda$ and $E_i\in\sE_\lambda$ for every
$0\le i<\alpha$, one has $\varinjlim_{i<\alpha}E_i\in\sE_\lambda$.
\end{enumerate}

 Recall from~\cite[Section~10]{Bueh} that a complex $M^\bu\in\sC(\sE)$
in an exact category $\sE$ is called \emph{exact} (or \emph{acyclic}) if each
differential $d^n\colon M^n\rarrow M^{n+1}$ can be factored into
an admissible epimorphism $M^n\rarrow N^{n+1}$ followed by
an admissible monomorphism $N^{n+1}\rarrow M^{n+1}$ in such
a way that we have for each $n$ a short exact sequence
$0\rarrow N^n\rarrow M^n\rarrow N^{n+1}\rarrow 0$.

 The following theorem is the main technical result of this section.

\begin{thm} \label{small-exact-subcomplex}
 Let $\kappa<\lambda$ be two infinite regular cardinals.
 Let\/ $\sE$ be a weakly idempotent-complete exact category with
exact colimits of $\kappa$\+directed chains of admissible monomorphisms
and an object size function~$\psi$ preserved by such colimits.
 Put $\mu=\phi_{\kappa,\sE}(\lambda)$.
 Let $K^\bu$ be a complex in\/ $\sE_\lambda$, let $M^\bu$ be
an exact complex in\/ $\sE$, and let $g\:K^\bu\rarrow M^\bu$ be
a morphism of complexes in\/~$\sE$.
 Then there exist an exact complex $L^\bu$ in\/ $\sE_\mu$
and morphisms of complexes $k\:K^\bu\rarrow L^\bu$ and $m\:L^\bu
\rarrow M^\bu$ such that $g=mk$ and the component $m^n\:L^n
\rarrow M^n$ of the morphism of complexes~$m$ in the degree~$n$
is an admissible monomorphism in\/ $\sE$ for every $n\in\boZ$.
\end{thm}

\begin{proof}
 Let $N^n\in\sE$ be objects for which there exist short exact sequences
$0\rarrow N^n\rarrow M^n\rarrow N^{n+1}\rarrow0$ in $\sE$ such that
the differential $M^n\rarrow M^{n+1}$ is equal to the composition
$M^n\rarrow N^{n+1}\rarrow M^{n+1}$ for every $n\in\boZ$.
 Consider the diagram formed by the morphisms $N^n\rarrow M^n$ and
$M^n\rarrow N^{n+1}$ in the category~$\sE$.

 Proceeding by transfinite induction, we will construct a chain of
admissible subdiagrams $(N_i^n\to M_i^n\to N_i^{n+1})_{n\in\boZ}$
in the diagram $(N^n\to M^n\to N^{n+1})_{n\in\boZ}$ indexed by
the ordinals $0\le i\le\kappa$.
 This means that for every $0\le i\le\kappa$ we will produce admissible
subobjects $M_i^n\subset M^n$ and $N_i^n\subset N^n$ such that
$M_i^n\subset M_j^n$ and $N_i^n\subset N_j^n$ for all
$0\le i\le j\le\kappa$, and the morphisms $N^n\rarrow M^n$ and
$M^n\rarrow N^{n+1}$ take $N_i^n$ into $M_i^n$ and $M_i^n$
into~$N_i^{n+1}$.
 Furthermore, the morphisms $g^n\:K^n\rarrow M^n$ will factorize
through the admissible monomorphisms $M_1^n\rarrow M^n$.
 Finally, we will have $M_i^n$, $N_i^n\in\sE_{\phi^i(\lambda)}$ for all
$0\le i\le\kappa$, \,$n\in\boZ$, and the short sequences $0\rarrow
N^n_\kappa\rarrow M^n_\kappa\rarrow N^{n+1}_\kappa\rarrow0$ will be
exact in $\sE$ (equivalently, in~$\sE_{\phi^i(\lambda)}$) for all
$n\in\boZ$.
 Then we will put $L^\bu=M^\bu_\kappa$, that is $L^n=M^n_\kappa$
for every $n\in\boZ$.

 Put $M_0^n=0=N_0^n$ for all $n\in\boZ$.
 To construct the admissible subobjects $M_1^n\subset M^n$ and
$N_1^n\subset N^n$ for every $n\in\boZ$, we start with the compositions
$K^n\rarrow M^n\rarrow N^{n+1}$ of the morphisms $g^n\:K^n\rarrow M^n$
with the admissible epimorphisms $M^n\rarrow N^{n+1}$.
 Using property~(iv), we choose for every~$n$ an admissible subobject
$D_1^{n+1}\subset N^{n+1}$ such that $D_1^{n+1}\in\sE_\lambda$ and
the morphism $K^n\rarrow N^{n+1}$ factorizes through the admissible
monomorphism $D_1^{n+1}\rarrow N^{n+1}$.
 Let $F_1^n$ denote the pullback as depicted in the following
commutative diagram with short exact sequences in the rows and
(necessarily) admissible monomorphisms in the columns:
\[
\xymatrix{
0 \ar[r] &
N^n \ar[r] \ar@{=}[d] &
F_1^n \ar[r] \ar[d] &
D_1^{n+1} \ar[r] \ar[d] &
0
\\
0
\ar[r] &
N^n \ar[r] &
M^n \ar[r] &
N^{n+1} \ar[r] &
0
}
\]
%
 Using property~(v), we choose an admissible subobject
$G_1^n\subset F_1^n$ such that $G_1^n\in\sE_\lambda$ and
the composition $G_1^n\rarrow F_1^n\rarrow D_1^{n+1}$ is
an admissible epimorphism.

 Notice that the morphism $g^n\:K^n\rarrow M^n$ factorizes through
the admissible monomorphism $F_1^n\rarrow M^n$, since the composition
$K^n\rarrow M^n\rarrow N^{n+1}$ factorizes through the admissible
monomorphism $D_1^{n+1}\rarrow N^{n+1}$.
 Using property~(iv), we choose an admissible subobject
$M_1^n\subset F_1^n$ such that $M_1^n\in\sE_\lambda$ and the three
morphisms $G_1^n\rarrow F_1^n$, \ $K^n\rarrow F_1^n$, and
$D_1^n\rarrow N^n\rarrow F_1^n$ factorize through the admissible monomorphism
$M_1^n\rarrow F_1^n$.
 The composition of admissible monomorphisms $M_1^n\rarrow F_1^n
\rarrow M^n$ is an admissible monomorphism.
 The composition $G_1^n\rarrow M_1^n\rarrow F_1^n\rarrow D_1^{n+1}$ is
an admissible epimorphism by construction; since the exact category
$\sE$ is weakly idempotent-complete, it follows that the morphism
$M_1^n\rarrow D_1^{n+1}$ is an admissible
epimorphism~\cite[Proposition~7.6]{Bueh}.
 Let $N_1^n$ denote the kernel of the latter admissible epimorphism;
then $N_1^n\in\sE_{\phi^1(\lambda)}$. Moreover, since the composition
$N_1^n\rarrow M_1^n\rarrow M^n$ is an admissible monomorphism
by the construction and $\sE$ is weakly idempotent-complete,
an application of~\cite[Proposition~7.6]{Bueh} to $\sE^\sop$
shows that $N_1^n\rarrow N^n$ is an admissible monomorphism.
Since the composition $D_1^n\rarrow N^n\rarrow M^n\rarrow N^{n+1}$
vanishes, and hence so does
$D_1^n\rarrow M_1^n\rarrow F_1^n\rarrow D_1^{n+1}$, the map
$D_1^n\rarrow M_1^n$ factors through $N_1^n\rarrow M_1^n$ and
the factorization $D_1^n\rarrow N_1^n$ is an admissible monomorphism
using the same argument as above.

To summarize, we have the following commutative diagram with short exact
sequences in the rows and all the vertical maps being admissible
monomorphisms for all $n\in\boZ$,
\[
\xymatrix{
& D_1^n \ar[d]
&& K^n \ar@{.>}[dl]_-{\exists!} \ar[ddl]^(.2){g^n}|\hole
\\
0 \ar[r] &
N_1^n \ar[r] \ar[d] &
M_1^n \ar[r] \ar[d] &
D_1^{n+1} \ar[r] \ar[d] &
0
\\
0
\ar[r] &
N^n \ar[r] &
M^n \ar[r] &
N^{n+1} \ar[r] &
0,
}
\]
%
%
%
and we have constructed the desired admissible subdiagram
$(N_1^n\to M_1^n\to N_1^{n+1})_{n\in\boZ}$ in the diagram
$(N^n\to M^n\to N^{n+1})_{n\in\boZ}$.
The first step of our transfinite induction is complete.

 Let $2\le j<\kappa$ be an ordinal.
 Suppose that we have already constructed the admissible subdiagrams
$(N_i^n\to M_i^n\to N_i^{n+1})_{n\in\boZ}$ in the diagram
$(N^n\to M^n\to N^{n+1})_{n\in\boZ}$ for all the ordinals $i<j$.
 Using property~(iv), choose for every $n\in\boZ$ an admissible
subobject $D_j^n\subset N^n$ containing all the admissible subobjects
$N_i^n\subset N^n$ for $i<j$ and such that
$D_j^n\in\sE_{\phi_j(\lambda)}$. We again construct the following
pullback diagram with short exact sequences in the rows and admissible
monomorphisms in the columns,
\[
\xymatrix{
0 \ar[r] &
N^n \ar[r] \ar@{=}[d] &
F_j^n \ar[r] \ar[d] &
D_j^{n+1} \ar[r] \ar[d] &
0
\\
0
\ar[r] &
N^n \ar[r] &
M^n \ar[r] &
N^{n+1} \ar[r] &
0
}
\]

 Using property~(v), choose an admissible subobject $G_j^n\subset F_j^n$
such that $G_j^n\in\sE_{\phi_j(\lambda)}$ and the composition
$G_j^n\rarrow F_j^n\rarrow D_j^{n+1}$ is an admissible epimorphism.
%
 Using property~(iv), choose an admissible subobject $M_j^n\subset
F_j^n$ such that $M_j^n\in\sE_{\phi_j(\lambda)}$ and the two morphisms
$G_j^n\rarrow F_j^n$ and $D_j^n\rarrow N^n\rarrow F_j^n$ factorize
through the admissible monomorphism $M_j^n\rarrow F_j^n$.
 The composition of admissible monomorphisms $M_1^n\rarrow F_1^n
\rarrow M^n$ is an admissible monomorphism.
 The composition $G_j^n\rarrow M_j^n\rarrow F_j^n\rarrow D_j^{n+1}$ is
an admissible epimorphism by construction;
by~\cite[Proposition~7.6]{Bueh}, it follows that the morphism
$M_j^n\rarrow D_j^{n+1}$ is an admissible epimorphism.
 Let $N_j^n$ denote its kernel; then $N_j^n\in\sE_{\phi^j(\lambda)}$.

 Arguing as in the case of $j=1$ above, one shows that $N_j^n\rarrow
N_n$ is an admissible monomorphism and $D_j^n\subset N_j^n\subset N^n$.
 This finishes the construction of the admissible subdiagrams
$(N_j^n\to M_j^n\to N_j^{n+1})_{n\in\boZ}$ in the diagram
$(N^n\to M^n\to N^{n+1})_{n\in\boZ}$ for all ordinals $0\le j<\kappa$.
In particular, we have the following commutative diagram with
short exact sequences in the rows and admissible monomorphisms in
the columns for each $n\in\boZ$ and $i<j$
\[
\xymatrix{
& D_j^n \ar[d]
&& N_i^{n+1} \ar[d]
\\
0 \ar[r] &
N_j^n \ar[r] \ar[d] &
M_j^n \ar[r] \ar[d] &
D_j^{n+1} \ar[r] \ar[d] &
0
\\
0
\ar[r] &
N^n \ar[r] &
M^n \ar[r] &
N^{n+1} \ar[r] &
0
}
\]

 Using the assertion dual to~\cite[Proposition~7.6]{Bueh} again,
the inclusions of admissible subobjects $M_i^n\rarrow M_j^n$ and
$N_i^n\rarrow N_j^n$ are admissible monomorphisms for all $0\le i<j
<\kappa$.
 Thus $(M_i^n)_{0\le i<\kappa}$ and $(N_i^n)_{0\le i<\kappa}$ are
chains of admissible monomorphisms.
 By condition~(vi), it follows that the colimits $M_\kappa^n=
\varinjlim_{i<\kappa}M_i^n$ and $N_\kappa^n=\varinjlim_{i<\kappa}N_i^n$
exist.
 By condition~(vii), the natural morphisms $M_\kappa^n\rarrow M^n$
and $N_\kappa^n\rarrow N^n$ are admissible monomorphisms.
 By condition~(viii), we have $M_\kappa^n$, $N_\kappa^n\in
\sE_{\phi_\kappa(\lambda)}=\sE_\mu$.
 We have constructed the admissible subdiagram
$(N_\kappa^n\to M_\kappa^n\to N_\kappa^{n+1})_{n\in\boZ}$ in
the diagram $(N^n\to M^n\to N^{n+1})_{n\in\boZ}$.

 In order to show that the complex $L^\bu=M_\kappa^\bu$ is exact, it
remains to check exactness of the short sequences
$0\rarrow N_\kappa^n\rarrow M_\kappa^n\rarrow N_\kappa^{n+1}\rarrow0$.
 Here we observe that $(D_i^n)_{0\le i<\kappa}$ is a chain of
admissible subobjects in $N^n$ mutually cofinal with the chain
$(N_i^n)_{0\le i<\kappa}$.
 Hence $\varinjlim_{i<\kappa}D_i^n=N_\kappa^n$.
 Applying condition~(vii) to the chain of admissible monomorphisms
of short exact sequences $0\rarrow N_i^n\rarrow M_i^n\rarrow D_i^{n+1}
\rarrow0$, \,$0\le i\le\kappa$, we conclude that $0\rarrow
\varinjlim_{i<\kappa}N_i^n\rarrow\varinjlim_{i<\kappa}M_i^n\rarrow
\varinjlim_{i<\kappa}D_i^{n+1}\rarrow0$ is a short exact sequence
in~$\sE$.
\end{proof}

 We refer to~\cite[Section~10.4]{Bueh} for the definition of
the \emph{derived category} $\sD(\sE)$ of an exact category~$\sE$
as the Verdier quotient category of the homotopy category $\sK(\sE)$
by the triangulated subcategory of exact complexes.
 We also refer to~\cite[Set-Theoretic Remark~10.3.3]{Wei}
for a discussion of ``existence'' of localizations of categories
(including triangulated Verdier quotient categories, such
as $\sD(\sE)$).
 In a different terminology, the question is whether the derived
category $\sD(\sE)$ ``has Hom sets''.

\begin{cor} \label{derived-of-exact-category-exists}
 Let\/ $\sE$ be a weakly idempotent-complete exact category with
exact colimits of $\kappa$\+directed chains of admissible monomorphisms
and an object size function preserved by such colimits, for some
regular cardinal~$\kappa$.
 Then the derived category\/ $\sD(\sE)$ ``exists'' or ``has Hom sets'',
in the sense that morphisms between any given two objects in\/
$\sD(\sE)$ form a set rather than a proper class.
\end{cor}

\begin{proof}
 It is clear from Theorem~\ref{small-exact-subcomplex} that for every
complex $K^\bu\in\sK(\sE)$ there exists a set of morphisms $f_t\:K^\bu
\rarrow L^\bu_t$, \,$t\in\cT$ from the complex $K^\bu$ to exact
complexes $L^\bu_t$ in $\sK(\sE)$ such that every morphism from $K^\bu$
to an exact complex $M^\bu$ in $\sK(\sE)$ factorizes through one of
the morphisms~$f_t$.
 It follows that the multiplicative system $\cS$ of all morphisms with
exact cones in $\sK(\sE)$ is \emph{locally small} in the sense
of~\cite[Set-Theoretic Considerations~10.3.6]{Wei}; hence
the localization $\sD(\sE)=\sK(\sE)[\cS^{-1}]$ exists.
\end{proof}

\begin{thm} \label{derived-of-locally-presentable-abelian-exists}
 Let\/ $\sE$ be a locally presentable abelian category.
 Then the derived category\/ $\sD(\sE)$ ``exists'', in the sense that
morphisms between any given two objects in\/ $\sD(\sE)$ form a set
rather than a proper class.
\end{thm}

\begin{proof}
 Let $\kappa$~be a regular cardinal for which the category $\sE$ is
locally $\kappa$\+presentable.
 Then condition~(vi) is satisfied, since all colimits exist in~$\sE$.
 Furthermore, condition~(vii) holds, because the functors of
$\kappa$\+directed colimit are exact in~$\sE$
\,\cite[Proposition~1.59]{AR}.
 A construction of an object size function~$\psi$ on $\sE$ is
spelled out in
Proposition~\ref{locally-presentable-abelian-object-size-function}.
 So conditions~(i\+-v) are satisfied as well.
 Finally, condition~(viii) is provable similarly to the proof of
condition~(iv) in
Proposition~\ref{locally-presentable-abelian-object-size-function}.
 Thus Corollary~\ref{derived-of-exact-category-exists} is applicable,
and we are done.
\end{proof}

\begin{qst}
 Let\/ $\sE$ be a locally presentable abelian category. \par
\textup{(a)} Do set-indexed coproducts necessarily exist in
the derived category $\sD(\sE)$\,? \par
\textup{(b)} If the answer to~(a) is positive, is then
the triangulated category $\sD(\sE)$ well-generated?
\end{qst}

\end{document}